\documentclass{article}
\usepackage[english]{babel}
\usepackage[utf8]{inputenc}
\usepackage{johd}
\usepackage{mathptmx}
\usepackage{amsmath}
\usepackage{mathtools}
\usepackage{verbatim}
\usepackage{appendix}

\numberwithin{equation}{section}

\usepackage{amsmath,amsthm}
\usepackage{mathrsfs}
\usepackage{amssymb}
\usepackage{amsfonts}

\usepackage{cancel}
\usepackage{indentfirst}

\usepackage{graphicx}
\usepackage{subfigure}
\usepackage{wrapfig}
\usepackage{multicol}
\usepackage{tikz}

\usepackage[framemethod=TikZ]{mdframed}
\usepackage{color}
\usepackage{authblk}
\usepackage{textcomp}

\newtheorem{thm}{Theorem}[section]
\newtheorem{lemma}[thm]{Lemma}
\newtheorem{definition}[thm]{Definition}

\newtheorem{remark}[thm]{Remark}
\newtheorem{prop}[thm]{Proposition}

\newtheorem{corollary}[thm]{Corollary}

\author[2]{Zihao Gu\footnote{Email address: \url{zihao.gu@sjtu.edu.cn} (Z. Gu).}}
\author[1, 2]{Yiqing Lin\footnote{Email address: \url{yiqing.lin@sjtu.edu.cn} (Y. Lin).}}
\author[2]{Kun Xu\footnote{Corresponding author. Email address: \url{1949101x_k@sjtu.edu.cn} (K. Xu).}}
\affil[1]{\small{MOE-LSC, 
Shanghai Jiao Tong University, 200240 Shanghai, China.}}
\affil[2]{\small{School of Mathematical Sciences,
Shanghai Jiao Tong University, 200240 Shanghai, China.}}

\title{Exponential growth BSDE driven by a marked point process  }

\date{April 17, 2024}

\begin{document}

\maketitle

\begin{abstract}
In this study, we investigate the well-posedness of exponential growth backward stochastic differential equations (BSDEs) driven by a marked point process (MPP) under unbounded terminal conditions. Our analysis utilizes a fixed-point argument, the $\theta$-method, and an approximation procedure. Additionally, we establish the solvability of mean-reflected exponential growth BSDEs driven by the MPP using the $\theta$-method.
\end{abstract}


\textbf{MSC2020-classification:} 60G55; 60H10.

\tableofcontents

\section{Introduction}
In 1990, Pardoux and Peng \cite{pardoux1990adapted} first built the well-posedness of the nonlinear backward stochastic differential equations (BSDEs) driven by a Brownian motion with a Lipschitz continuous generator. Subsequently, numerous generalizations have been explored regarding the well-posedness of different types of BSDEs. Notably, quadratic BSDEs have garnered significant interest. Kobylanski \cite{kobylanski2000backward} constructed the existence and comparison theorem for quadratic BSDEs with a bounded terminal condition. Using a localization technique and some a priori estimates, Briand and Hu \cite{briand2006bsde} generalized the existence result to quadratic BSDEs with exponentially integrable terminals. Additionally, employing the $\theta$-method, Briand and Hu \cite{briand2008quadratic} proved a comparison theorem and stability results for quadratic BSDEs, assuming the generator $f$ is convex or concave in $z$, even when the terminal condition possesses an arbitrary order of exponential moment. Tevzadze \cite{tevzadze2008solvability} introduced a fixed-point approach to ascertain the solvability of quadratic BSDEs.

Apart from BSDEs within a Brownian framework, extending BSDEs to settings involving jumps broadens their applicability, such as in insurance modeling, as discussed by Liu and Ma \cite{Liu_2009}. Employing a fixed-point approach akin to that utilized in \cite{pardoux1990adapted}, Li and Tang \cite{Tang_1994}, and Barles, Buckdahn, and Pardoux \cite{barles1997backward} established the well-posedness for Lipschitz BSDEs with jumps (BSDEJs). Since then, numerous researchers have explored various types of BSDEJs. Notably, Becherer \cite{Becherer_2006} and Morlais \cite{morlais2010new} investigated BSDEJs driven by quadratic coefficients within an exponential utility maximization problem.
Moreover, Antonelli and Mancini \cite{antonelli2016solutions} constructed the well-posedness of BSDEJs with local Lipschitz drivers. \cite{Becherer_2006, morlais2010new, antonelli2016solutions} all adopted Kobylanski’s approach \cite{kobylanski2000backward} to the jump setting. In contrast, Cohen and Elliott \cite{cohen2015stochastic} and Kazi-Tani, Possama\"i, and Zhou \cite{kazi2015quadratic} employed the fixed-point approach of Tevzadze \cite{tevzadze2008solvability}. The well-posedness of BSDEJs with bounded terminals was established in these works. Furthermore, based on the stability of quadratic semimartingales, Barrieu and El Karoui \cite{Barrieu_2013} demonstrated the existence of a solution with an unbounded terminal under a quadratic structure condition in a continuous setup. To address BSDEJs, the quadratic structure condition was extended to a quadratic–exponential structure condition by Ngoupeyou \cite{ngoupeyou2010optimisation}, Jeanblanc, Matoussi \& Ngoupeyou \cite{jeanblanc2012robust}, and El Karoui, Matoussi \& Ngoupeyou \cite{karoui2016quadratic}. However, these results for unbounded terminals only provided existence without uniqueness. Recently, relying on the $\theta$-method, Kaaka\"i, Matoussi and Tamtalini \cite{kaakai2022utility} obtained the well-posedness of a special class of quadratic–exponential BSDEJs with unbounded terminal conditions. These were utilized in a robust utility maximization problem under several special structural conditions.

In this paper, going beyond the jump setting, we focus on a specific class of exponential growth BSDEs driven by a random measure associated with a marked point process (MPP), elucidating their characteristics.
\begin{equation}
\label{eq01}
Y_t= \xi + \int_t^T f\left(t, Y_s, U_s\right)dA_s-\int_t^T \int_E U_s(e) q(d sd e).
\end{equation}
Here, $q$ represents a compensated integer random measure corresponding to an MPP $(T_n,\zeta_n)_{n\ge 0}$, while $A$ denotes a continuous and increasing process. The well-posedness of BSDEs driven by general MPPs has been explored in various contexts: Confortola \& Fuhrman \cite{Confortola2013} for the weighted-$L^2$ solution, Becherer \cite{Becherer_2006} and Confortola \& Fuhrman \cite{Confortola_2014} for the $L^2$ case, Confortola, Fuhrman \& Jacod \cite{Confortola2016} for the $L^1$ case, and Confortola \cite{Confortola_2018} for the $L^p$ case. Additionally, Foresta \cite{foresta2021optimal} studied a more general BSDE featuring a Brownian motion diffusion term and a highly general MPP characterized by nonexplosive behavior and inaccessible jumps.

This paper extends the findings of \cite{briand2008quadratic} to establish the well-posedness of (\ref{eq01}) with an unbounded terminal condition when the generator $f$ exhibits convex or concave behavior in $u$. We adopt the exponential growth condition on $f$ as outlined in \cite{karoui2016quadratic}. To enhance clarity, we mainly focus on the well-posedness of (\ref{eq01}), which encompasses a pure jump diffusion term. However, the generalization of our result to BSDEs with a generator $f(t,y,z,u)$ under the quadratic–exponential assumption from literature such as \cite{karoui2016quadratic}, including a Brownian diffusion term, is straightforward. We establish uniqueness in the space $(Y,U)\in \mathcal E\times H_{\nu}^{2,p}$, for each $p\ge 1$, aligning to some extent with the results for quadratic BSDEs, as discussed in \cite{briand2008quadratic}. Compared to \cite{kaakai2022utility}, our result applies to a broader range of BSDEs. To address existence, drawing inspiration from \cite{foresta2021optimal}, we initially derive a uniform estimate for Lipschitz BSDEs under a linear bound condition, eliminating the $A_\gamma$ condition utilized in \cite{karoui2016quadratic}. Subsequently, we construct solutions for exponential growth BSDEs by approximating the generator with a family of Lipschitz generators. Unlike \cite{kazi2015quadratic}, we do not assume a bounded terminal condition or require the Fr\'echet differentiability condition on the generator. Additionally, our linear bound condition is less restrictive than the $A_{\gamma}$ condition used therein, enabling our results to apply under broader circumstances. For completeness, we provide a list of assumptions suitable for BSDEs with the generator $f(t,y,z,u)$.

We further extend our findings to include exponential growth BSDEs with mean reflection.
\begin{equation}
\label{eq001}
\left\{\begin{array}{l}
Y_t=\xi+\int_t^T f\left(s, Y_s, U_s\right) d A_s-\int_t^T \int_E U_s(e)q(dsde)+K_T-K_t, \quad 0 \leq t \leq T, \\
\mathbb{E}\left[\ell\left(t, Y_t\right)\right] \geq 0, \quad \forall t \in[0, T] \text { and } \int_0^T \mathbb{E}\left[\ell\left(t, Y_{t^-}\right)\right] d K_t=0.
\end{array}\right.
\end{equation}
Mean-reflected BSDEs were introduced by Briand, Elie \& Hu \cite{briand2018bsdes} in addressing superhedging problems under running risk management constraints. The basic structure of this problem is as follows:
\begin{equation}
\label{original}
\left\{\begin{array}{l}
 Y_t=\xi+\int_t^T f\left(s, Y_s, Z_s\right) d s-\int_t^T Z_s d W_s+K_T-K_t, \quad 0 \leq t \leq T, \\
\mathbb{E}\left[\ell\left(t, Y_t\right)\right] \geq 0, \quad 0 \leq t \leq T,
\end{array}\right.
\end{equation}
where $K$ is a deterministic process, and the running loss function $\left(\ell(t, \cdot)_{0 \leq t \leq T}\right.$ represents a set of nondecreasing bi-Lipschitz real-valued mappings. The Skorokhod condition we consider is of the following type.
$$
\int_0^T \mathbb{E}\left[\ell\left(t, Y_t\right)\right] d K_t=0,
$$
This condition ensures the existence and uniqueness of the deterministic flat solution. A generalization to a quadratic generator with a bounded terminal condition can be found in Hibon et al. \cite{Hibon2017quadratic}. Hu, Moreau and Wang \cite{hu2022general} established the well-posedness of general mean-reflected BSDEs with unbounded terminal conditions, where the driver $f$ takes the form $f\left(t, Y_t, P_{Y_{t}}, Z_t\right)$, with $P_{Y_{t}}$ representing the marginal probability distribution of the process $Y$ at time $t$.

Building upon the well-posedness of (\ref{eq01}), we employ the successive approximation procedure outlined in \cite{hu2022general} to establish the well-posedness of (\ref{eq001}) under exponential growth conditions. We assume that the terminal value possesses exponential moments of arbitrary order.

The remainder of the paper is structured as follows. In section \ref{section pre}, we introduce notation and present basic results for BSDEs driven by the MPP. Section \ref{section non reflected} investigates the well-posedness of (\ref{eq01}) with an exponential growth generator and an unbounded terminal. Finally, in section \ref{section mean reflected}, we explore the solvability of (\ref{eq001}) with mean reflection.

\section{Preliminaries}
\label{section pre}
In this section, we review some concepts related to MPPs and introduce certain assumptions. Throughout the paper, inequalities between random variables are assumed to hold $\mathbb P-a.s$. Further details regarding MPPs can be found in \cite{foresta2021optimal, Bremaud1981, last1995marked, cohen2012existence}.

Let $E$ denote the marked space, which is a Borel space, and $\mathcal{B}(E)$ represents its Borel $\sigma$-algebra. Consider a complete probability space $(\Omega, \mathscr{F}, \mathbb{P})$. Given a sequence of random variables $(T_n,\zeta_n)$ taking values in $[0,\infty]\times E$, we set $T_0=0$ and $\mathbb P\text{-a.s.}$.
\begin{itemize}
\item $T_n\le T_{n+1},\ \forall n\ge 0;$
\item $T_n<\infty$ implies $T_n<T_{n+1} \ \forall n\ge 0.$
\end{itemize}
The sequence $(T_n,\zeta_n)_{n\ge 0}$ is referred to as an MPP. Additionally, we assume that the MPP is nonexplosive, meaning $T_n\to\infty,\ \mathbb P\text{-a.s.}$.

We define a random discrete measure $p$ on $((0,+\infty) \times E, \mathscr{B}((0,+\infty) \times E))$ associated with the MPP.
\begin{equation}
\label{eq p}
  p(\omega, D)=\sum_{n \geq 1} \mathbf{1}_{\left(T_n(\omega), \zeta_n(\omega)\right) \in D}.
\end{equation}
For each $\tilde C \in \mathcal B(E)$, we define the counting process $N_t(\tilde C)=p((0, t] \times \tilde C)$ and denote $N_t=N_t(E)$. Both of these processes are right-continuous and increasing, starting from zero. For $t \geq 0$, we define
$$
\mathscr{F}_t^0=\sigma\left(N_s(\tilde C): s \in[0, t], \tilde C \in \mathcal{B}(E)\right)
$$
and $\mathscr{F}_t=\sigma\left(\mathscr{F}_t^0, \mathscr{N}\right)$, where $\mathscr{N}$ is the family of $\mathbb{P}$-null sets of $\mathscr{F}$. We denote by $\mathbb{F}=\left(\mathscr{F}_t\right)_{t \geq 0}$ the completed filtration generated by the MPP, which is right-continuous and satisfies the usual hypotheses\footnote{Given a standard Brownian motion $W\in \mathbb R^d$, independent of the MPP, to handle BSDEs with a Brownian diffusion term as discussed in \cite{foresta2021optimal}, it is natural to extend the filtration to $\mathbb G=(\mathcal G_t)$, the completed filtration generated by the MPP and $W$, which also satisfies the usual conditions.}.

Each MPP satisfying above conditions possesses a unique compensator $v$, which is a predictable random measure satisfying
$$
\mathbb{E}\left[\int_0^{+\infty} \int_E C_t(e) p(d t d e)\right]=\mathbb{E}\left[\int_0^{+\infty} \int_E C_t(e) v(d t d e)\right]
$$
for all $C$ which is nonnegative and $\mathscr{P}^{\mathbb{F}} \otimes \mathcal{B}(E)$-measurable, where $\mathscr{P}^{\mathbb{F}}$ denotes the $\sigma$-algebra generated by $\mathscr{F}$-predictable processes.

We are working within the stochastic basis defined above $(\Omega, \mathcal{F}, \mathbb{F}, \mathbb{P})$ with a finite time horizon $T<+\infty$. On this basis, let $p$ be the random measure defined earlier such that $\nu(\omega, d t, d e)=\phi_t(de)dA_t$, where $\phi_t(de)$ is a probability measure on $(E, \mathcal{B}(E))$, $\phi_t(\cdot)$ is predictable, and $A$ is the dual predictable projection of $N$.
In other words, the process $A$ is the unique right-continuous increasing process with $A_0=0$ such that for any non-negative predictable process $D$, it holds that
$$
\mathbb E\left[\int_{0}^\infty D_t dN_t\right]=\mathbb E\left[\int_{0}^\infty D_t dA_t\right].
$$

Fix a terminal time $T>0$, and we can define the integral
$$
\int_0^T \int_E C_t(e) q(d t d e)=\int_0^T \int_E C_t(e) p(d t d e)-\int_0^T \int_E C_t(e) \phi_t(d e) d A_t,
$$
under the condition
$$
\mathbb{E}\left[\int_0^T \int_E\left|C_t(e)\right| \phi_t(d e) d A_t\right]<\infty.
$$
Indeed, the process $\int_0^{\cdot} \int_E C_t(e) q(d t d e)$ is a martingale. Note that $\int_a^b$ denotes an integral on $(a, b]$ if $b<\infty$ or on $(a, b)$ if $b=\infty$.

Next, we introduce the following spaces.

\begin{itemize}
\item $\mathbb{L}^0$ denotes the space of all real-valued, $\mathcal{F}_T$-measurable random variables.
\item $\mathbb{L}^p \triangleq\left\{\xi \in \mathbb{L}^0 :\|\xi\|_p \triangleq\left\{E\left[|\xi|^p\right]\right\}^{\frac{1}{p}}<\infty\right\}$ for all $p \in[1, \infty)$.
\item $\mathbb{L}^{\infty}\triangleq\left\{\xi \in \mathbb{L}^0:\|\xi\|_{\infty} \triangleq {esssup}_{\omega \in \Omega}|\xi(\omega)|<\infty\right\}$.

\item $\mathscr{A}_D$ represents the space of all nondecreasing deterministic c\`adl\`ag processes $K$ starting from the origin, i.e., $K_0=0$.
\item ${S}^0$ denotes the set of real-valued, adapted, and c\`adl\`ag processes $\left\{Y_t\right\}_{t \in[0, T]}$.

\item For any $\left\{\ell_t\right\}_{t \in[0, T]} \in S^0$, define $\ell_*^{ \pm} \triangleq \sup _{t \in[0, T]}\left(\ell_t\right)^{ \pm}$. Then
$$
\ell_* \triangleq \sup _{t \in[0, T]}\left|\ell_t\right|=\sup _{t \in[0, T]}\left(\left(\ell_t\right)^{-} \vee\left(\ell_t\right)^{+}\right)=\sup _{t \in[0, T]}\left(\ell_t\right)^{-} \vee \sup _{t \in[0, T]}\left(\ell_t\right)^{+}=\ell_*^{-} \vee \ell_*^{+}.
$$

\item For any real $p \geq 1, {S}^p$ denotes the set of real-valued, adapted, and c\`adl\`ag processes $\left\{Y_t\right\}_{t \in[0, T]}$ such that
$$
\|Y\|_{{S}^p}:=\mathbb{E}\left[\sup _{0 \leq t \leq T}\left|Y_t\right|^p\right]^{1 / p}<+\infty.
$$
Then, $\left({S}^p,\|\cdot\|_{ {S}^p}\right)$ is a Banach space.
\item ${S}^{\infty}$ is the space of $\mathbb{R}$-valued c\`adl\`ag and $\mathbb F$-progressively measurable processes $Y$ such that
$$
\|Y\|_{{S}^\infty}:=\left\|\sup _{0 \leq t \leq T}|Y_t|\right\|_{\infty}<+\infty.
$$
\item For any $p \geq 1$, we denote by $\mathcal{E}^p$ the collection of all stochastic processes $Y$ such that $e^{|Y|} \in$ $S^p\left(0, T \right)$. We write $Y \in \mathcal{E}$ if $Y \in \mathcal{E}^p$ for any $p \geq 1$.
\item $L^{2}(A)$ is the space of all $\mathbb{F}$-progressively measurable processes $Y$ such that
$$
\|Y\|_{L^{2}(A)}^2=\mathbb{E}\left[\int_0^T \left|Y_s\right|^2 d A_s\right]<\infty.
$$
\item $\mathbb{H}^p$ is the space of $\mathbb{R}^d$-valued and $\mathbb{F}$-progressively measurable processes $Z$ such that
$$
\|Z\|_{\mathbb{H}^p}^p:=\mathbb{E}\left[\left(\int_0^T\left|Z_t\right|^2 d t\right)^{\frac{p}{2}}\right]<+\infty.
$$
\item $L^0\left(\mathscr{B}(E)\right)$ denotes the space of $\mathcal B(E)$-measurable functions. For $u \in L^0\left(\mathscr{B}(E)\right)$, define
$$
L^2(E,\mathcal{B}(E),\phi_t(\omega,dy)):=\left\{u \in L^0\left(\mathscr{B}(E)\right):\ \left\|u\right\|_t:=\left(\int_E\left|u(e)\right|^2  \phi_t(d e)\right)^{1 / 2}<\infty\right\}.
$$
\item $H^{2,p}_{\nu}$ is the space of all predictable processes $U$ such that
$$
\|U\|_{H_{\nu}^{2,p}}:=\left(\mathbb{E}\left[\int_{[0, T]} \int_E\left|U_s(e)\right|^2 \phi_s(de) d A_s\right ]^{\frac{p}{2}}\right)^{\frac{1}{p}}<\infty.
$$
\item $H^{2,loc}_{\nu}$ is the space of all predictable processes $U$ such that
$$
\int_{[0, T]} \int_E\left|U_s(e)\right|^2 \phi_s(de) d A_s<\infty,\ \mathbb P\text{-a.s.}
$$
Following \cite{Confortola2013}, we define $U, U^{\prime} \in H^{2,p}_{\nu}$ (respectively, $U, U^{\prime} \in H^{2,loc}_{\nu}$) as equivalent if they coincide almost everywhere with respect to the measure $\phi_t(\omega, d y) d A_t(\omega) \mathbb{P}(d \omega)$, and this holds if and only if $\left\|U-U^{\prime}\right\|_{H_{\nu}^{2,p}}=0$ (equivalently, 
$\int_E\left|U_s(e)-U^{\prime}_s(e)\right|^2 \phi_s(de) d A_s=0,\,\ \mathbb P\text{-a.s.}$). With a slight abuse of notation, we continue to denote $H^{2,p}_{\nu}$ (respectively, $H^{2,loc}_{\nu}$) as the corresponding set of equivalence classes. Additionally, $\left(H_{\nu}^{2,p},\|\cdot\|_{H_\nu^{2,p}}\right)$ forms a Banach space.
\item $\mathbb{J}^{\infty}$ is the space of functions such that
$$
\|\psi\|_{\mathbb{J}^\infty}:=\left\| \left\|\psi\right\|_{{L}^{\infty}(\nu(\omega))}\right\|_{\infty}<\infty.
$$
\item $\mathcal S_{0,T}$ denotes the collection of $\mathbb F$-stopping times $\tau$ such that $0\le\tau\le T,\ \mathbb P\text{-a.s.}$. For any $\tau\in\mathcal S_{0,T}$, $\mathcal S_{\tau,T}$ denotes the collection of $\mathbb F$-stopping times $\tilde\tau$ such that $\tau\le\tilde\tau\le T,\ \mathbb P\text{-a.s.}$
\end{itemize}

For any $u \in L^0\left(\mathcal{B}(E)\right)$ and $\lambda >0$, we define a positive predictable process as
$$
j_\lambda (\cdot, u)=\int_E (\mathrm{e}^{\lambda u(e)}-1-\lambda u(e))\phi_{\cdot}(de).
$$

Without loss of generality, we focus on BSDEs of the form:
\begin{equation}
\label{BSDE}
Y_t= \xi + \int_t^T f\left(t, Y_s, U_s\right)dA_s-\int_t^T \int_E U_s(e) q(d sd e).
\end{equation}

The BSDE (\ref{BSDE}) with coefficient $(\xi,f)$ is referred to as BSDE $(\xi,f)$.
Subsequently, we outline the general assumptions consistently employed in this paper.
\hspace*{\fill}\\
\hspace*{\fill}\\
\noindent(\textbf{H1}) The process A is continuous, with $\|A_T\|_\infty<\infty$.
\hspace*{\fill}\\

The aforementioned assumption relies on the dual predictable projection $\nu$ of the counting process $N$ associated with $p$. It is worth noting that for $A_t$, absolute continuity with respect to the Lebesgue measure is not required. In simpler terms, the compensator $v$ does not permit a decomposition $v_t(\omega,dt,dx)=\xi(\omega,t,x)\lambda(dx)dt$ as illustrated in \cite{karoui2016quadratic}. Moreover, thanks to \cite[Corollary 5.28]{he2019semimartingale}, the process $p$ is totally inaccessible, i.e., $\{T_n\}_{n\ge 1}$ are totally inaccessible stopping times.
\hspace*{\fill}\\
\hspace*{\fill}\\
\hspace*{\fill}\\
\noindent(\textbf{H2}) For every $\omega \in \Omega,\ t \in[0, T],\ r \in \mathbb{R}$, the mapping
$
f(\omega, t, r, \cdot):L^2(E,\mathcal{B}(E),\phi_t(\omega,dy))\rightarrow \mathbb{R}
$
satisfies:
For every $U \in {H_\nu^{2,2}}$,
$$
(\omega, t, r) \mapsto f\left(\omega, t, r, U_t(\omega, \cdot)\right)
$$
is Prog $\otimes \mathscr{B}(\mathbb{R})$-measurable.

\hspace*{\fill}\\
\hspace*{\fill}\\
\noindent(\textbf{H3})

\textbf{(a) (Continuity condition)}
For every $\omega \in \Omega, t \in[0, T], y \in \mathbb{R}$, $u\in L^2(E,\mathcal{B}(E),\phi_t(\omega,dy))$, $(y, u) \longrightarrow f(t, y, u)$ is continuous.
\hspace*{\fill}\\
\hspace*{\fill}\\

\textbf{(b) (Lipschitz condition in $y$)}
There exists $\beta\geq 0$ such that for every $\omega \in \Omega,\ t \in[0, T],\ y, y^{\prime} \in \mathbb{R}$,\ $u\in L^2(E,\mathcal{B}(E),\phi_t(\omega,dy))$, the following holds
$$
\begin{aligned}
& \left|f(\omega, t, y, u(\cdot))-f\left(\omega, t, y^{\prime}, u(\cdot)\right)\right| \leq \beta\left|y-y^{\prime}\right|.
\end{aligned}
$$
\hspace*{\fill}\\

\textbf{(c) (Exponential growth condition)}
For all $t \in[0, T], \ (y,u) \in \mathbb{R} \times L^2(E,\mathcal{B}(E),\phi_t(\omega,dy)):\ \mathbb{P}$-a.s, there exists $\lambda>0$ such that
$$
\underline{q}(t, y, u)=-\frac{1}{\lambda} j_{\lambda}(t,- u)-\alpha_t-\beta|y| \leq f(t, y, u) \leq \frac{1}{\lambda} j_{\lambda}(t, u)+\alpha_t+\beta|y|=\bar{q}(t, y, u),
$$
where $\{\alpha_t\}_{0 \leq t \leq T}$ is a progressively measurable non-negative stochastic process.

\hspace*{\fill}\\
\hspace*{\fill}\\

\textbf{(d) (Integrability condition)}
We assume that
$$
\forall\ p>0, \quad \mathbb{E}\left[\exp \left(p\left(\left|\xi\right|+\int_0^T \alpha_s d A_s\right)\right)\right]<+\infty.
$$
\hspace*{\fill}\\

\textbf{(e) (Convexity/Concavity condition)}
For each $(t, y) \in [0, T] \times \mathbb{R}$, $u\in L^2(E,\mathcal{B}(E),\phi_t(\omega,dy)),\ u \rightarrow f(t, y, u)$ is either convex or concave.

\section{Exponential growth BSDE}
\label{section non reflected}
We aim to obtain a unique solution $(Y, U)$ in the space $\mathcal E\times H_{\nu}^{2,p}$, for all $p\ge 1$, to the BSDE (\ref{BSDE}).
To establish the well-posedness of the exponential growth BSDE (\ref{BSDE}), following the approach in \cite{karoui2016quadratic}, we begin by restating the well-posedness condition with a Lipschitz generator as introduced in \cite{foresta2021optimal}. Subsequently, we apply an approximation procedure.

\subsection{Lipschitz BSDE}

 The following well-posed results for Lipschitz BSDEs driven by an MPP are adapted from \cite{foresta2021optimal}. We present the result under our assumptions for subsequent approximation.

Specifically, the additional assumptions are as follows:
\hspace*{\fill}\\
\hspace*{\fill}\\
\noindent(\textbf{H3'})

(\textbf{a})
 The final condition $\xi: \Omega \rightarrow \mathbb{R}$ is $\mathscr{F}_T$-measurable, and
$$\mathbb E[\xi^2]<\infty.$$
\hspace*{\fill}\\
\hspace*{\fill}\\

\textbf{(b)}
There exist $L_f \geq 0,\ L_U \geq 0$ such that for every $\omega \in \Omega,\ t \in[0, T],\ y, y^{\prime} \in \mathbb{R}$,\ $u, u^{\prime} \in L^2(E,\mathcal{B}(E),\phi_t(\omega,dy))$, the following holds
$$
\left|f(\omega, t, y, u(\cdot))-f\left(\omega, t, y^{\prime}, u^{\prime}(\cdot)\right)\right| \leq L_f\left|y-y^{\prime}\right|
+L_U\left(\int_E\left|u(e)-u^{\prime}(e)\right|^2 \phi_t(\omega, d e)\right)^{1 / 2}.
$$
\hspace*{\fill}\\
\hspace*{\fill}\\

\textbf{(c)}
We assume
$$
\mathbb{E}\left[\int_0^T |f(s, 0,0)|^2 d A_s\right]<\infty.
$$

The well-posedness result for the Lipschitz BSDE (\ref{BSDE}) is as follows, based on \cite[Theorem 4.1]{foresta2021optimal} and the a priori estimate \cite[Lemma 3.2]{foresta2021optimal}.

\begin{thm}
\label{thm_lip}
If assumptions (H1), (H2), and (H3') hold, then there exists a unique solution $(Y,U)\in L^2(A)\times H_{\nu}^{2,2}$ to (\ref{BSDE}). In addition, $Y\in S^2$.
\end{thm}

Subsequently, we employ the approximation method outlined in \cite{karoui2016quadratic} to construct the well-posed exponential growth BSDE (\ref{BSDE}).

\subsection{A priori estimates for exponential growth BSDEs }
We begin to investigate the solutions to BSDE(\ref{BSDE}), which are defined as follows.
\begin{definition}[Solution to the BSDE]
\label{def solution}
Under assumptions (H1)-(H3), a solution to BSDE (\ref{BSDE}) is a couple process $(Y,U)$ on $[0,T]$, in which $Y$ is a c\`adl\`ag process,  and $U$ is an $\mathbb F$-predictable random field. Moreover, for each $p\ge 1$, processes $\int_0^{\cdot}\int_E(e^{p\lambda U_t(e)}-1)q(dsde)$ and 
$\int_0^{\cdot}\int_E(e^{p\lambda U_t(e)}-1)q(dsde)$ are local martingales on $[0,T]$.

\end{definition}

\subsubsection{A priori estimate on $Y$}
The following a priori estimate for exponential growth BSDEs driven by an MPP is pivotal. The proof of Lemma \ref{submartingle property} is presented in Appendix \ref{appexdix A}. The concept was inspired by \cite[Proposition 3.3]{karoui2016quadratic}.

\begin{lemma}
\label{submartingle property}
 Under assumptions (H1) and (H2), let $(y, u)$ be a solution to BSDE (\ref{BSDE}). Assume that for some $p\ge 1$,
\begin{equation}
\label{integrability_condition}
\mathbb{E}\left[\exp \left\{pe^{\beta A_T}\lambda|\xi|+ p\lambda \int_0^Te^{\beta A_t} \alpha_t d t\right\}\right]<\infty.
\end{equation}
Then, the following a priori estimates hold.

(i) If $-\alpha_t-\beta|y_t|-\frac{1}{\lambda}j_\lambda(t,-u_t)\le f(t,y, u) \leq \alpha_t+\beta|y_t|+\frac{1}{\lambda}j_\lambda(t,u_t)$, then for each $t \in[0, T]$,
\begin{equation}
\label{eY bound}
\exp \left\{p \lambda\left|y_t\right|\right\} \leq \mathbb{E}_t\left[\exp \left\{p \lambda e^{\beta A_T}|\xi|+p\lambda \int_t^T e^{\beta A_s}\alpha_s d A_s\right\}\right].
\end{equation}
Therefore, considering Doob’s inequality, for every $p>0$, $E[e^{py_*}]$ remains uniformly bounded. This bound is denoted by $\Xi(p,\alpha,\beta)$.

(ii) If $f(t,y,u) \leq \alpha_t+\beta |y_t|+\frac{1}{\lambda}j_{\lambda}(t,u_t)$, then for each $t \in[0, T]$,
$$
\exp \left\{p \lambda\left(y_t\right)^{+}\right\} \leq \mathbb{E}_t\left[\exp \left\{p\lambda e^{\beta A_T} \xi^{+}+p\lambda \int_t^T e^{\beta A_s}\alpha_s d A_s\right\}\right].
$$
\end{lemma}

The following corollary represents a degenerate case of Lemma \ref{submartingle property}.
\begin{corollary}
\label{a priori esti 2}
 Under assumptions (H1) and (H2), let $(y, u)$ be a solution to a simplified version of BSDE (\ref{BSDE}) with $f(t,y,u)\equiv f(t,u)$. Suppose there exists a constant $p \geq 1$ such that
\begin{equation}
\mathbb{E}\left[\exp \left\{p\lambda|\xi|+ p\lambda \int_0^T \alpha_t d t\right\}\right]<\infty.
\end{equation}
Then, we have

(i) If $-\alpha_t-\frac{1}{\lambda}j_\lambda(t,-u_t)\le f(t, u) \leq \alpha_t+\frac{1}{\lambda}j_\lambda(t,u_t)$, then for each $t \in[0, T]$,
$$
\exp \left\{p \lambda\left|y_t\right|\right\} \leq \mathbb{E}_t\left[\exp \left\{p \lambda |\xi|+p\lambda \int_t^T \alpha_s d A_s\right\}\right].
$$

(ii) If $f(t,u) \leq \alpha_t+\frac{1}{\lambda}j_{\lambda}(t,u_t)$, then for each $t \in[0, T]$,
$$
\exp \left\{p \lambda\left(y_t\right)^{+}\right\} \leq \mathbb{E}_t\left[\exp \left\{p\lambda \xi^{+}+p\lambda \int_t^T \alpha_s d A_s\right\}\right].
$$
\end{corollary}

\subsubsection{A priori estimate on $U$}
We then establish a priori estimate of $U$. The proof is also included in Appendix \ref{appexdix A} for completeness.

\begin{prop}
\label{priori estimate on U and K}
Consider $(\xi,f)$ as parameters satisfying assumptions (H1)–(H3). If $(Y,U)$ is a solution of the BSDE (\ref{BSDE}) with $Y\in\mathcal E$, then
\begin{equation}
\label{U bound}
\mathbb E\left[\left(\int_0^T\int_E|U_t(e)|^2\phi_t(de)dA_t\right)^{p/2}\right]\le C_p\mathbb E\left[e^{8p\lambda(1+\beta\|A_T\|_\infty) Y_*}\right]\leq C_p\Xi(8p\lambda(1+\beta\|A_T\|_\infty),\alpha,\beta)<\infty.
\end{equation}
In fact, the following stronger estimate for $U$ also holds: for each $p,q\ge 1$,
\begin{equation}
\label{e|u|q01}
\mathbb E\left[\left(\int_0^T\int_E\left(e^{q \lambda |U_t(e)|}-1\right)^2 \phi_t\left(de\right) d A_t\right)^p\right]\le  C_p\mathbb E\left[e^{8pq\lambda (1+\beta\|A_T\|_\infty)Y_*}\right]\le C_p\Xi(8pq\lambda(1+\beta\|A_T\|_\infty),\alpha,\beta)<\infty.
\end{equation}
\end{prop}

\subsection{Comparison theorem for exponential growth BSDEs with a convex/concave generator }

We now prove a comparison theorem for (\ref{BSDE}) using the $\theta$-method, following the approach in \cite{briand2008quadratic}. The additional convexity/concavity assumption (H3)(e) plays a crucial role in our proof. Furthermore, the exponential term in the growth condition (H3)(c) parallels the role of the quadratic term in the conditions of quadratic BSDEs in \cite{briand2008quadratic}. Consequently, our results align formally with those of quadratic BSDEs. The uniqueness of BSDE(\ref{BSDE}) follows directly as a corollary of the comparison theorem.
\begin{thm}
\label{Comparison Theorem for BSDE}
Let $\left(Y, U\right)$ (resp. $(Y',U')$) be a solution of the BSDE (\ref{BSDE}) in $\mathcal{E}\times {H}_{\nu}^{2,p}$ associated with $\left(f, \xi\right)$ (resp. $\left(f', \xi'\right)$), for some $p\ge 1$. If $f\leq f', \ \xi\le \xi'$, and assumptions (H1)–(H3) hold, then
$$
\forall t \in[0, T], \quad Y_t\leq Y_t'\quad \mathbb P\text{-a.s.}
$$
\end{thm}

\begin{proof}
The approach is inspired by \cite{briand2008quadratic}. Without loss of generality, let us assume $f$ is convex with respect to $u$; otherwise, refer to Remark \ref{concave01}.
Let $\theta \in(0,1)$ and set $\tilde Y_t=Y_t-\theta {Y_t}^{\prime}$ and $\tilde U_t=U_t-\theta {U_t}^{\prime}$. Consider a progressively measurable process $\{a(t)\}_{0 \leq t \leq T}$, with appropriate integrability, to be chosen later. We define, for all $t \in[0, T], \ \tilde A_t=\int_0^t a(s) d A_s$. Since $\tilde Y$ is c\`adl\`ag and $A$ is continuous, by It\^o’s formula,
$$
e^{\tilde A_t} \tilde Y_t=e^{\tilde A_T} \tilde Y_T+\int_t^T e^{\tilde A_s} F_sd A_s-\int_t^T\int_E e^{\tilde A_s} \tilde U_s(e) q(dsde), \quad 0 \leq t \leq T,
$$
where,
\begin{equation}
\label{def-F}
\begin{aligned}
F_t=&\left(f\left(t, Y_t, U_t\right)-\theta f'\left(t, {Y_t}^{\prime}, {U_t}^{\prime}\right)\right)-a(t) \tilde Y_{t}\\
= & \left(f\left(t, Y_t, U_t\right)-f\left(t, {Y_t}^{\prime}, U_t\right)\right)+\left(f\left(t, {Y_t}^{\prime}, U_t\right)-\theta f\left(t, {Y_t}^{\prime}, {U_t}^{\prime}\right)\right)-a(t)\tilde Y_{t}+\theta\delta f(t)
\end{aligned}
\end{equation}
and $
\delta f(t)=\left(f-f^{\prime}\right)\left(t, {Y_t}^{\prime}, {U_t}^{\prime}\right).
$

Because $f$ is convex with respect to $u$, the second term on the right-hand side of (\ref{def-F}) can be bounded by considering the growth condition (H3)(c). Specifically,
$$
\begin{aligned}
f\left(t, {Y_t}^{\prime}, U_t\right) & =f\left(t, {Y_t}^{\prime}, \theta {U_t}^{\prime}+(1-\theta) \frac{{U_t}-\theta {U_t}^{\prime}}{1-\theta}\right) \\
& \leq \theta f\left(t, {Y_t}^{\prime}, {U_t}^{\prime}\right)+(1-\theta) f\left(t, {Y_t}^{\prime}, \frac{{U_t}-\theta {U_t}^{\prime}}{1-\theta}\right)
\end{aligned}
$$
and from the growth condition (H3)(c) of the generator $f$,
\begin{equation}
\label{f-bound}
f\left(t, {Y_t}^{\prime}, {U_t}\right) \leq \theta f\left(t, {Y_t}^{\prime}, {U_t}^{\prime}\right)+(1-\theta)\left(\alpha_t+\beta\left|{Y_t}^{\prime}\right|\right)+\frac{1-\theta}{\lambda}j_{\lambda}(\frac{\tilde U_t}{1-\theta}).
\end{equation}

Note that for the first term in (\ref{def-F}),
$$
\begin{aligned}
f\left(t, Y_t, {U_t}\right)-f\left(t, {Y_t}^{\prime}, {U_t}\right)= & f\left(t, Y_t, {U_t}\right)-f\left(t, \theta {Y_t}^{\prime}, {U_t}\right) \\
& +f\left(t, \theta {Y_t}^{\prime}, {U_t}\right)-f\left(t, {Y_t}^{\prime}, {U_t}\right) \\
= & a(t) \tilde Y_t+f\left(t, \theta {Y_t}^{\prime}, {U_t}\right)-f\left(t, {Y_t}^{\prime}, {U_t}\right),
\end{aligned}
$$
where $a(t)=\left[f\left(t, Y_t, {U_t}\right)-f\left(t, \theta {Y_t}^{\prime}, {U_t}\right)\right] / \tilde Y_t$ when $\tilde Y_t \neq 0$ and $a(t)=\beta$ otherwise. Since $f$ is $\beta$-Lipschitz in $y$, then $a$ is bounded by $\beta$ and
$$
f\left(t, Y_t, {U_t}\right)-f\left(t, {Y_t}^{\prime}, {U_t}\right) \leq a(t) \tilde Y_t+(1-\theta) \beta\left|{Y_t}^{\prime}\right|.
$$
Recalling the definition (\ref{def-F}) of $F$, we obtain
\begin{equation}
\label{F}
F_t \leq(1-\theta)\left(\alpha(t)+2 \beta\left|{Y_t}^{\prime}\right|\right)+\frac{1-\theta}{\lambda}j_{\lambda}(\frac{\tilde U_t}{1-\theta})+\theta\delta f(t).
\end{equation}

Now, we will eliminate the exponential term with an exponential change of variables. Let $c > 0$ and set $P_t=e^{c e^{\tilde A_t}\tilde Y_t}, Q_t=P_{t^-}(e^{ce^{\tilde A_t}\tilde U_t}-1)$. Applying It\^o’s formula, we deduce that
$$
\begin{aligned}
P_t&=P_T+\int_t^T P_{s^-}\left[ce^{\tilde A_s}F_s-\int_E(e^{ce^{\tilde A_s}\tilde U_s}-ce^{\tilde A_s}\tilde U_s-1)\phi_s(de)\right]dA_s-\int_t^T\int_EP_{s^-}(e^{ce^{\tilde A_s}\tilde U_s}-1)q(dsde)\\
&:=P_T+\int_t^T G_s d A_s-\int_t^T \int_E Q_s q(dsde).
\end{aligned}
$$
Equation (\ref{F}) yields because $c$ is nonnegative,
$$
\begin{aligned}
G_t= & P_{t^-}\left[ce^{\tilde A_t}F_t-\int_E(e^{ce^{\tilde A_t}\tilde U_t}-ce^{\tilde A_t}\tilde U_t-1)\phi_t(de)\right]\\
\le & P_{t^-}ce^{\tilde A_t}(1-\theta)\left(\alpha(t)+2 \beta\left|{Y_t}^{\prime}\right|\right)+P_{t^-}ce^{\tilde A_t}\theta\delta f(t)\\
&\quad +P_{t^-}\left[\int_E\left[\frac{ce^{\tilde A_t}(1-\theta)}{\lambda}\left(e^{\frac{\lambda\tilde U_t}{1-\theta}}-\frac{\lambda\tilde U_t}{1-\theta}-1\right)-\left(e^{ce^{\tilde A_t}\tilde U_t}-ce^{\tilde A_t}\tilde U_t-1\right)\right]\phi_t(de)\right].
\end{aligned}
$$
Because $\tilde A_t \geq-\beta A_T$, we can choose $c=\zeta_\theta:=\lambda e^{\beta \|A_T\|_\infty} /(1-\theta)$ to find the following inequality.
\begin{equation}
\label{G-estimate}
G_t \leq P_{t^{-}} e^{\tilde A_t}\lambda e^{\beta \|A_T\|_\infty}\left(\alpha(t)+2 \beta\left|{Y_t}^{\prime}\right|\right).
\end{equation}
We present a proof for (\ref{G-estimate}) below. Note that $f\le f'$, $\delta f(t)\le 0$. Therefore, it is enough to demonstrate that
\begin{equation}
\label{negative integral}
\int_E\left[\frac{ce^{\tilde A_t}(1-\theta)}{\lambda}\left(e^{\frac{\lambda\tilde U_t}{1-\theta}}-\frac{\lambda\tilde U_t}{1-\theta}-1\right)-\left(e^{ce^{\tilde A_t}\tilde U_t}-ce^{\tilde A_t}\tilde U_t-1\right)\right]\phi_t(de):=\int_E g_t(c,\tilde U_t)\phi_t(de)\le 0,
\end{equation}
where for each $v\in\mathbb R$ and $c>0$,
$$
g_t(v,c):=\frac{ce^{\tilde A_t}(1-\theta)}{\lambda}\left(e^{\frac{\lambda v}{1-\theta}}-\frac{\lambda v}{1-\theta}-1\right)-\left(e^{ce^{\tilde A_t}v}-ce^{\tilde A_t}v-1\right).
$$
Because $\phi_t(\cdot)$ is a probability measure, it is enough to show that $g_t(v,c)\le 0$ for every $v\in\mathbb R$ and $c>0$, which will be chosen later.
Choosing $c=\zeta_\theta:=\lambda e^{\beta \|A_T\|_\infty} /(1-\theta)$, we claim
$$
g_t(v, \zeta_\theta)=e^{\tilde A_t+\beta\|A_T\|_\infty}(e^{\frac{\lambda v}{1-\theta}}-1)-e^{\frac{\lambda v}{1-\theta}e^{(\tilde A_t+\beta \|A_T\|_\infty)}}+1\le 0.
$$
Notice that $g_t(0,\zeta_\theta)=0$, and for $v\in\mathbb R$, taking the derivative with respect to $v$,
$$
g_t'(v,\zeta_\theta)=\frac{\lambda}{1-\theta}e^{\tilde A_t+\beta\|A_T\|_\infty}\left(e^{\frac{\lambda v}{1-\theta}}-e^{\frac{\lambda v}{1-\theta}e^{(\tilde A_t+\beta \|A_T\|_\infty)}}\right).
$$
Therefore, $g_t'(v,\zeta_\theta)>0$ if $v<0$, and $g_t'(v,\zeta_\theta)\le0$ if $v\ge 0$, which implies $g_t(v,\zeta_\theta)\le g_t(0,\zeta_\theta)=0$. Thus (\ref{negative integral}) holds, which implies (\ref{G-estimate}).

Finally, we introduce the processes
$$
\begin{gathered}
D_t=\exp \left(\int_0^t e^{\tilde A_s}\left(\lambda e^{\beta \|A_T\|_\infty}\left(\alpha(s)+2 \beta\left|Y_s^{\prime}\right|\right)\right) d A_s\right), \\
\widetilde{P}_t=D_t P_t, \quad \widetilde{Q}_t=D_t Q_t.
\end{gathered}
$$
Again, from It\^o’s formula, it follows that for any stopping time $\tau$ such that $0 \leq t \leq \tau \leq T$,
\begin{equation}
\label{martingale}
\widetilde{P}_t \leq \widetilde{P}_\tau-\int_t^\tau\int_E \widetilde{Q}_sq(dsde).
\end{equation}
Let us consider that for $n \geq 1, \tau_n>t$ is the localization sequence such that $\tau_n\to\infty$ as $n\to\infty$ and,
$$
\int_t^{\tau_n\wedge\cdot}\int_E\widetilde{Q}_sq(dsde)
$$
is a martingale.

We deduce from (\ref{martingale}) that
$$
P_t \leq \mathbb{E}\left[\exp\left(\int_t^{\tau_n}e^{\tilde A_s}\left(\lambda e^{\beta \|A_T\|_\infty}\left(\alpha(s)+2 \beta\left|Y_s^{\prime}\right|\right)\right) d A_s\right) P_{\tau_n} \mid \mathcal{F}_t\right],
$$
and, in view of the integrability assumption on $\alpha, Y$ and $Y^{\prime}$, since $\left|\tilde A_s\right| \leq \beta A_T$, we can send $n$ to infinity to obtain
$$
P_t \leq \mathbb{E}\left[\exp\left(\int_t^{T}e^{\tilde A_s}\left(\lambda e^{\beta \|A_T\|_\infty}\left(\alpha(s)+2 \beta\left|Y_s^{\prime}\right|\right)\right) d A_s\right) P_{T} \mid \mathcal{F}_t\right].
$$
Equivalently,
$$
\begin{aligned}
& \exp \left(\frac{\lambda e^{\beta\|A_T\|_\infty+\tilde A_t}}{1-\theta}\left(Y_t-\theta {Y_t}^{\prime}\right)\right) \\
& \quad \leq \mathbb{E}\left(\exp \left\{\lambda e^{2 \beta ||A_T\|_\infty}\left(|\xi|+\int_t^T\left(\alpha(s)+2 \beta\left|Y_s^{\prime}\right|\right) d A_s\right)\right\} \mid \mathcal{F}_t\right).
\end{aligned}
$$
In particular, because $\beta \|A_T\|_\infty+\tilde A_t \geq 0$,
$$
Y_t-\theta {Y_t}^{\prime} \leq \frac{1-\theta}{\lambda} \log \mathbb{E}\left(\exp \left\{\lambda e^{2 \beta \|A_T\|_\infty}\left(|\xi|+\int_t^T\left(\alpha(s)+2 \beta\left|Y_s^{\prime}\right|\right) d A_s\right)\right\} \mid \mathcal{F}_t\right)
$$
and sending $\theta$ to 1, we get $Y_t-{Y_t}^{\prime} \leq 0$ and complete the proof of the comparison theorem.
\end{proof}

\begin{remark}
\label{concave01}
When $f$ is concave in $u$, we use $\theta Y-Y'$ in the aforementioned proof. The proof remains valid. Therefore, in the following discussion, unless otherwise stated, we assume that $f$ is convex in $u$ in Section \ref{section non reflected}.
\end{remark}

\begin{corollary}[Uniqueness]
\label{uniqueness}
If assumptions (H1)–(H3) hold, then the BSDE (\ref{BSDE}) admits at most one solution $(Y, U)$ in the space $ \mathcal E\times H_{\nu}^{2,p}$, for all $p\ge 1$.
\end{corollary}

\subsection{Existence of exponential growth BSDE with bounded terminal}
To prove the existence of the result, we require the following additional assumption.
\hspace*{\fill}\\

\noindent\textbf{(H4) (Uniform linear bound condition)}
There exists a positive constant $C_0$ such that for each $t \in[0, T]$, $u\in L^2(E,\mathcal{B}(E),\phi_t(\omega,dy))$, if
 $f$ is convex (resp. concave) in $u$, then $f(t,0,u)-f(t,0,0)\ge -C_0\|u\|_t$ (resp. $f(t,0,u)-f(t,0,0)\le C_0\|u\|_t$).

\hspace*{\fill}\\

Before proceeding with the proof, we require the following lemma, which establishes the essential properties of the auxiliary drivers. In the following discussion, unless stated otherwise, we assume that $f$ is convex with respect to $u$. For $t\in[0,T]$,
on $(y, u) \in \mathbb R \times L^2(E,\mathcal{B}(E),\phi_t(\omega,dy))$, we define a set of auxiliary generators $\left(f^{n}\right)_n$ as follows.
$$
\begin{aligned}
{f}^{n}(t, y, u)=\inf _{ r\in L^{2}(E,\mathcal{B}(E),\phi_t(\omega,dy))} \left\{f(t, y, r)+n \|u-r\|_t\right\}.
\end{aligned}
$$
The properties of the auxiliary drivers are outlined below.

\begin{lemma}
\label{lemma monotonicity}
Under the assumptions (H1)–(H4),

(i) The sequence $\{{f}^{n}\}_n$ is globally Lipschitz with respect to (y, u) in $\mathbb R \times L^2(E,\mathcal{B}(E),\phi_t(\omega,dy))$.

(ii)
The sequence $\{f^{n}\}_n$ is convex with respect to $u$ if $f$ is convex with respect to $u$ for $u\in L^2(E,\mathcal{B}(E),\phi_t(\omega,dy))$.

(iii)
For $t\in[0,T]$, the sequence $\{{f}^{n}\}_n$ converges to ${f}$ on $(y, u) \in \mathbb R \times L^2(E,\mathcal{B}(E),\phi_t(\omega,dy))$.

(iv) For $n>C_0$,
\begin{equation}
\label{fn bound}
-3\alpha_t-3\beta|y|-\frac{1}{\lambda}j_\lambda(t,-u)\leq f^{n}(t,y,u)\leq f(t,y,u) \leq \alpha_t+\beta|y|+\frac{1}{\lambda}j_\lambda(t,u)\leq 3\alpha_t+3\beta|y|+\frac{1}{\lambda}j_\lambda(t,u).
\end{equation}

(v) For each $n>C_0$ and $t\in[0,T]$, $f^{n}(t,0,0)=f(t,0,0)$.
\end{lemma}
\begin{proof}[Proof of Lemma \ref{lemma monotonicity}]
Assertions (i)–(iii) follow directly from, for example, \cite[Lemma 1]{Lepeltier_1997}. We only prove assertion (iv) and (v). For assertion (iv), the right-hand side is evident from the sequence $(f^n)_n$ definition. For the left-hand side, first note that owing to the convexity of $f^n$ with respect to $u$, it holds for every $t\in [0,T]$ and $(y, u) \in \mathbb R \times L^2(E,\mathcal{B}(E),\phi_t(\omega,dy))$,
$$
f^n(t,y,0)\le \frac{1}{2}f^n(t,y,u)+\frac{1}{2}f^n(t,y,-u).
$$
Thus, owing to the right-hand side of (\ref{fn bound}), it follows that for $n > C_0$,
$$
\begin{aligned}
f^n(t,y,u)&\ge 2f^n(t,y,0)-f^n(t,y,-u)\\
&\ge 2\inf _{ r\in L^{2}(E,\mathcal{B}(E),\phi_t(\omega,dy))} \left\{f(t, y, r)+n \|r\|_t\right\}-\alpha_t-\beta|y|-\frac{1}{\lambda}j_\lambda(t,-u)\\
&\ge 2\inf _{ r\in L^{2}(E,\mathcal{B}(E),\phi_t(\omega,dy))} \left\{f(t, 0, r)+n \|r\|_t\right\}-\alpha_t-3\beta|y|-\frac{1}{\lambda}j_\lambda(t,-u)\\
&\ge 2\inf _{ r\in L^{2}(E,\mathcal{B}(E),\phi_t(\omega,dy))} \left\{-C_0\|r\|_t+n \|r\|_t\right\}+2f(t,0,0)-\alpha_t-3\beta|y|-\frac{1}{\lambda}j_\lambda(t,-u)\\
&\ge -3\alpha_t-3\beta|y|-\frac{1}{\lambda}j_\lambda(t,-u),
\end{aligned}
$$
where we employ condition (H4) in the second-to-last inequality and condition (H3)(c) in the final inequality.

For assertion (v), according to the definition of the sequence $\{f^n\}$, for each $n\ge 1$ and $t\in[0,T]$, $f^n(t,0,0)\le f(t,0,0)$. On the other hand, for $n>C_0$ and $t\in[0,T]$, from assumption (H4):
$$
\begin{aligned}
f^n(t,0,0)&=\inf _{ r\in L^{2}(E,\mathcal{B}(E),\phi_t(\omega,dy))} \left\{f(t, 0, r)+n \|r\|_t\right\}\\
&=\inf _{ r\in L^{2}(E,\mathcal{B}(E),\phi_t(\omega,dy))} \left\{f(t, 0, r)-f(t,0,0)+n \|r\|_t\right\}+f(t,0,0)\\
&\ge \inf _{ r\in L^{2}(E,\mathcal{B}(E),\phi_t(\omega,dy))} \left\{-C_0\|r\|_t+n \|r\|_t\right\}+f(t,0,0)\\
&\ge f(t,0,0).
\end{aligned}
$$
Thus, for each $n>C_0$ and $t\in[0,T]$, $f^n(t,0,0)=f(t,0,0)$.

\end{proof}

\begin{remark}
If $f$ is concave in $u$, the auxiliary generators should be defined as follows
$$
\begin{aligned}
\tilde {f}^{n}(t, y, u)=\sup _{ r\in L^{2}(E,\mathcal{B}(E),\phi_t(\omega,dy))} \left\{f(t, y, r)-n \|u-r\|_t\right\}.
\end{aligned}
$$
Similar properties hold by a similar argument.
\end{remark}

Given the properties of the auxiliary generators, we are prepared to construct a solution of BSDE (\ref{BSDE}) with a bounded terminal. It is noteworthy that in light of (\ref{fn bound}), the parameters $(\alpha,\ \beta)$ in the a priori estimates (\ref{eY bound}) and (\ref{U bound}) are replaced by $(3\alpha,\ 3\beta)$ when estimating solutions of BSDE($\xi,\ f^n$).

\begin{thm}
\label{thm Quadratic bounded}
Assuming that (H1)–(H4) and  and (H3')(c) hold, with the additional condition $|\xi|<M_0$, where $M_0$ is a positive constant, the BSDE (\ref{BSDE}) has a unique solution $(Y, U)$ in the space $ \mathcal E\times H_{\nu}^{2,p}$, for all $p\ge 1$.
\end{thm}
The uniqueness stated in Theorem \ref{thm Quadratic bounded} follows from Corollary \ref{uniqueness}. Now, let us focus on the existence part of Theorem \ref{thm Quadratic bounded}. It is important to note that we consistently assume $f$ is convex with respect to $u$.

\begin{proof}[Proof of existence in Theorem \ref{thm Quadratic bounded}]
With the aid of Lemma \ref{lemma monotonicity} and Theorem \ref{thm_lip}, for $n>C_0$, there exists a unique solution $(Y^n,U^n)\in L^2(A)\times H_{\nu}^{2,2}$ for the BSDE($\xi,\ f^n$). For $n>C_0$, to find a uniform estimate for $(Y^n,\
 U^n)$, we first consider BSDE ($f^{n,k},\ \xi$), where $f^{n,k}=(f^n\wedge -k)\vee k$, for $k\in\mathbb N$. By employing the definition of truncated generators, it follows that for each $k\ge 1$, $t\in[0,T]$ and $(y,\ y',\ u,\ u') \in \mathbb R \times \mathbb R\times L^2(E,\mathcal{B}(E),\phi_t(\omega,dy))\times L^2(E,\mathcal{B}(E),\phi_t(\omega,dy))$,
$$
|f^{n,k}(t,y,u)-f^{n,k}(t,y',u')|\le|f^n(t,y,u)-f^n(t,y',u')|\le 3\beta|y-y'|+n\|u-u'\|_t.
$$
Thus, making use of Theorem \ref{thm_lip} again, for $n>C_0$, there exists a unique solution $(Y^{n,k},U^{n,k})\in L^2(A)\times H_{\nu}^{2,2}$ for BSDE($\xi,\ f^{n,k}$). Moreover, note that
$$
Y_t^{n,k}=\mathbb E_t\left[\xi+\int_t^Tf^{n,k}(t,Y_t^{n,k},U_t^{n,k})dA_t\right].
$$
It then holds,
\begin{equation}
\label{bdd Ynk}
\left\|Y_t^{n,k}\right\|_{S^\infty}\le M_0+k\|A_T\|_\infty<\infty.
\end{equation}
In addition, by Corollary 1 of \cite{Morlais_2009}, it holds on $[0,T]$, $\mathbb P\text{-a.s.}$ that
\begin{equation}
\label{unkbdd}
|U_s^{n,K}|\le2\left\|Y_t^{n,k}\right\|_{S^\infty}\le 2M_0+2k\|A_T\|_\infty<\infty.
\end{equation}
Notice also that
$$
f^{n,k}(t,y,u)\le(f^{n,k})^+(t,y,u)\le (f^{n})^+(t,y,u)\le 3\alpha_t+3\beta|y|+\frac{1}{\lambda}j_{\lambda}(t,u),
$$
and
$$
f^{n,k}(t,y,u)\ge-(f^{n,k})^-(t,y,u)\ge -(f^{n})^-(t,y,u)\ge -3\alpha_t-3\beta|y|-\frac{1}{\lambda}j_{\lambda}(t,-u),
$$
where $(f^{n,k})^+=f^{n,k}1_{\{f^{n,k}\ge 0\}}$, and $(f^{n,k})^-=-f^{n,k}1_{\{f^{n,k}< 0\}}$.
Then,
\begin{equation}
\label{fnk bound}
-3\alpha_t-3\beta|y|-\frac{1}{\lambda}j_\lambda(t,-u)\leq f^{n,k}(t,y,u)\leq 3\alpha_t+3\beta|y|+\frac{1}{\lambda}j_\lambda(t,u).
\end{equation}
Therefore, considering the a priori estimate provided by Lemma \ref{submartingle property}, along with (\ref{bdd Ynk}), (\ref{unkbdd}) and (\ref{fnk bound}), we can deduce that for every $n>C_0$, $p,q\ge 1$ and $k\ge 1$,
\begin{equation}
\label{Ykn}
\exp \left\{p \lambda\left|Y^{n,k}_t\right|\right\} \leq \mathbb{E}_t\left[\exp \left\{p \lambda e^{3\beta A_T}|\xi|+p\lambda \int_t^T e^{3\beta A_s}3\alpha_s d A_s\right\}\right],
\end{equation}
and
\begin{equation}
\label{Ukn}
\mathbb E\left[\left(\int_0^T\int_E\left(e^{q\lambda |U_t^{n,k}(e)|}-1\right)^2 \phi_t\left(de\right) d A_t\right)^p\right]\le  C_p\Xi(8pq\lambda(1+3\beta\|A_T\|_\infty),3\alpha,3\beta).
\end{equation}

The remaining proof is divided into five steps. First, we establish the convergence of $\{(Y^{n,k}, U^{n,k})\}_k$ to $(Y^n, U^n)$ and demonstrate that the solutions  $(Y^n,U^n)$ satisfy Definition \ref{def solution} with $Y^n\in\mathcal E$. Subsequently, leveraging Lemma \ref{submartingle property} and Proposition \ref{priori estimate on U and K}, we derive a uniform estimate for $\{Y^n, U^n\}$. In the second step, we construct a "presolution" $Y^0$ employing the comparison theorem, complemented by finding a corresponding "presolution" $U^0$. The third step involves deriving an a priori estimate of $|Y^n-Y^m|$. Utilizing this estimate and the uniform convergence in probability (u.c.p.) of the sequence $Y^n$, we obtain the c\`adl\`ag version of $Y^0$, denoted as $\tilde Y^0$, in the fourth step. Finally, in the last step, we confirm that $(\tilde Y^0, U^0)$ constitutes a solution within appropriate spaces.

\textbf{Step 1: The convergence of the sequence $\{(Y^{n,k}, U^{n,k})\}_k$ to $(Y^n, U^n)$.}

In this step, we demonstrate that $\lim_{k\to\infty}\mathbb E\left[|Y^n_t-Y^{n,k}_t|^2\right]=0$ and $\lim_{k\to\infty}\mathbb E\left[\int_0^T\int_E\left|U^{n,k}_t-U^{n}_t\right|^2\phi_t(de)dA_t\right]=0$, for a fixed $n>C_0$, for a fixed $n>C_0$. From It\^o’s formula,
$$
\begin{aligned}
d(Y^{n,k}_t-Y^n_t)^2 &=2(Y_t^{n,k}-Y_t^{n})d(Y_t^{n,k}-Y_t^n)+\int_E|U_t^{n,k}-U_t^n|^2p(dtde)\\
&=2(Y_t^{n,k}-Y_t^n)\left(-f^{n,k}(t,Y_t^{n,k},U_t^{n,k})+f^{n}(t,Y_t^n,U_t^n)\right)dA_t+2(Y_t^{n,k}-Y_t^n)\int_E(U_t^{n,k}-U_t^n)q(dtde)\\
&\quad+\int_E\left|U_t^{n,k}-U_t^n\right|^2q(dtde)+\int_E\left|U_t^{n,k}-U_t^n\right|^2\phi_t(de)dA_t.
\end{aligned}
$$
Then, by integrating from $t$ to $T$ and rearranging the terms, and considering the Lipschitz conditions of $f^{n,k}$, we obtain, owing to $Y_T^{n,k}=Y_T^n=\xi$,
$$
\begin{aligned}
&|Y^{n,k}_t-Y^n_t|^2+\int_t^T\int_E\left|U_s^{n,k}-U_s^n\right|^2\phi_s(de)dA_s\\
&=2\int_t^T(Y_{s}^{n,k}-Y_{s}^n)\left(f^{n,k}(s,Y_s^{n,k},U_s^{n,k})-f^{n}(s,Y_s^n,U_s^n)\right)dA_s-\int_t^{T}2(Y_{s^-}^{n,k}-Y_{s^-}^n)\int_E(U_s^{n,k}-U_s^n)q(dsde)\\
&\quad\quad-\int_{t}^{T}\int_E\left|U_s^{n,k}-U_s^n\right|^2q(dsde)\\
&\le 2\int_t^T|Y_s^{n,k}-Y_s^n|\left(\left|f^{n,k}(s,Y_s^{n,k},U_s^{n,k})-f^{n,k}(s,Y_s^n,U_s^n)\right|+\left|f^{n,k}(s,Y_s^{n},U_s^{n})-f^{n}(s,Y_s^n,U_s^n)\right|\right)dA_s\\
&\quad\quad-\int_t^{T}2(Y_{s^-}^{n,k}-Y_{s^-}^n)\int_E(U_s^{n,k}-U_s^n)q(dsde)-\int_{t}^{T}\int_E\left|U_s^{n,k}-U_s^n\right|^2q(dsde)\\
&\le 2\int_t^T\left(3\beta|Y_s^{n,k}-Y_s^n|^2+n|Y_s^{n,k}-Y_s^n|\|U_s^{n,k}-U_s^n\|_s+|Y_s^{n,k}-Y_s^n|\left|f^{n,k}(s,Y_s^{n},U_s^{n})-f^{n}(s,Y_s^n,U_s^n)\right|\right)dA_s\\
&\quad\quad-\int_t^{T}2(Y_{s^-}^{n,k}-Y_{s^-}^n)\int_E(U_s^{n,k}-U_s^n)q(dsde)-\int_{t}^{T}\int_E\left|U_s^{n,k}-U_s^n\right|^2q(dsde)\\
&\le 2\int_t^T\left(3\beta|Y_s^{n,k}-Y_s^n|^2+{n^2}|Y_s^{n,k}-Y_s^n|^2+\frac{1}{4}\|U_s^{n,k}-U_s^n\|_s^2+\frac{1}{2}|Y_s^{n,k}-Y_s^n|^2\right.\\
&\quad\quad\quad\quad\left.+\frac{1}{2}\left|f^{n,k}(s,Y_s^{n},U_s^{n})-f^{n}(s,Y_s^n,U_s^n)\right|^2\right)dA_s\\
&\quad\quad-\int_t^{T}2(Y_{s^-}^{n,k}-Y_{s^-}^n)\int_E(U_s^{n,k}-U_s^n)q(dsde)-\int_{t}^{T}\int_E\left|U_s^{n,k}-U_s^n\right|^2q(dsde),
\end{aligned}
$$
where we make use of the inequality of arithmetic and geometric means in the last line. Now, rearranging the terms again, we obtain
\begin{equation}
\label{grown1}
\begin{aligned}
&|Y^{n,k}_t-Y^n_t|^2+\frac{1}{2}\int_t^T\int_E\left|U_s^{n,k}-U_s^n\right|^2\phi_s(de)dA_s\\
&\quad\le (1+2n^2+6\beta)\int_t^T|Y_s^{n,k}-Y_s^n|^2dA_s+\int_t^T\left|f^{n,k}(s,Y_s^{n},U_s^{n})-f^{n}(s,Y_s^n,U_s^n)\right|^2dA_s\\
&\quad\quad-\int_t^{T}2(Y_{s^-}^{n,k}-Y_{s^-}^n)\int_E(U_s^{n,k}-U_s^n)q(dsde)-\int_{t}^{T}\int_E\left|U_s^{n,k}-U_s^n\right|^2q(dsde).
\end{aligned}
\end{equation}
Inspired by the proof of stochastic Gronwall inequality (as detailed in \cite[Theorem 4]{SCHEUTZOW_2013}), for any given $t\in[0,T]$, applying It\^o's formula to the process
$$
\exp\left\{\int_{t}^{\cdot}{(1+2n^2+6\beta)}dA_s\right\}|Y^{n,k}_{\cdot}-Y^n_{\cdot}|^2,
$$
by (\ref{grown1}), it turns out that on $(t,T]$,
\begin{equation}
\label{eqgr}
\begin{aligned}
&d\left(\exp\left\{\int_{t}^s{(1+2n^2+6\beta)}dA_u\right\}|Y^{n,k}_s-Y^n_s|^2\right)\\
&\ge\exp\left\{\int_{t}^s{(1+2n^2+6\beta)}dA_u\right\}\left(-\left|f^{n,k}(s,Y_s^{n},U_s^{n})-f^{n}(s,Y_s^n,U_s^n)\right|^2dA_s\right.\\
&\left.+2(Y_{s^-}^{n,k}-Y_{s^-}^n)\int_E(U_s^{n,k}-U_s^n)q(dsde)+\int_E\left|U_s^{n,k}-U_s^n\right|^2q(dsde)\right).
\end{aligned}
\end{equation}
Note the fact that 
$$
\int_t^{\cdot}\exp\left\{\int_{t}^s{(1+2n^2+6\beta)}dA_u\right\}2(Y_{s^-}^{n,k}-Y_{s^-}^n)\int_E(U_s^{n,k}-U_s^n)q(dsde)$$ 
and
$$
\int_{t}^{\cdot}\exp\left\{\int_{t}^s{(1+2n^2+6\beta)}dA_u\right\}\int_E\left|U_s^{n,k}-U_s^n\right|^2q(dsde)
$$
are martingales, owing to the integrability conditions $Y^{n,k},\ Y^{n}\in L^2(A)$ and $U^{n,k},\ U^n\in H_{\nu}^{2,2}$ thanks to Theorem \ref{thm_lip}. 
Then, taking integration on $(t,T]$ and followed by taking expectation on both sides of (\ref{eqgr}), we find that
$$
\mathbb E[|Y^{n,k}_t-Y^n_t|^2]\le \mathbb E\left[\int_t^T\left|f^{n,k}(s,Y_s^{n},U_s^{n})-f^{n}(s,Y_s^n,U_s^n)\right|^2dA_s\right]\exp\{(1+2n^2+6\beta)\|A_T\|_\infty\}.
$$
Recall that $f^{n,k}\to f^n$ as $k\to\infty$ and by the dominated convergence theorem, it holds that
$$
\lim_{k\to\infty}\mathbb E[|Y^{n,k}_t-Y^n_t|^2]=0.
$$
Then, up to a subsequence, we have $\mathbb P\text{-a.s.}$, $Y_t^{n,k}\to Y_t^n$ as $k\to\infty$. Consequently, leveraging (\ref{Ykn}), for $n>C_0$, we obtain,
\begin{equation}
\label{sub mart |yn|}
\exp \left\{p \lambda\left|Y^{n}_t\right|\right\} \leq \mathbb{E}_t\left[\exp \left\{p \lambda e^{3\beta A_T}|\xi|+p\lambda \int_t^T e^{3\beta A_s}3\alpha_s d A_s\right\}\right].
\end{equation}
Therefore, $Y^n\in\mathcal E$, for $n>C_0$. Moreover, in (\ref{grown1}), we first take $t=0$ , then take expectation on both sides and finally let $k\to\infty$ to conclude that thanks to dominated convergence theorem, it holds
\begin{equation}
\label{cvgz unk01}
\lim_{k\to\infty}\mathbb E\left[\int_0^T\int_E\left|U_s^{n,k}-U_s^n\right|^2\phi_s(de)dA_s\right]=0.
\end{equation}
Then, up to a subsequence, we have 
$$
\lim_{k\to\infty}U_t^{n,k}(e)=U_t^n(e),\ \phi_t(\omega, d e) d A_t(\omega) \mathbb{P}(d \omega)-a.e.
$$
Therefore, in view of (\ref{Ukn}) and Fatou's Lemma, for each  $p,\ q\ge 1$,
\begin{equation}
\label{Un bound01}
\sup_{n>C_0}\mathbb E\left[\left(\int_0^T\int_E\left(e^{q\lambda |U_t^{n}(e)|}-1\right)^2 \phi_t\left(de\right) d A_t\right)^p\right]\le  C_p\Xi(8pq\lambda(1+3\beta\|A_T\|_\infty),3\alpha,3\beta).
\end{equation}
Hence, the sequence $\{U^n\}_n$ satisfies Definition \ref{def solution}, and using Proposition \ref{priori estimate on U and K}, we derive
\begin{equation}
\label{Un bound}
\sup_{n>C_0}\mathbb E\left[\left(\int_0^T\int_E|U_t^n(e)|^2\phi_t(de)dA_t\right)^{p/2}\right]\le C_p\Xi(8p\lambda(1+3\beta\|A_T\|_\infty),3\alpha,3\beta)<\infty.
\end{equation}

\textbf{Step 2: Construction of "presolution" $(Y^0, U^0)$. }

Recall that for $n>C_0$, $(Y^n, U^n)\in \mathcal E\times H_{\nu}^{2,2}$ is the unique solution to BSDE($\xi,\ f^n$) and $f^n\uparrow f$. Then, $Y^n_t$ increases with respect to $n$ for each $t\in[0,T]$. Let $Y_t^n\uparrow Y_t^0$ as $n\to\infty$. Then, the process $Y^0$ is adapted and admits a progressive version. Hence, we assume that the process $Y^0$ is progressively measurable. Moreover, by Fatou’s lemma, we have,
\begin{equation}
\label{Y0}
\mathbb E\left[\int_0^T|Y^0_s|^2dA_s\right]\le\mathbb E[\sup_{n>C_0}|Y^n_*|^2]\|A_T\|_\infty<\infty.
\end{equation}

Next, we find a "presolution" $U^0$. For any $m,\ n >C_0$, applying It\^o's formula to the process $\left|Y^n-Y^m\right|^2$, we can deduce that
\begin{equation}
\label{U-cauchy}
\begin{aligned}
&\int_0^T\int_E\left|U_s^n(e)-U_s^m(e)\right|^2 \phi_s(de)d A_s\\
&= -\left|Y_0^n-Y_0^m\right|^2+2 \int_0^T\left(Y_{s^-}^n-Y_{s^-}^m\right)\left(f^n\left(s, Y_s^n, U_s^n\right)-f^m\left(s, Y_s^m, U_s^m\right)\right) d A_s \\
&\quad -2 \int_0^T\int_E\left(\left(Y_{s^-}^n-Y_{s^-}^{m}\right)\left(U_s^n-U_s^m\right)+|U_s^n-U_s^m|^2\right) q(dsde) \\
&\leq 2\left(\int_0^T\left|Y_t^n-Y_t^m\right|^2dA_t\right)^{1/2}\left(\int_0^T\left(f^n\left(s, Y_s^n, U_s^n\right)-f^m\left(s, Y_s^m, U_s^m\right)\right)^2dA_t \right)^{1/2}\\
& \quad -\left|Y_0^n-Y_0^m\right|^2-2 \int_0^T\int_E\left(\left(Y_{s^-}^n-Y_{s^-}^{m}\right)\left(U_s^n-U_s^m\right)+|U_s^n-U_s^m|^2\right) q(dsde)\\
&\leq 2(\int_0^T\left|Y_t^n-Y_t^m\right|^2dA_t)^{1/2}\left(\int_0^T2\left(3\alpha_s+3\beta|Y_s^m|+\frac{1}{\lambda} \left(j_\lambda(U_s^m)+j_{\lambda}(-U_s^m)\right)\right)^2d A_s\right.\\
&\left.\quad+\int_0^T2\left(3\alpha_s+3\beta|Y_s^n|+\frac{1}{\lambda} \left(j_\lambda(U_s^n)+j_{\lambda}(-U_s^n)\right)\right)^2dA_s\right)^{1/2}\\
& \quad -\left|Y_0^n-Y_0^m\right|^2-2 \int_0^T\int_E\left(\left(Y_{s^-}^n-Y_{s^-}^{m}\right)\left(U_s^n-U_s^m\right)+|U^n-U^m|^2\right) q(dsde), \quad \mathbb P \text {-a.s.}
\end{aligned}
\end{equation}
The last term is a martingale due to the integrability condition of $Y^n,\ Y^m$ and $U^n,\ U^m$. Then, by taking expectations and applying H\"older’s inequality, we infer from (\ref{sub mart |yn|}) and (\ref{e|u|}) in the proof of {Proposition \ref{priori estimate on U and K}} that there exists a constant $c>0$, which may vary, such that
$$
\begin{aligned}
\mathbb E\left[\left(\int_0^T\int_E\left|U_s^n-U_s^m\right|^2\phi_s(e) d A_s\right)\right]&\le c\left(\mathbb E[\int_0^T\left|Y_t^n-Y_t^m\right|^2dA_t]\right)^{1/2}-\mathbb E[\left|Y_0^n-Y_0^m\right|^2]\\
&\le c\left(\mathbb E[\int_0^T\left|Y_t^n-Y_t^0\right|^2dA_t]+\mathbb E[\int_0^T\left|Y_t^m-Y_t^0\right|^2dA_t]\right)^{1/2}+2\mathbb E[\left|Y_0^n-Y_0^0\right|^2]+2\mathbb E[\left|Y_0^m-Y_0^0\right|^2].
\end{aligned}
$$

Hence, it follows from (\ref{Y0}) and the monotone convergence theorem that
$$
\lim _{N\rightarrow \infty} \sup_{m,n\ge N}\mathbb E\left[\left(\int_0^T\int_E\left|U_s^n(e)-U_s^m(e)\right|^2\phi_s(de) d A_s\right)\right]=0.
$$
Thus, $\{U^n\}$ is a Cauchy sequence in $H_\nu^{2,2}$, implying the existence of a $U^0\in H_\nu^{2,2}$ such that
\begin{equation}
\label{estimate-U}
\lim _{ n \rightarrow \infty} \mathbb E\left[\left(\int_0^T\int_E\left|U_s^n(e)-U_s^0(e)\right|^2\phi_s(de) d A_s\right)\right]=0.
\end{equation}
Thus, up to a subsequence,
$$
\lim _{ n \rightarrow \infty} \left(\int_0^T\int_E\left|U_s^n(e)-U_s^0(e)\right|^2\phi_s(de) d A_s\right)=0, \ \mathbb P\text{-a.s.}
$$
Indeed, thanks to Proposition \ref{priori estimate on U and K} and Fatou’s lemma, $U^0\in H_{\nu}^{2,p}$ for each $p\ge 1$. Specifically, for some positive constant $C_p$ depending on $p$,
\begin{equation}
\label{u0}
\mathbb E\left[\left(\int_0^T\int_E|U_t^0(e)|^2\phi_t(de)dA_t\right)^p\right]\le C_p\sup_{n>C_0}\mathbb E\left[\left(\int_0^T\int_E|U_t^n(e)|^2\phi_t(de)dA_t\right)^p\right]<\infty.
\end{equation}
In the same manner, for each $p\ ,q\ge 1$,
\begin{equation}
\label{U0e}
\mathbb E\left[\left(\int_0^T\int_E\left(e^{q\lambda |U_t^{0}(e)|}-1\right)^2 \phi_t\left(de\right) d A_t\right)^p\right]\le C_p\sup_{n>C_0}\mathbb E\left[\left(\int_0^T\int_E\left(e^{q\lambda |U_t^{n}(e)|}-1\right)^2 \phi_t\left(de\right) d A_t\right)^p\right]<\infty.
\end{equation}

Moreover, in the set $\{(t,\omega)\in[0,T]\times \Omega:dA_t(\omega)\not=0\}$, it holds that
$$
\lim _{n \rightarrow \infty} \left(\int_E\left|U_t^n(e)-U_t^0(e)\right|^2\phi_t(de) \right)=0.
$$

\textbf{Step 3: A priori estimate of $|Y^n-Y^m|$.}
For $m,\ n>C_0,\ \theta \in(0,1)$, next we will show that $\mathbb P$-a.s.
$$
\left|Y_t^n-Y_t^m\right| \leq(1-\theta)\left(\left|Y_t^m\right|+\left|Y_t^n\right|\right)+\frac{1-\theta}{\lambda} \ln \left(\sum_{i=1}^2 J_t^{m,n, i}\right), \quad t \in[0, T],
$$
where $\zeta_\theta \triangleq \frac{\lambda e^{3\beta \|A_T\|_\infty}}{1-\theta}$, and $J_t^{n,m, i} \triangleq E\left[J_T^{n,m, i} \mid \mathcal{F}_t\right]$ for $i=1,2$ such that
$$
\begin{aligned}
& J_T^{n,m, 1} \triangleq (D^m_T+D^n_T) \eta \text { with }
\eta \triangleq \exp \left\{\zeta_\theta e^{3\beta\|A_T\|}\left(1-\theta\right)|\xi|\right\}\le\exp\{\lambda e^{6\beta\|A_T\|_\infty}|\xi|\},\\
&D^m_t \triangleq \exp \left\{\lambda e^{6 \beta \|A_T\|} \int_0^t\left(3\alpha_s+6\beta\left|Y_s^m\right|\right) d A_s\right\},\\
&D^n_t \triangleq \exp \left\{\lambda e^{6 \beta \|A_T\|} \int_0^t\left(3\alpha_s+6\beta\left|Y_s^n\right|\right) d A_s\right\},t \in[0, T],\\
& J_T^{n,m,2}\triangleq \zeta_\theta e^{3\beta \|A_T\|} (D^m_T+D^n_T) \Upsilon_{n,m} \int_0^T\left|\Delta_{n,m} f(s)\right| d A_s,\\
&\Upsilon_{n,m} \triangleq \exp \left\{\zeta_\theta e^{3\beta\|A_T\|}\left(Y_*^n+Y_*^m\right)\right\},\\
&\Delta_{n,m} f(t) \triangleq |f^n\left(t, Y_t^m, U_t^m\right)-f^m\left(t, Y_t^m, U_t^m\right)|+|f^m\left(t, Y_t^n, U_t^n\right)-f^n\left(t, Y_t^n, U_t^n\right)|, t \in[0, T].
\end{aligned}
$$

With the aforementioned notations, we show
\begin{equation}
\label{gamma01}
\exp \left(\zeta_\theta e^{\tilde{A}_t^{n, m}} \tilde{Y}_t^{(n, m)}\right):=\Gamma_t^{n, m} \leq \underbrace{\mathbb{E}\left(J_T^{m, n, 1}+J_T^{m, n, 2} \mid \mathcal{F}_t\right)}_{=J_t^{m, n, 1}+J_t^{m, n, 2}}.
\end{equation}

We set $\tilde Y^{(n,m)} \triangleq Y^n-\theta Y^m, \tilde U^{(n,m)} \triangleq U^n-\theta U^m$ and define two processes
$$
a_t^{(n,m)} \triangleq \mathbf{1}_{\left\{\tilde Y_t^{(n,m)} \neq 0\right\}} \frac{f^n\left(t, Y_t^n, U_t^n\right)-f^n\left(t, \theta Y_t^ m, U_t^n\right)}{\tilde Y_t^{(n,m)}}-\beta \mathbf{1}_{\left\{\tilde Y_t^{(n,m)}=0\right\}}, \quad \tilde A_t^{(n,m)} \triangleq \int_0^t a_s^{(n,m)} d A_s, \quad t \in[0, T].
$$

Applying It\^o’s formula to the process $\Gamma^{n,m}$ defined in (\ref{gamma01}) on $s \in(t, T]$ yields
\begin{equation}
\label{ess01}
\begin{aligned}
\Gamma_t^{n,m}& =\Gamma_{T}^{n,m}+\int_t^{T} \Gamma_{s}^{n,m}\left[\zeta_\theta e^{\tilde A^{(n,m)}_s}F^{n,m}_s-\int_E(e^{\zeta_\theta e^{\tilde A^{(n,m)}_s}\tilde U^{(n,m)}_s}-\zeta_\theta e^{\tilde A^{(n,m)}_s}\tilde U^{(n,m)}_s-1)\phi_s(de)\right]dA_s\\
\quad &\quad -\int_t^{T}\int_E\Gamma_{s^-}^{n,m}(e^{\zeta_\theta e^{\tilde A^{(n,m)}_s}\tilde U^{n,m}_s}-1)q(dsde)\\
&:=\Gamma_{T}^{n,m}+\int_t^{ T} G^{n,m}_s d A_s-\int_t^{T} \int_E Q^{n,m}_s q(dsde),
\end{aligned}
\end{equation}
where,
\begin{equation}
\label{Fnm}
F_s^{n,m}=f^n\left(s, Y^{n}_s, U^{n}_s\right)-\theta {f}^m\left(s, {Y}^{m}_s, {U}^{m}_s\right)-a^{(n,m)}_s \tilde Y^{(n,m)}_s,
\end{equation}
\begin{equation}
\label{G}
G_s^{n,m}=\zeta_\theta \Gamma_s^{n,m} e^{\tilde A^{(n,m)}_s}\left(f^n\left(s, Y^{n}_s, U^{n}_s\right)-\theta {f}^m\left(s, {Y}^{m}_s, {U}^{m}_s\right)-a^{(n,m)}_s \tilde Y^{(n,m)}_s\right)-\Gamma_s^{n,m} j_1(\zeta_\theta e^{\tilde A^{(n,m)}_s}\tilde U_s^{(n,m)}).
\end{equation}
Similar to (\ref{f-bound}), (\ref{fn bound}) and the convexity of $f^n$ in $u$ imply that $d s \otimes d P$-a.e.
$$
f^n\left(s, Y_s^{m}, U_s^n\right) \leq \theta f^n\left(s, Y_s^m, U_s^m\right)+ (1-\theta)\left(3\alpha_s+3\beta\left|Y_s^m\right|\right)+\frac{1-\theta}{\lambda} j_{\lambda}(\frac{\tilde U_s^{(n,m)}}{1-\theta}),
$$
which with (\ref{G}) implies that $d s \otimes d P$-a.e.
$$
\begin{aligned}
G_s^{ n,m} & =\zeta_\theta \Gamma_s^{n,m} e^{\tilde A_s^{(n,m)}} \left(f^n\left(s, \theta Y_s^m, U_s^n\right)-\theta f^m\left(s, Y_s^m, U_s^m\right)\right)-\Gamma_s^{n,m} j_1(\zeta_\theta e^{\tilde A_s^{(n,m)}}\tilde U_s^{(n,m)})\\
& \leq \zeta_\theta \Gamma_s^{n,m} e^{\tilde A_s^{(n,m)}} \left(\left|f^n\left(s, \theta Y_s^m, U_s^n\right)-f^n\left(s, Y_s^m, U_s^n\right)\right|+f^n\left(s, Y_s^m, U_s^n\right)-\theta f^m\left(s, Y_s^m, U_s^m\right)\right)- \Gamma_s^{n,m} j_1(\zeta_\theta e^{\tilde A_s^{(n,m)}}\tilde U_s^{(n,m)})\\
& \le\lambda e^{6 \beta\|A_T\|} \Gamma_t^{n,m}\left(3\alpha_s+6\beta\left|Y_s^m\right|\right)+\zeta_\theta e^{2\beta \|A_T\|} \Gamma_s^{n,m}\left|\Delta_{n,m} f(s)\right|.
\end{aligned}
$$ Integration by parts gives
\begin{equation}
\label{ess04}
\begin{aligned}
\Gamma_t^{n,m} & \leq D_t^m \Gamma_t^{n,m} \\
&\leq D_{ T}^m \Gamma_{ T}^{n,m}-\zeta_\theta \int_t^{ T}\int_E D^m_s \Gamma_{s^-}^{n,m}(e^{\zeta_\theta e^{\tilde A^{n,m}_s}\tilde U^{n,m}_s}-1)q(dsde)\\
&\quad +\zeta_\theta e^{3\beta \|A_T\|} \int_t^{T} D^{m}_s \Gamma_s^{ n,m}\left|\Delta_{n,m} f(s)\right| d A_s\\
& \leq D_T^m\eta-\zeta_\theta \int_t^{ T}\int_E D^m_s \Gamma_{s^-}^{n,m}(e^{\zeta_\theta e^{\tilde A^{n,m}_s}\tilde U^{n,m}_s}-1)q(dsde)+J_T^{n,m,2}.
\end{aligned}
\end{equation}

For each $p \in(1, \infty)$, (\ref{sub mart |yn|}) and (\ref{Un bound}) imply that

\begin{equation}
\label{uniform bounded Y U}
\sup _{n >C_0}\mathbb E\left[e^{p \lambda Y_*^{n}}+\left(\int_0^T\int_E\left|U_s^{n}\right|^2\phi_t(de) d A_s\right)^p\right] \leq C_p \Xi\left(16 p \lambda(1+3\beta\|A_T\|_\infty),\ 3\alpha,\ 3\beta\right),
\end{equation}
where we use the notation $\Xi$ defined in Lemma \ref{submartingle property} to simplify notation. Thus, it follows that for $m,\ n>C_0$,
\begin{equation}
\label{ess-eta-nm}
\mathbb E\left[\eta^p\right] \leq\mathbb E\left[e^{p \zeta_\theta e^{3\beta \|A_T\|}(1-\theta)\left|\xi\right|}\right]\le\mathbb E\left[e^{p \lambda e^{6\beta \|A_T\|}\left|\xi\right|}\right] \leq \Xi\left(p\lambda e^{6\beta \|A_T\|},\ 3\alpha,\ 3\beta\right).
\end{equation}
\begin{equation}
\label{ess02}
\begin{aligned}
\mathbb E\left[\Upsilon_{n,m}^p\right] & \leq \frac{1}{2}\mathbb E\left[e^{2 p \zeta_\theta e^{3\beta \|A_T\|} Y_*^n}+e^{2 p \zeta_\theta e^{3\beta \|A_T\|} Y_*^m}\right] \leq {\Xi}\left(\frac{2 p\lambda}{1-\theta} e^{6 \beta \|A_T\|},\ 3\alpha,\ 3\beta\right).\\
\end{aligned}
\end{equation}

\begin{equation}
\label{ess03}
\begin{aligned}
&\mathbb E\left[\left(\int_0^T\left|\Delta_{n,m} f(s)\right| d A_s\right)^p\right]\\
&\quad\leq \mathbb E\left[\left(2\int_0^T3\alpha_sdA_s+ 2 \|A_T\|_\infty\left(3\beta Y_*^m+3\beta Y_*^n\right)+2\int_0^T\left(j_{\lambda}(U^m)+j_{\lambda}(-U^m)+j_{\lambda}(U^n)+j_{\lambda}(-U^n) \right)d A_s\right)^p\right].
\end{aligned}
\end{equation}
From (\ref{uniform bounded Y U}) and $D_T^m \leq \exp \left\{\lambda e^{6 \beta \|A_T\|} \int_0^T\left(3\alpha_t+6\beta Y_*^n\right) d A_s\right\},\ \mathbb \mathbb P$-a.s., we obtain that $D_T^m \in \mathbb{L}^p\left(\mathcal{F}_T\right)$. Thus, the random variables $J_T^{n,m, i},\ i=1,2$ are integrable with Young’s inequality and (\ref{ess-eta-nm})-(\ref{ess03}). In addition, $\eta^{n,m}_{T_K^t\wedge T}\le \exp\{\frac{\lambda e^{6\beta\|A_T\|_\infty}}{1-\theta}(Y^n_*+Y^m_*)\}$, where the right-hand side of the inequality is also integrable.

Note that $\int_t^{\cdot}\int_E D^m_s \Gamma_{s^-}^{n,m}(e^{\zeta_\theta e^{\tilde A^{(n,m)}_s}\tilde U^{(n,m)}_s}-1)q(dsde)$ is a true martingale due to the estimates (\ref{sub mart |yn|}) and (\ref{Un bound01}).
Taking $\mathbb E\left[\cdot \mid \mathcal{F}_t\right]$ in (\ref{ess04}),
it then yields
$$
\Gamma_t^{n,m} \leq \mathbb E_t[D_T^m\eta] +J_t^{n, m, 2}\le J_t^{n, m, 1}+J_t^{n, m, 2},\ \mathbb P\text{-a.s.}
$$
It then follows that
$$
Y_t^n-\theta Y_t^m \leq \frac{1-\theta}{\lambda} e^{-3\beta\|A_T\|-\tilde A_t^{n,m}} \ln \left(\sum_{i=1}^2 J_t^{n,m, i}\right) \leq \frac{1-\theta}{\lambda} \ln \left(\sum_{i=1}^2J_t^{n,m, i}\right), \quad \mathbb P \text {-a.s. }
$$
which implies
$$
Y_t^n-Y_t^m \leq(1-\theta)\left|Y_t^m\right|+\frac{1-\theta}{\lambda} \ln \left(\sum_{i=1}^2 J_t^{n,m, i}\right)\leq(1-\theta)(\left|Y_t^m\right|+|Y_t^n|)+\frac{1-\theta}{\lambda} \ln \left(\sum_{i=1}^2 J_t^{n,m, i}\right), \quad \mathbb P \text {-a.s. }
$$
Exchanging the roles of $Y^m$ and $Y^n$, we deduce the other side of the inequality.
Hence, for $m,\ n>C_0$,
\begin{equation}
\label{ess05}
|Y_t^m-Y_t^n| \leq(1-\theta)(\left|Y_t^m\right|+\left|Y_t^n\right|)+\frac{1-\theta}{\lambda} \ln \left(\sum_{i=1}^2 J_t^{m,n, i}\right), \quad \mathbb P \text {-a.s. }.
\end{equation}

\textbf{Step 4: U.C.P. convergence of the sequence $Y^n$. } From (\ref{ess05}), we apply Doob’s martingale inequality and H\"older’s inequality to have, for any $\delta>0$,
\begin{equation}
\label{prob}
\begin{aligned}
& \mathbb P\left(\sup _{t \in[0, T]}\left|Y_t^n-Y_t^m\right| \geq \delta\right) \leq \mathbb P\left((1-\theta)\left(Y_*^m+Y_*^n\right) \geq \delta / 2\right)+\mathbb P\left(\frac{1-\theta}{\lambda} \ln \left(\sum_{i=1}^2 J_*^{n,m, i}\right) \geq \delta / 2\right) \\
& \leq 2 \frac{1-\theta}{\delta}\mathbb E\left[Y_*^m+Y_*^n\right]+\sum_{i=1}^2 \mathbb P\left(J_*^{n,m, i} \geq \frac{1}{2} e^{\frac{\delta \lambda}{2(1-\theta)}}\right) \leq \frac{1-\theta}{\delta \lambda} \mathbb E\left[e^{2 \lambda Y_*^m}+e^{2 \lambda Y_*^n}\right]+2 e^{\frac{-\delta \lambda}{2(1-\theta)}} \sum_{i=1}^2 \mathbb E\left[J_*^{n,m, i}\right] \\
& \leq 2 \frac{1-\theta}{\delta \lambda} {\Xi}(2\lambda,\ 3\alpha, \ 3\beta)+4 e^{\frac{-\delta \lambda}{2(1-\theta)}}\left(\mathbb E\left[\left(D^m_T+D^n_T\right)^2\eta^2\right]\right)^{1/2}\\
&\quad+4 e^{\frac{-\delta \lambda}{2(1-\theta)}}\left(\mathbb E\left[\left(D^m_T+D^n_T\right)^2\left(\zeta_\theta e^{3\beta \|A_T\|} \Upsilon_{n,m} \int_0^T\left|\Delta_{n,m} f(s)\right| d A_s\right)^2 \right]\right)^{1/2}\\
&\leq 2 \frac{1-\theta}{\delta \lambda} {\Xi}(2\lambda,\ 3\alpha, \ 3\beta)+4 e^{\frac{-\delta \lambda}{2(1-\theta)}}\sup_{m,n>C_0}\left(\mathbb E\left[\left(D^m_T+D^n_T\right)^2\exp\{2\lambda e^{6\beta\|A_T\|_\infty}|\xi|\}\right]\right)^{1/2}\\
&\quad\quad+4 e^{\frac{-\delta \lambda}{2(1-\theta)}}\left(\sup_{m,n>C_0}\mathbb E\left[\left(D^m_T+D^n_T\right)^4\left(\zeta_\theta e^{3\beta \|A_T\|} \Upsilon_{n,m}\right)^4\right]\right)^{1/4}\left(\mathbb E\left[\left(\int_0^T\left|\Delta_{n,m} f(s)\right| d A_s\right)^4\right]\right)^{1/4}.
\end{aligned}
\end{equation}

To prove the u.c.p. convergence of the sequence $Y^n$, we first deal with the integral term $ \int_0^T\left|\Delta_{n,m} f(s)\right| d A_s$.
Consider a sequence of stopping times $\left(T_K\right)_{K \geq 0}$ defined as
\begin{equation}
\label{def-TKt}
T_K=\inf \left\{s>0, \mathbb{E}\left[\exp \left(e^ {3\beta
A_T}\lambda\left|\xi\right|+\int_0^T \lambda e^{3\beta A_s}3\alpha_s d A_s\right) \mid \mathcal{F}_s\right]>e^{\lambda K}\right\}.
\end{equation}
Relying on (\ref{integrability_condition}) in Lemma \ref{submartingle property} and (\ref{sub mart |yn|}), we deduce that$\left(T_K\right)_{K \geq 0}$ tends to $+\infty$ and $\sup_{n>C_0}\|Y^{n}_{\cdot}1_{\cdot<T_K}\|_{S^{\infty}(t,T]}\le K$. Thus, in view of \cite[Corollary 1]{Morlais_2009}, $\mathbb P\text{-a.s.}$, one can replace the process $U^n$ by the process $U^{n,K}$ such that $|U_s^{n,K}|\le 2K$ on $(t,T_K)$. 
Hence, $\mathbb P$-a.s., for $t<T_K$, on $\{|y|\le K,\|u\|_t\le 2K\}$, $|f(t,y,u)|\le 3\alpha_t+3\beta|y|+\frac{2}{\lambda}e^{2\lambda K }.$ Then, inspired by \cite[Lemma 1]{Lepeltier_1997}, on $\{(\omega,t):\ dA_t(\omega)>0\}\cap\{t< T_K\}$, $\lim_{n\to\infty}f^n(t,Y_t^n,U^n_t)=f(t,Y^0_t,U^0_t)$. Similarly, it can also be proven that $\lim_{N\to\infty}\sup_{m,n>N}|f^n(t,Y_t^m,U^m_t)-f(t,Y^0_t,U^0_t)|=0$.
Note that
$$
\Delta_{n,m}f(t)=\left|f^n\left(t, \omega, Y_t^m, U_t^m\right)-f^m\left(t, \omega, Y_t^m, U_t^m\right)\right|+ \left|f^m\left(t, \omega, Y_t^n, U_t^n\right)-f^n\left(t, \omega, Y_t^n, U_t^n\right)\right|.
$$
It is then apparent that on the set $\{(\omega,t):\ dA_t(\omega)>0\}\cap\{(\omega,t):\ t< T_K\}$,
$$
\lim _{N \rightarrow \infty}\sup_{m,n>N}\Delta_{n,m}f(t,\omega)=0.
$$
Therefore, given the uniform integrability conditions of $Y^n$, $Y^m$ $U^n$, $U^m$ on the set $\{(\omega,t):\ dA_t(\omega)>0\}\cap\{(\omega,t):\ t< T_K\}$ and (\ref{ess03}), applying the dominated convergence theorem, we conclude that for every $p\ge 1$,
\begin{equation}
\label{cvg_deltafmn}
\lim_{N\to\infty}\sup_{m,n\ge N}\mathbb E\left[\left(\int_0^{T\wedge T_K}\left|\Delta_{n,m}f(t)\right|dA_t\right)^p\right]=0.
\end{equation}
At the same time, based on the a priori estimates of the sequence $\{(Y^n,\ U^n)\}$ and H\"older’s inequality, we establish that for every $p\ge 1$, there exists a constant $C_p>0$, independent of $N$, such that
$$
\sup_{m,n\ge N}\mathbb E\left[\left(\int_{T\wedge T_K}^T\left|\Delta_{n,m}f(t)\right|dA_t\right)^p\right]=\sup_{m,n\ge N}\mathbb E\left[\left(\int_{T_K}^T\left|\Delta_{n,m}f(t)\right|dA_t\right)^p1_{T_K<T}\right]\le C_p\sqrt{\mathbb P(T_K<T)}.
$$
Subsequently, for every $p\ge 1$, in conjunction with (\ref{cvg_deltafmn}), it is valid that for each $K>0$,
$$
\lim_{N\to\infty}\sup_{m,n\ge N}\mathbb E\left[\left(\int_0^{T}\left|\Delta_{n,m}f(t)\right|dA_t\right)^p\right]\le 2^{p-1}\left(C_p\sqrt{\mathbb P(T_K<T)}\right)^p,
$$
Thus let $K\to\infty$, we obtain
\begin{equation}
\label{cvg_deltafmnT}
\lim_{N\to\infty}\sup_{m,n\ge N}\mathbb E\left[\left(\int_0^{T}\left|\Delta_{n,m}f(t)\right|dA_t\right)^p\right]=0,\ \forall p\ge 1.
\end{equation}

Consequently, first letting $n,m \rightarrow \infty$ in (\ref{prob}) and then letting $\theta\rightarrow 1$ yield
 $$
 \lim _{N\rightarrow \infty}\sup_{m,n\ge N} \mathbb P\left(\sup _{t \in[0, T]}\left|Y_t^n-Y_t^m\right| \geq \delta\right)=0,
 $$
 which implies that the sequence $Y^n$ uniformly converges in probability (u.c.p.).
 Then, up to a subsequence, there exists a process $\tilde Y^0$ such that $\lim_{n\to\infty}\sup_{t\in[0,T]}|Y^n_t-\tilde Y^0_t|=0$. Then, $\tilde Y^0$ is ca\`dla\`g. With the help of Lemma \ref{submartingle property} and Fatou’s lemma, $\tilde Y^0\in\mathcal E.$

Fix $p \in[1, \infty)$. Since $\mathbb E\left[\exp \left\{2 p \lambda \cdot \sup _{t \in[0, T]}\left|Y_t^n-\tilde Y_t^0\right|\right\}\right] \leq \frac{1}{2} E\left[e^{4 p \lambda Y_*^n}+e^{4 p \lambda \tilde Y_*^0}\right] \leq {\Xi}(4 p\lambda, 3\alpha, 3 \beta)$ holds for any $n >C_0$, it turns out that $\left\{\exp \left\{p \lambda \cdot \sup _{t \in[0, T]}\left|Y_t^n-\tilde Y_t^0\right|\right\}\right\}_{n >C_0}$ is a uniformly integrable sequence in $\mathbb{L}^1\left(\mathcal{F}_T\right)$. Then, it follows that $\lim _{n \rightarrow \infty} \mathbb E\left[\exp \left\{p \lambda \cdot \sup _{t \in[0, T]}\left|Y_t^n-\tilde Y_t^0\right|\right\}\right]=1$, which in particular implies that
\begin{equation}
\label{cvg_Y}
\lim _{n \rightarrow \infty}\mathbb E\left[\sup _{t \in[0, T]}\left|Y_t^n-\tilde Y_t^0\right|^q\right]=0, \quad \forall q \in[1, \infty).
\end{equation}

\textbf{Step 5: Verification of the solution $(\tilde Y^0,U^0)$. }

We have constructed a "presolution" $(\tilde Y^0,\ U^0)\in \mathcal E\times H_{\nu}^{2,p}$. It is left to verify that it is a true solution.
Consider the sequence of stopping times $\left(T_K\right)_{K \geq 0}$ defined above as in (\ref{def-TKt}). For $t \leq T_K\wedge T$, we obtain exactly as in step 4 on $\{(t,\omega), dA_t(\omega)\not=0\}\cap\{t<T_K\}$,
$$
\lim _{n \rightarrow \infty} f^n\left(t, \omega, Y_t^n, U_t^n\right)=f\left(t, \omega, \tilde Y_t^0, U_t^0\right).
$$
Thus, by the dominated convergence theorem,
$$
\begin{aligned}
&\mathbb{E}\left[\int_0^{T_K\wedge T}\left|f^{n}\left(t, Y_t^{n},U_t^{n}\right)-f\left(t, \tilde Y^0_t, U^0_t\right)\right| d A_t\right]\\
&=\mathbb{E}\left[\int_0^{T_K\wedge T}\left|f^{n}\left(t, Y_t^{n},U_t^{n,K}\right)-f\left(t, \tilde Y^{0}_t, U^{0,K}_t\right)\right| d A_t\right],
\end{aligned}
$$
which becomes zero when $n$ goes to infinity.
Consequently, by first letting $n$ approach infinity and subsequently allowing $K$ to go to infinity, in the same manner as (\ref{cvg_deltafmnT}), we observe that
\begin{equation}
\label{cvg_f}
\lim_{n\to\infty }\mathbb{E}\left[\int_0^{T}\left|f^{n}\left(t, Y_t^{n}, U_t^{n}\right)-f\left(t, \tilde Y^0_t,U^0_t\right)\right| d A_t\right]=0.
\end{equation}

Lastly, we need to confirm that $(\tilde Y^0,U^0)$ is a solution. By combining (\ref{cvg_Y}), (\ref{cvg_f}), (\ref{estimate-U}), and the Burkholder–Davis–Gundy inequality, we obtain
\begin{equation}
\label{verify_solution}
\begin{aligned}
&\mathbb E\left[\left|\tilde Y^0_t-(\xi+\int_t^T f\left(s,\tilde Y^0_s, U^0_s\right) d A_s -\int_t^T \int_E U^0_s(e) q(d s, d e))\right|\right]\\
&\quad\le\lim_{n\to\infty} \mathbb E\left[\left|\tilde Y^0_t-Y^n_t\right|\right]+\lim_{n\to\infty}\mathbb E\left[\left|\left(\int_t^T f\left(s,\tilde Y^0_s, U^0_s\right) d A_s-\int_t^T \int_E U^0_s(e) q(d s, d e)\right)\right.\right.\\
&\quad\quad\left.\left.-\left(\int_t^T f^n\left(s,Y^n_s, U^n_s\right) d A_s -\int_t^T \int_E U^n_s(e) q(d s, d e)\right)\right|\right]\\
&\quad= 0.
\end{aligned}
\end{equation}

Then, we find a solution $(\tilde Y^0, U^0)\in \mathcal E\times H_{\nu}^{2,p}$ for the BSDE($\xi, f$) when $\xi$ is bounded. Moreover, considering assumption (H3), Lemma \ref{submartingle property}, and Proposition \ref{priori estimate on U and K}, the a priori estimates (\ref{eY bound}), (\ref{U bound}) and (\ref{e|u|q01}) hold for $(\tilde Y^0, U^0)$. Meanwhile, the estimate (\ref{e|u|q01}) guarantees that $U^0$ satisfies Definition \ref{def solution}.

The proof is complete.

\end{proof}

\subsection{Existence of exponential growth BSDEs with an unbounded terminal}

Now, we proceed to establish the existence of solutions for the BSDE(\ref{BSDE}) with an unbounded terminal. The main theorem is stated as follows.
\begin{thm}
\label{thm Quadratic unbounded}
Given that (H1)–(H4) are satisfied, the BSDE (\ref{BSDE}) possesses a unique solution $(Y, U)$ in the space $\mathcal E\times H_{\nu}^{2,p}$, for all $p\ge 1$.
\end{thm}
The uniqueness aspect is also derived from Corollary \ref{uniqueness}, leaving us with the task of proving its existence. We achieve this by approximating the solution using the solutions of BSDEs with bounded terminals.  It is worth motioning that we drop assumption (H3')(c) in Theorem \ref{thm Quadratic bounded}, which can not be induced by (H1)-(H4).

\begin{proof}[Proof of the existence of BSDEs with unbounded terminals: the existence part of Theorem \ref{thm Quadratic unbounded}]
With a slight abuse of notation, following Theorem \ref{thm Quadratic bounded}, let us denote the unique solution to BSDE($\xi^n,\ \bar f^n$) by ($Y^n,\, U^n$), where $\xi^n=(\xi\wedge n)\vee -n$, $\bar f^n(t,\cdot,\cdot)=f(t,\cdot,\cdot)-f(t,0,0)+f^n(t,0,0)$, $f^n(t,0,0)=(f(t,0,0)\wedge n)\vee -n$. Hence, $(Y^n,U^n)\in \mathcal E\times H_{\nu}^{2,p}$ for every $p\ge 1$.  Meanwhile, note that for each $t\in[0,T]$ and $(y,u) \in \mathbb{R} \times L^2(E,\mathcal{B}(E),\phi_t(\omega,dy)):\ \mathbb{P}$-a.s.
\begin{equation}
\label{growth3}
-3\alpha_t-\beta|y|-\frac{1}{\lambda}j_{\lambda}(t,-u)\le f(t,y,u)-2\alpha_t\le \bar f^n(t,y,u)\le f(t,y,u)+2\alpha_t\le 3\alpha_t+\beta|y|+\frac{1}{\lambda}j_{\lambda}(t,u).
\end{equation}
Thus, according to Lemma \ref{submartingle property}, for each $n\in \mathbb N$
\begin{equation}
\label{sub mart |yn|2}
\exp \left\{p \lambda\left|Y^{n}_t\right|\right\} \leq \mathbb{E}_t\left[\exp \left\{p \lambda e^{\beta A_T}|\xi|+p\lambda  \int_t^T e^{\beta A_s}3\alpha_s d A_s\right\}\right].
\end{equation}
We use the sequence $(Y^n,U^n)$ to approximate the solution for BSDE($ \xi,\ f$). As in the proof of the bounded terminal case, we divide the proof into three steps. In the initial two steps, we establish a candidate solution $(Y^0,\ U^0)$, reserving the verification of the solution for the final step.

\textbf{Step 1: Construction of "presolution" $Y^0$.}
Using reasoning similar to that in the case of BSDEs with bounded terminals, we first derive an a priori estimate for $|Y^n-Y^m|$.
For $m,\ n\in\mathbb N,\ \theta \in(0,1)$, we will show that $\mathbb P$-a.s.
$$
\left|Y_t^n-Y_t^m\right| \leq(1-\theta)\left(\left|Y_t^m\right|+\left|Y_t^n\right|\right)+\frac{1-\theta}{\lambda} \ln \left(\sum_{i=1}^2\widetilde J_t
^{n,m,i}\right), \quad t \in[0, T],
$$
where $\zeta_\theta^0 \triangleq \frac{\lambda e^{\beta \|A_T\|_\infty}}{1-\theta}$, and $\widetilde J_t^{n,m,i} := \mathbb E\left[\widetilde J_T^{n,m,i} \mid \mathcal{F}_t\right]$, $i=1,2$, such that
$$
\begin{aligned}
&{ \widetilde {J}_T^{n,m,1} := (\widetilde D^m_T+\widetilde D^n_T)\widetilde \eta^{n,m} \text { with }
\widetilde\eta^{n,m} := \exp \left\{\zeta_\theta^0 e^{\beta\|A_T\|}\left(|\xi^m-\theta\xi^n|\vee|\xi^n-\theta\xi^m|\right)\right\},}\\
&\widetilde D^m_t := \exp \left\{\lambda e^{2 \beta \|A_T\|} \int_0^t\left(\alpha_s+2\beta\left|Y_s^m\right|\right) d A_s\right\},\\
&\widetilde D^n_t := \exp \left\{\lambda e^{2 \beta \|A_T\|} \int_0^t\left(\alpha_s+2\beta\left|Y_s^n\right|\right) d A_s\right\},t \in[0, T].\\
& \widetilde J_T^{n,m,2}:= \zeta_\theta^0 e^{\beta \|A_T\|} (\tilde D^m_T+\tilde D^n_T) \tilde \Upsilon_{n,m} \int_0^T\left|\tilde\Delta_{n,m}\bar f(s)\right| d A_s,\\
&\tilde\Upsilon_{n,m} := \exp \left\{\zeta_\theta^0 e^{\beta\|A_T\|}\left(Y_*^n+Y_*^m\right)\right\},\\
&\tilde\Delta_{n,m} \bar f(t) := \left|\bar f^n\left(t, 0,0\right)-\bar f^m\left(t, 0, 0\right)\right|, \ t \in[0, T].
\end{aligned}
$$
With the aforementioned notations, following a similar justification as in (\ref{gamma01}), we derive
\begin{equation}
\label{Gamma02}
\tilde\Gamma_t^{n,m}:= \exp \left\{\zeta_\theta^0 e^{\tilde A^{(n,m)}_s} \tilde Y_s^{(n,m)}\right\} \leq \mathbb E_t[\widetilde D_T^m\widetilde \eta^{n,m}]+\widetilde J_t^{n,m,2}\le \mathbb E_t[\widetilde J^{n,m,1}_T]+\widetilde J_t^{n,m,2}=\widetilde J_t^{n, m,1}+\widetilde J_t^{n, m,2}.
\end{equation}
Indeed, by setting $\{f^n\}=\{\bar f^n\}$  in the proof of (\ref{gamma01}) and considering the fact that $f$ satisfies the growth condition (\ref{growth3}), the proof of (\ref{gamma01}) unfolds step by step when we substitute the data $(C_0,\ 3\alpha,\ 3\beta,\ \eta)$ with $(1,\ 3\alpha,\ \beta,\ \widetilde\eta^{n,m})$.

Then, similar to (\ref{ess05}),
\begin{equation}
\label{ess0205}
|Y_t^m-Y_t^n| \leq(1-\theta)(\left|Y_t^m\right|+\left|Y_t^n\right|)+\frac{1-\theta}{\lambda} \ln \left(\sum_{i=1}^2\widetilde J_t^{m,n,i}\right), \quad \mathbb P \text {-a.s. }
\end{equation}

From (\ref{ess0205}), we utilize Doob’s martingale inequality and H\"older’s inequality to derive, for any $\delta>0$,
\begin{equation}
\label{prob2}
\begin{aligned}
& \mathbb P\left(\sup _{t \in[0, T]}\left|Y_t^n-Y_t^m\right| \geq \delta\right) \leq \mathbb P\left((1-\theta)\left(Y_*^m+Y_*^n\right) \geq \delta / 2\right)+\mathbb P\left(\frac{1-\theta}{\lambda} \ln \left(\sum_{i=1}^2\widetilde J_*^{n,m,i}\right) \geq \delta / 2\right) \\
& \leq 2 \frac{1-\theta}{\delta}\mathbb E\left[Y_*^m+Y_*^n\right]+\sum_{i=1}^2\mathbb P\left(\widetilde J_*^{n,m,i} \geq \frac{1}{2} e^{\frac{\delta \lambda}{2(1-\theta)}}\right) \\
&\leq \frac{1-\theta}{\delta \lambda} \mathbb E\left[e^{2 \lambda Y_*^m}+e^{2 \lambda Y_*^n}\right]+2 e^{\frac{-\delta \lambda}{2(1-\theta)}}\sum_{i=1}^2 \mathbb E\left[\widetilde J_*^{n,m,i}\right] \\
& \leq 2 \frac{1-\theta}{\delta \lambda} {\Xi}(2\lambda,\ 3\alpha, \ \beta)+4 e^{\frac{-\delta \lambda}{2(1-\theta)}}\left(\mathbb E\left[\left(\widetilde D^m_T+\widetilde D^n_T\right)^2(\widetilde\eta^{n,m})^2 \right]\right)^{1/2}\\
\
&\quad\quad+4e^{\frac{-\delta \lambda}{2(1-\theta)}}\left(\mathbb E\left[\left(\zeta_\theta^0 e^{\beta \|A_T\|} (\tilde D^m_T+\tilde D^n_T) \tilde \Upsilon_{n,m} \int_0^T\left|\tilde\Delta_{n,m}\bar f(s)\right| d A_s\right)^2\right]\right)^{1/2}\\
&\le 2 \frac{1-\theta}{\delta \lambda} {\Xi}(2\lambda,\ 3\alpha, \ \beta)+4 e^{\frac{-\delta \lambda}{2(1-\theta)}}\left(\mathbb E\left[\left(\widetilde D^m_T+\widetilde D^n_T\right)^4\right]\right)^{1/4}\left(\mathbb E[(\widetilde\eta^{n,m})^4] \right)^{1/4}\\
&\quad\quad+4e^{\frac{-\delta \lambda}{2(1-\theta)}}\left(\mathbb E\left[\left(\zeta_\theta^0 e^{\beta \|A_T\|} (\tilde D^m_T+\tilde D^n_T) \tilde \Upsilon_{n,m}\right)^4\right]\right)^{1/4}\left(\mathbb E\left[\left( \int_0^T\left|\tilde\Delta_{n,m}\bar f(s)\right| d A_s\right)^4\right]\right)^{1/4}\\
&\le 2 \frac{1-\theta}{\delta \lambda} {\Xi}(2\lambda,\ 3\alpha, \ \beta)+4 e^{\frac{-\delta \lambda}{2(1-\theta)}}\left(\sup_{m,n}\left(\mathbb E\left[\left(\widetilde D^m_T+\widetilde D^n_T\right)^4\right]\right)^{1/4}\right)\left(\mathbb E[(\widetilde\eta^{n,m})^4] \right)^{1/4}\\
&\quad\quad+4e^{\frac{-\delta \lambda}{2(1-\theta)}}\left(\sup_{m,n}\left(\mathbb E\left[\left(\zeta_\theta^0 e^{\beta \|A_T\|} (\tilde D^m_T+\tilde D^n_T) \tilde \Upsilon_{n,m}\right)^4\right]\right)^{1/4}\right)\left(\mathbb E\left[\left( \int_0^T\left|\tilde\Delta_{n,m}\bar f(s)\right| d A_s\right)^4\right]\right)^{1/4}.
\end{aligned}
\end{equation}
Note that for any $t\in[0,T]$,
$$
\lim_{n,m\to\infty}\left|\tilde\Delta_{n,m}\bar f(t)\right|=0,\quad  \left|\tilde\Delta_{n,m}\bar f(t)\right|\le 6\alpha_t.
$$
Therefore, considering $\lim_{n,m \rightarrow \infty}\widetilde\eta^{n,m}=\exp\{\lambda e^{2\beta\|A_T\|_\infty}|\xi|\}$ and the dominated convergence theorem, we first let $n,m \rightarrow \infty$ in (\ref{prob2}) and then let $\theta\rightarrow 1$, resulting in
 $\lim _{N\rightarrow \infty}\sup_{m,n\ge N} \mathbb P\left(\sup _{t \in[0, T]}\left|Y_t^n-Y_t^m\right| \geq \delta\right)=0$, which implies that the sequence $Y^n$ uniformly converges in probability (u.c.p.).
Then, there exists a process $Y^0$, up to a subsequence, such that
\begin{equation}
\label{ycvgunif}
\lim_{n\to\infty}\sup_{t\in[0,T]}|Y^n_t- Y^0_t|=0.
\end{equation}
Then, $ Y^0$ is c\` adl\`ag. Making use of Lemma \ref{submartingle property} and Fatou’s lemma, $ Y^0\in\mathcal E$. Furthermore, analogous to (\ref{cvg_Y}), we deduce that
\begin{equation}
\label{cvg_Y2}
\lim _{n \rightarrow \infty}\mathbb E\left[\sup _{t \in[0, T]}\left|Y_t^n-Y_t^0\right|^q\right]=0, \quad \forall q \in[1, \infty).
\end{equation}

\textbf{Step 2: Construction of "presolution" $U^0$.}
Compared to Step 2 of the proof in the bounded terminal case, thanks to (\ref{cvg_Y2}), the current proof is more straightforward.

For any $m,\ n \in \mathbb{N}$, applying It\^o’s formula to the process $\left|Y^n-Y^m\right|^2$, we can infer that
$$
\begin{aligned}
& \int_0^T \int_E\left|U_s^n(e)-U_s^m(e)\right|^2 \phi_s(d e) d A_s \\
& =\left|\xi^n-\xi^m\right|^2-\left|Y_0^n-Y_0^m\right|^2+2 \int_0^T\left(Y_{s^{-}}^n-Y_{s^{-}}^m\right)\left(f\left(s, Y_s^n, U_s^n\right)-f\left(s, Y_s^m, U_s^m\right)\right) d A_s \\
& \quad-2 \int_0^T \int_E\left(\left(Y_{s^{-}}^n-Y_{s^{-}}^m\right)\left(U_s^n-U_s^m\right)+\left|U_s^n-U_s^m\right|^2\right) q(d s d e) \\
& \leq 2 \sup _{t \in[0, T]}\left|Y_t^n-Y_t^m\right|\left(2 \int_0^T \alpha_s d A_s+\beta\left\|A_T\right\|\left(Y_*^n+Y_*^m\right)+\frac{1}{\lambda} \int_0^T\left(j_\lambda\left(U_s^m\right)+j_\lambda\left(U_s^n\right)+j_\lambda\left(-U_s^m\right)+j_\lambda\left(-U_s^n\right)\right) d A_s\right) \\
& \quad+\sup _{t \in[0, T]}\left|Y_t^n-Y_t^m\right|^2-2 \int_0^T \int_E\left(\left(Y_{s^{-}}^n-Y_{s^{-}}^m\right)\left(U_s^n-U_s^m\right)+\left|U_s^n-U_s^m\right|^2\right) q(d s d e).
\end{aligned}
$$
The last term is a martingale owing to the integrability condition of $Y^n, Y^m$ and $U^n, U^m$. Then, similar to (\ref{U-cauchy}), by taking expectations and using H\"older’s inequality, we deduce, from (\ref{sub mart |yn|2}) and (\ref{e|u|}) in the proof of Proposition \ref{priori estimate on U and K}, the existence of a constant $c>0$, which varies, such that
$$
\begin{aligned}
\mathbb{E}[ & \left.\left(\int_0^T \int_E\left|U_s^n-U_s^m\right|^2 \phi_s(e) d A_s\right)\right] \\
\leq & \mathbb{E}\left[\sup _{t \in[0, T]}\left|Y_t^n-Y_t^m\right|^2\right]+c\left\{\mathbb{E}\left[\sup _{t \in[0, T]}\left|Y_t^n-Y_t^m\right|^2\right]\right\}^{\frac{1}{2}} \\
& \times\left\{\sup _{m \in \mathbb{N}} \mathbb{E}\left[e^{2 \lambda Y_*^m}+\int_0^T \int_E\left(e^{\lambda U_s^m}-1\right)^2 \phi_s(e) d A_s+\int_0^T \int_E\left(e^{-\lambda U_s^m}-1\right)^2 \phi_s(e) d A_s\right]\right\}^{\frac{1}{2}} \\
\leq & \mathbb{E}\left[\sup _{t \in[0, T]}\left|Y_t^n-Y_t^m\right|^2\right]+c\left\{\mathbb{E}\left[\sup _{t \in[0, T]}\left|Y_t^n-Y_t^m\right|^2\right]\right\}^{\frac{1}{2}}.
\end{aligned}
$$
Hence, it follows that
$$
\lim _{N \rightarrow \infty} \sup _{m, n \geq N}\mathbb E\left[\left(\int_0^T \int_E\left|U_s^n(e)-U_s^m(e)\right|^2 \phi_s(d e) d A_s\right)\right]=0.
$$
Thus, $\left\{U^n\right\}$ is a Cauchy sequence in $H_{\nu}^{2,2}$, which implies that there exists a $U^0 \in H_{\nu}^{2,2}$ such that
\begin{equation}
\label{cvg U2}
\lim _{n \rightarrow \infty} \mathbb E\left[\left(\int_0^T \int_E\left|U_s^n(e)-U_s^0(e)\right|^2 \phi_s(d e) d A_s\right)\right]=0.
\end{equation}
Thus, up to a subsequence,
$$
\lim _{n \rightarrow \infty}\left(\int_0^T \int_E\left|U_s^n(e)-U_s^0(e)\right|^2 \phi_s(d e) d A_s\right)=0,\ \mathbb{P}\text{-a.s}.
$$
Then, on the set $\left\{(t, \omega) \in[0, T] \times \Omega: d A_t(\omega) \neq 0\right\}$, it holds that
\begin{equation}
\label{ucvgphi}
\lim _{n \rightarrow \infty}\left(\int_E\left|U_t^n(e)-U_t^0(e)\right|^2 \phi_t(d e)\right)=0.
\end{equation}
Moreover, with the help of Proposition \ref{priori estimate on U and K} and Fatou’s lemma, similar to (\ref{u0}) and (\ref{U0e}), $U^0 \in H_{\nu}^{2, p}$, for each $p\ge 1$. Moreover, $U^0$ satisfies the martingale property in Definition \ref{def solution}.

\textbf{Step 3: Verification of the solution $(Y^0,U^0)$. }

We have constructed a "presolution" $(Y^0,\ U^0)\in \mathcal E\times \cap_{p\ge 1}H_{\nu}^{2,p}$. It remains to be verified whether it is a true solution.
Since $f$ is continuous with respect to $y$ and $u$, referring to (\ref{ycvgunif}), (\ref{ucvgphi}) and the definition of $\{\bar f^n\}$, for $t \in [0,T]$, on $\{(t,\omega), dA_t(\omega)\not=0\}$,
$$
\lim _{n \rightarrow \infty} \bar f^n\left(t, \omega, Y_t^n, U_t^n\right)=f\left(t, \omega,  Y_t^0, U_t^0\right).
$$
Thus, by the dominated convergence theorem,
\begin{equation}
\label{cvgf2}
\lim_{n\to\infty}\mathbb{E}\left[\int_0^{T}\left|\bar f^n\left(t, Y_t^{n},U_t^{n}\right)-f\left(t, Y^0_t, U^0_t\right)\right| d A_t\right]=0.
\end{equation}
Then, with the help of (\ref{cvg_Y2}), (\ref{cvgf2}), (\ref{cvg U2}), and the Burkholder–Davis–Gundy inequality, we obtain
\begin{equation}
\begin{aligned}
&\mathbb E\left[\left|Y^0_t-(\xi+\int_t^T f\left(s,Y^0_s, U^0_s\right) d A_s -\int_t^T \int_E U^0_s(e) q(d s, d e))\right|\right]\\
&\quad\le\lim_{n\to\infty} \left(\mathbb E\left[\left|Y^0_t-Y^n_t\right|^2\right]\right)^{1/2}+\lim_{n\to\infty}\mathbb E\left[\left|\left(\int_t^T f\left(s, Y^0_s, U^0_s\right) d A_s-\int_t^T \int_E U^0_s(e) q(d s, d e)\right)\right.\right.\\
&\quad\quad\left.\left.-\left(\int_t^T \bar f^n\left(s,Y^n_s, U^n_s\right) d A_s -\int_t^T \int_E U^n_s(e) q(d s, d e)\right)\right|\right]\\
&\quad= 0.
\end{aligned}
\end{equation}
The proof is complete.

\end{proof}

\begin{remark}
Under unbounded terminal conditions, the existence of exponential growth BSDEs holds even without the assumption of convexity, as noted in \cite{karoui2016quadratic}. The authors use a different method to approximate the solution from solutions of a set of Lipschitz BSDEs. However, \cite{karoui2016quadratic} does not provide a uniqueness result. Recently, a uniqueness result for a specific type of quadratic–exponential BSDEs, arising from a robust utility maximization problem under a jump setting, was established in Kaaka\"i, Matoussi and Tamtalini\cite{kaakai2022utility}. Compared to their work, in this study, by imposing the additional assumption of convexity/concavity on the driver $f$, we remove the absolutely continuous assumption on the compensator and obtain a uniqueness result for a broader class of BSDEs using the $\theta-$method.
\end{remark}

\begin{remark}
\label{remark general}
Consider the following BSDE
\begin{equation}
\label{BSDE with B}
  Y_t=\xi+\int_t^T f\left(s,Y_s, U_s\right) d A_s +\int_t^T g\left(s,Y_s, Z_s\right) d s-\int_t^T Z_s d W_s-\int_t^T \int_E U_s(e) q(d sd e),
\end{equation}
where $W\in \mathbb R^d$ denotes a $d$-dimensional standard Brownian motion independent of the MPP. Assume that (H1)–(H4) hold, and $g(t,y,z)$ satisfies the additional condition:\\

\noindent(i) for all $t \in[0, T]$, for all $y \in \mathbb{R}, z \longmapsto g(t, y, z)$ is convex or concave;\\

\noindent(ii) for all $(t, z) \in[0, T] \times \mathbb{R}^d$,
$$
\forall\left(y, y^{\prime}\right) \in \mathbb{R}^2, \quad\left|g(t, y, z)-g\left(t, y^{\prime}, z\right)\right| \leq \beta\left|y-y^{\prime}\right|;
$$

\noindent(iii) $g$ fulfills the following growth condition:
$$
\forall(t, y, z) \in[0, T] \times \mathbb{R} \times \mathbb{R}^d, \quad|g(t, y, z)| \leq \alpha_t+\beta|y|+\frac{\gamma}{2}|z|^2.
$$
For every $p\ge 1$, there exists a unique solution $(Y,Z,U)\in \mathcal E\times \mathbb H^p\times H_{\nu}^{2,p}$ to (\ref{BSDE with B}).
By combining the comparison theorem from Briand and Hu \cite[Theorem 5]{briand2008quadratic} with Theorem \ref{Comparison Theorem for BSDE}, we can also derive a comparison theorem for BSDE (\ref{BSDE with B}), which inherently implies uniqueness. The existence result is inherited from Theorem \ref{thm Quadratic unbounded} and Briand and Hu \cite[Theorem 2]{briand2006bsde}.
\end{remark}

\begin{remark}
\label{remark_RBSDEJ}
Our techniques can be readily adapted to BSDEs of the following form, featuring a general generator $f(t,y,z,u)$ satisfying the quadratic–exponential growth condition, where $C_t$ is a predictable, continuous non-decreasing process starting from zero.
\begin{equation}
\label{BSDEJ with B}
Y_t=\xi+\int_t^T f\left(s,Y_s, Z_s, U_s\right) d C_s -\int_t^T Z_s d B_s-\int_t^T \int_E U_s(e) q(d s d e), \quad 0 \leq t \leq T, \quad \mathbb{P} \text {-a.s. }
\end{equation}
This form is more general than the BSDEJ settings described in \cite{kazi2015quadratic}.
Motivated by Theorem \ref{thm Quadratic unbounded}, we assume that the structural and growth conditions on $f$ are as follows, where the differential form is understood as integration over any measurable subset of $[0,T]$.
\hspace*{\fill}\\

\noindent(\textbf{H1}) The process A is continuous, with $\|A_T\|_\infty<\infty$.
\hspace*{\fill}\\

\noindent(\textbf{H2*}) For every $\omega \in \Omega,\ t \in[0, T],\ r\in \mathbb{R},\ z\in\mathbb R^d$, the mapping
$
f(\omega, t, r, z,\cdot):L^2(E,\mathcal{B}(E),\phi_t(\omega,dy))\rightarrow \mathbb{R}
$
satisfies:
for every $U \in {H_\nu^{2,2}}$,
$$
(\omega, t, r,z) \mapsto f\left(\omega, t, r,z, U_t(\omega, \cdot)\right)
$$
is Prog $\otimes \mathscr{B}(\mathbb{R})\otimes \mathscr{B}(\mathbb{R}^d)$-measurable.

\hspace*{\fill}\\
\hspace*{\fill}\\
\noindent(\textbf{H3*})

\textbf{(a) (Continuity condition)}
For every $\omega \in \Omega, t \in[0, T], y \in \mathbb{R}$, $z\in\mathbb R^d$ $u\in L^2(E,\mathcal{B}(E),\phi_t(\omega,dy))$, $(y,z, u) \longrightarrow f(t, y,z, u)$ is continuous.
\hspace*{\fill}\\
\hspace*{\fill}\\

\textbf{(b) (Lipschitz condition in $y$)}
There exists $\beta\geq 0$, such that for every $\omega \in \Omega,\ t \in[0, T],\ y, y^{\prime} \in \mathbb{R}$, $z\in\mathbb R^d$, \ $u\in L^2(E,\mathcal{B}(E),\phi_t(\omega,dy))$, we have
$$
\begin{aligned}
& \left|f(\omega, t, y,z, u(\cdot))-f\left(\omega, t, y^{\prime},z, u(\cdot)\right)\right| dC_t\leq \beta\left|y-y^{\prime}\right|(dA_t+dt).
\end{aligned}
$$

\hspace*{\fill}\\

\textbf{(c) (Quadratic–exponential growth condition)}
For all $t \in[0, T], \ (y,z, u) \in \mathbb{R} \times\mathbb R^d\times L^2(E,\mathcal{B}(E),\phi_t(\omega,dy)):\ \mathbb{P}$-a.s, there exists $\lambda>0$ such that
\begin{equation}
\label{fc condition}
\begin{aligned}
-\left(\alpha_t+\beta|y|+\frac{\gamma}{2}|z|^2 \right)dt+\left(-\alpha_t-\beta|y|-\frac{1}{\lambda} j_{\lambda}(t, -u)\right)dA_t\\
\le f(t, y, z,u)dC_t
\leq \left(\alpha_t+\beta|y|+\frac{1}{\lambda} j_{\lambda}(t, u)\right)dA_t+\left(\alpha_t+\beta|y|+\frac{\gamma}{2}|z|^2 \right)dt,
\end{aligned}
\end{equation}
where $\{\alpha_t\}_{0 \leq t \leq T}$ is a progressively measurable non-negative stochastic process.

\hspace*{\fill}\\
\hspace*{\fill}\\

\textbf{(d) (Integrability condition)}
We assume that
$$
\forall p>0, \quad \mathbb{E}\left[\exp \left(p\left(\left|\xi\right|+\int_0^T \alpha_s (d A_s+ds)\right)\right)\right]<+\infty.
$$
\hspace*{\fill}\\

\textbf{(e) (Convexity/Concavity condition)}
 For all $t \in[0, T]$ and $y \in \mathbb{R}$, $(z,u)\in\mathbb R^d\times L^2(E,\mathcal{B}(E),\phi_t(\omega,dy))\longmapsto f(t, y, z,u)$ is
jointly convex or concave.
\hspace*{\fill}\\

\noindent\textbf{(H4*) (Uniform linear bound condition)}
There exists a positive constant $C_0$ such that for each $t \in[0, T]$, $u\in L^2(E,\mathcal{B}(E),\phi_t(\omega,dy))$, if
 $f$ is convex (resp. concave) in $u$, then $f(t,0,,0,u)-f(t,0,0,0)\ge -C_0\|u\|_t$ (resp. $f(t,0,0,u)-f(t,0,0,0)\le C_0\|u\|_t$).
\hspace*{\fill}\\
\hspace*{\fill}\\
Under the abovementioned assumptions, the well-posedness of (\ref{BSDEJ with B}) holds, following from a parallel proof of Theorem \ref{thm Quadratic unbounded}, with assistance from Briand and Hu \cite[Theorem 5]{briand2008quadratic} in conjunction with \cite[Theorem 2]{briand2006bsde}. Specifically, for each $p\ge 1$, there exists a unique solution $(Y,Z,U)\in \mathcal E\times \mathbb H^p\times H_{\nu}^{2,p}$ to (\ref{BSDEJ with B}).

In comparison with \cite{kazi2015quadratic}, we do not assume a bounded terminal condition, nor do we impose the Fr\'echet differentiability condition or the local Lipschitz condition on the components $(z,u)$ of the generator $f$. Moreover, the linear bound condition (H4*) is weaker than the $A_{\gamma}$ condition utilized therein. Additionally, the growth condition (H3*)(c) is weaker, and it is unnecessary to subtract $f(t,0,0,0)$. Thus, our results apply under more general circumstances.
\end{remark}

\section{Mean-reflected BSDE}
\label{section mean reflected}

We further consider the BSDE (\ref{eq standard}) with the mean-reflected condition given by
\begin{equation}
\label{eq standard}
\left\{\begin{array}{l}
Y_t=\xi+\int_t^T f\left(s, Y_s, U_s\right) d A_s-\int_t^T \int_E U_s(e)q(dsde)+K_T-K_t, \quad 0 \leq t \leq T, \\
\mathbb{E}\left[\ell\left(t, Y_t\right)\right] \geq 0, \quad \forall t \in[0, T] \text { and } \int_0^T \mathbb{E}\left[\ell\left(t, Y_{t^-}\right)\right] d K_t=0.
\end{array}\right.
\end{equation}
Additional assumptions on the running loss function $\ell$ are also required.\\
\textbf{(H5)} $\ell:\Omega \times [0, T] \times \mathbb{R} \rightarrow \mathbb{R}$ satisfies the following properties:

1. $(t, y) \rightarrow \ell(t, y)$ is continuous.

2. $\forall t \in[0, T], y \rightarrow \ell(t, y)$ is strictly increasing,

3. $\forall t \in[0, T], \mathbb{E}\left[\lim_{y\uparrow \infty} \ell(t, y)\right]>0$,

4. $\forall t \in[0, T], \forall y \in \mathbb{R},|\ell(t, y)| \leq C(1+|y|)$ for some constant $C \geq 0$.\\
\textbf{(H6)} There exist two constants $\overline\kappa>\underline\kappa>0$ such that for each $t\in[0, T]$ and $y_1, y_2 \in \mathbb{R}$,
$$
\underline{\kappa}\left|y_1-y_2\right| \leq\left|\ell\left(t, y_1\right)-\ell\left(t, y_2\right)\right| \leq \overline\kappa\left|y_1-y_2\right|.
$$

To investigate mean-reflected BSDEs, we define the following map $L_t: \mathcal E\rightarrow \mathbb{R}$ for each $t \in[0, T]:$
$$
L_t(\eta)=\inf \{x \geq 0: \mathbb{E}[\ell(t, x+\eta)] \geq 0\}, \quad \forall \eta \in \mathcal E.
$$
When assumption (H5) is satisfied, the operator $X \mapsto L_t(X)$ is well-defined, similar to \cite{briand2018bsdes}.
\begin{remark}
\label{remark1}
Moreover, if assumption (H6) is also fulfilled, then for each $t \in[0, T]$, $\kappa:=\overline\kappa/\underline\kappa>1$,
\begin{equation}
\label{eqL}
\left|L_t\left(\eta^1\right)-L_t\left(\eta^2\right)\right| \leq \kappa \mathbb{E}\left[\left|\eta^1-\eta^2\right|\right], \quad \forall \eta^1, \eta^2 \in \mathcal E.
\end{equation}
\end{remark}

Moreover, in order to first construct local solutions on small time intervals and then stitch them to obtain the global solution on $[0,T]$, we need the following further assumption on the process $A$ in BSDE(\ref{BSDE}).
\hspace*{\fill}\\

\noindent(\textbf{H1'}) The process $A$ is continuous with $\|A_s-A_t\|_\infty<|\rho(|s-t|)$, for any $s,t\in[0,T]$, where $\rho(\cdot)$ is a deterministic continuous increasing function with $\rho(0)=0$ and $\rho(T)<\infty$.
\hspace*{\fill}\\
\hspace*{\fill}\\
It is obvious that (H1') implies (H1), i.e.,  $\|A_T\|_\infty<\infty$.

The main result of this section is as follows.
\begin{thm}
\label{thm1}
  If assumptions (H1') and (H2)–(H6) are satisfied, then the BSDE (\ref{eq standard}) with mean reflection possesses a unique deterministic flat solution $(Y, U, K)$ in the space $ \mathcal E\times H_{\nu}^{2,p} \times \mathcal{A}_D$, for every $p\ge 1$.
\end{thm}

\subsection{Uniqueness of mean-reflected BSDEs}

We first prove the uniqueness of Theorem \ref{thm1}. The following simple case plays an important role in the proof.

\noindent(\textbf{H3''})

\textbf{(a) (Continuity condition)}
For every $\omega \in \Omega, t \in[0, T]$, $u\in L^2(E,\mathcal{B}(E),\phi_t(\omega,dy))$, $(t, u) \longrightarrow f(t, u)$ is continuous.
\hspace*{\fill}\\
\hspace*{\fill}\\

\textbf{(b) (Growth condition)}
For all $t \in[0, T],\ u \in L^2(E,\mathcal{B}(E),\phi_t(\omega,dy)):\ \mathbb{P}$-a.s,
$$
-\frac{1}{\lambda} j_{\lambda}(t,- u)-\alpha_t \leq f(t, u) \leq \frac{1}{\lambda} j_{\lambda}(t, u)+\alpha_t.
$$
where $\{\alpha_t\}_{0 \leq t \leq T}$ is a non-negative progressively measurable stochastic process.
\hspace*{\fill}\\
\hspace*{\fill}\\

\textbf{(c) (Integrability condition)}
$$
\forall p>0, \quad \mathbb{E}\left[\exp \left(p\left(\left|\xi\right|+\int_0^T \alpha_s d A_s\right)\right)\right]<+\infty.
$$
\hspace*{\fill}\\

\textbf{(d) (Convexity/concavity condition)}
$\forall t\in[0, T]$, $u\in L^2(E,\mathcal{B}(E),\phi_t(\omega,dy))$, $u \rightarrow f(t, u)$ is convex or concave.
\hspace*{\fill}\\

\noindent\textbf{(H4') (Uniform linear bound condition)}
There exists a positive constant $C_0$ such that for each $t \in[0, T]$, $u\in L^2(E,\mathcal{B}(E),\phi_t(\omega,dy))$, if $f$ is convex (resp. concave) in $u$, then $f(t,u)-f(t,0)\ge -C_0\|u\|_t$ (resp. $f(t,u)-f(t,0)\le C_0\|u\|_t$).

\hspace*{\fill}\\

\begin{lemma}[A simple case]
\label{simple}
Assuming that assumptions (H1'), (H2), (H3''), (H4'), (H5), and (H6) are satisfied, the exponential growth mean-reflected BSDE (\ref{eq simple}), with a driver free of $Y$, $f(t,u)$, possesses a unique deterministic flat solution $(Y, U, K)$ in the space $ \mathcal{E}\times \cap_{p\ge 1}{H}_{\nu}^{2,p}\times \mathcal{A}_D$.
\begin{equation}
\label{eq simple}
\left\{\begin{array}{l}
Y_t=\xi+\int_t^T f\left(s, U_s\right) d A_s-\int_t^T \int_E U_s(e)q(dsde)+K_T-K_t, \quad 0 \leq t \leq T, \\
\mathbb{E}\left[\ell\left(t, Y_t\right)\right] \geq 0, \quad \forall t \in[0, T] \text { and } \int_0^T \mathbb{E}\left[\ell\left(t, Y_{t^-}\right)\right] d K_t=0.
\end{array}\right.
\end{equation}
\end{lemma}

\begin{proof}[Proof of Lemma \ref{simple}]
Consider the following BSDE:
\begin{equation}
\label{eq small fixed}
y_t=\xi+\int_t^T f(s,u_s) d A_s-\int_t^T \int_E u_s(e) q(d s d e),
\end{equation}
which is a special case of (\ref{BSDE}). Therefore, owing to Theorem \ref{thm Quadratic unbounded}, the BSDE (\ref{eq small fixed}) has a unique solution $(y, u) \in\mathcal E\times \cap_{p\ge 1}{H}_{\nu}^{2,p}$.

Thus, inspired by \cite{briand2018bsdes,Briand_2020}, we define
$$
k_t=\sup_{0\le s\le T}L_s(y_s)-\sup _{t\le s\le T}L_s(y_s).
$$
Inherited from the proof of Theorem 4.1 in \cite{2310.15203},
$(Y_t, U_t, K_t)=((y_t+k_T-k_t), u_t, k_t)\in \mathcal E\times \cap_{p\ge 1} {H}_{\nu}^{2,p} \times \mathcal{A}_D$ is the unique deterministic flat solution to the BSDE with mean reflection (\ref{eq simple}).

\end{proof}

A straightforward corollary from Lemma \ref{simple} reads.
\begin{corollary}
\label{cor_1}
Suppose the assumptions in Theorem \ref{thm1} hold, and $\hat Y$ belongs to the space $\mathcal{E}$. Then, the exponential growth mean-reflected BSDE (\ref{eq standard}), with the driver $f^{\hat Y}(t,u):=f(t,\hat Y,u)$, has a unique deterministic flat solution $(Y, U, K)$ in the space $ \mathcal{E}\times \cap_{p\ge 1}{H}_{\nu}^{2,p}\times \mathcal{A}_D$.
\end{corollary}

We provide a representation of the solution to (\ref{eq standard}), which proves essential in subsequent analyses.
\begin{lemma}[The representation of the solution]
\label{representation}
  Assuming that assumptions (H1') and (H2)–(H6) hold, let $(Y, U, K) \in \mathcal E\times\cap_{p\ge 1} H_{\nu}^{2,p}\times \mathcal A_D$ be a deterministic flat solution to the BSDE with mean reflection (\ref{eq standard}). Then, for each $t \in [0, T]$
$$
\left(Y_t, U_t, K_t \right)=\left(y_t+\sup _{t \leq s \leq T} L_s\left(y_s\right), u_t, \sup _{0 \leq s \leq T} L_s\left(y_s\right)-\sup _{t \leq s \leq T} L_s\left(y_s\right)\right),
$$
where $(y, u) \in \mathcal E\times\cap_{p\ge 1} H_{\nu}^{2,p}$ represents the solution to the following BSDE (\ref{eq samll}) with the driver $f\left(s, Y_s, u_s\right)$ over the time horizon $[0, T]$, and $Y \in \mathcal E$ is determined by the solution of (\ref{eq standard}).
\begin{equation}
{\label{eq samll}}
y_t=\xi+\int_t^T f\left(s, Y_s, u_s\right) d A_s-\int_t^T \int_E u_s(e) q(d s d e).
\end{equation}
\end{lemma}

\begin{proof}
First, given $Y\in\mathcal E$, it is obvious that $f\left(s, Y_s, u_s\right)$ satisfies assumption (H2). Consequently, owing to Theorem \ref{thm Quadratic unbounded}, (\ref{eq samll}) has a unique solution $(y, u) \in\mathcal E\times\cap_{p\ge 1} H_{\nu}^{2,p}$.

Define
$$
k_t=\sup_{0\le s\le T}L_s(y_s)-\sup _{t\le s\le T}L_s(y_s).
$$
By means of the proof of Lemma \ref{simple}, $(y_t+k_T-k_t,u_t,k_t)\in \mathcal E\times\cap_{p\ge 1} H_{\nu}^{2,p} \times \mathcal A_D$ is the unique deterministic flat solution to the following BSDE with mean reflection, with the driver $f\left(s, Y_s, u_s\right)$:
\begin{equation}{\label{tilde}}
\left\{\begin{array}{l}
\hat{Y}_t =\xi+\int_t^T f\left(s, Y_s, \hat U_s\right) d A_s
-\int_t^T \int_E \hat{U}_s(e) q(d s d e)+ \left(\hat{K}_T-\hat{K}_t\right), \quad \forall t \in[0, T] \text { a.s. };\\
\mathbb{E}\left[\ell\left(t, \hat{Y}_t\right)\right] \geq 0, \quad \forall t \in[0, T].
\end{array}\right.
\end{equation}

Notice that $(Y,U,K)$ is also a deterministic flat solution to (\ref{tilde}). By uniqueness, $(Y_t,U_t,K_t)=(y_t+k_T-k_t,u_t,k_t)$.

Therefore,
$$
\left(Y_t,U_t,K_t \right)=\left(y_t+\sup _{t \leq s \leq T} L_s\left(y_s\right), u_t, \sup _{0 \leq s \leq T} L_s\left(y_s\right)-\sup _{t \leq s \leq T} L_s\left(y_s\right)\right).
$$
\end{proof}

Next, we establish the uniqueness of the mean-reflected BSDE (\ref{eq standard}).
\begin{proof}[Proof of Uniqueness in Theorem \ref{thm1}]
For $i=1,2$, let $\left(Y^i, U^i, K^i\right)$ be a deterministic $\mathcal{E}\times \cap_{p\ge 1}{H}_{\nu}^{2,p} \times \mathcal{A}_D$-solution to the exponential growth mean-reflected BSDE. Following from the representation provided in Lemma \ref{representation}, we obtain
\begin{equation}
\label{L_relation}
Y_t^i:=y_t^i+\sup _{t \leq s \leq T} L_s\left(y_s^i\right), \quad \forall t \in[0, T],
\end{equation}
where $\left(y^i, z^i\right) \in \mathcal{E}\times \cap_{p\ge 1}{H}_{\nu}^{2, p}$ is the solution to the following BSDE:
\begin{equation}
\label{BSDE01}
y_t^i=\xi+\int_t^T f\left(s, Y_s^i, u_s^i\right) d A_s-\int_t^T \int_E u_s^iq(dsde).
\end{equation}
Without loss of generality, let us assume that $f(t, y,\cdot)$ is convex. For each $\theta \in (0,1)$, we denote
$$
\delta_\theta \ell=\frac{\ell^1-\theta\ell^2}{1-\theta}, \delta_\theta \widetilde{\ell}=\frac{\ell^2-\theta\ell^1}{1-\theta} \text { and } \delta_\theta \bar{\ell}:=\left|\delta_\theta \ell\right|+\left|\delta_\theta \widetilde{\ell}\right|
$$
for $\ell=Y,\ y$ and $u$. Then, the pair of processes $\left(\delta_\theta y, \delta_\theta u\right)$ satisfies the following BSDE:
$$
\delta_\theta y_t=\xi+\int_t^T\left(\delta_\theta f\left(s, \delta_\theta u_s\right)+\delta_\theta f_0(s)\right) d A_s-\int_t^T\int_E \delta_\theta u_s q(dsde),
$$
where the generator is given by
$$
\begin{aligned}
& \delta_\theta f_0(t)=\frac{1}{1-\theta}\left(f\left(t, Y_t^1, u_t^1\right)-f\left(t, Y_t^2, u_t^1\right)\right), \\
& \delta_\theta f(t, u)=\frac{1}{1-\theta}\left(f\left(t, Y_t^2,\theta u_t^2+(1-\theta)u\right)-\theta f(t, Y_t^2,u_t^2)\right).
\end{aligned}
$$
Recalling assumption (H3), we have
$$
\begin{aligned}
&| \delta_\theta f_0(t) |\leq \beta\left(\left|Y_t^2\right|+\left|\delta_\theta Y_t\right|\right), \\
& \delta_\theta f(t, u) \leq f\left(t, Y_t^2, u\right) \leq \alpha_t+\beta\left|Y_t^2\right|+\frac{1}{\lambda}j_\lambda(t, u).
\end{aligned}
$$
Set,
$$
\begin{aligned}
\chi & =\int_0^T \alpha_s d A_s+2 \beta A_T\left(\sup _{s \in[0, T]}\left|Y_s^1\right|+\sup _{s \in[0, T]}\left|Y_s^2\right|\right), \\
\tilde{\chi} & =\int_0^T \alpha_s d A_s+2 \beta A_T\left(\sup _{s \in[0, T]}\left|Y_s^1\right|+\sup _{s \in[0, T]}\left|Y_s^2\right|\right)+\sup _{s \in[0, T]}\left|y_s^1\right|+\sup _{s \in[0, T]}\left|y_s^2\right|.
\end{aligned}
$$
Applying assertion (ii) of Lemma \ref{a priori esti 2} to (\ref{BSDE01}), we deduce that for any $p \geq 1$,
$$
\exp \left\{p\lambda \left(\delta_\theta y_t\right)^{+}\right\} \leq \mathbb{E}_t\left[\exp \left\{p\lambda\left(|\xi|+\chi+\beta(A_T-A_t)\sup _{s \in[t, T]}\left|\delta_\theta Y_s\right|\right)\right\}\right].
$$
Similarly, we have
$$
\exp \left\{p\lambda\left(\delta_\theta \widetilde{y}_t\right)^{+}\right\} \leq \mathbb{E}_t\left[\exp \left\{p\lambda\left(|\xi|+\chi+\beta(A_T-A_t)\sup _{s \in[t, T]}\left|\delta_\theta \widetilde Y_s\right|\right)\right\}\right].
$$

Note the fact that
$$
\left(\delta_\theta y\right)^{-} \leq\left(\delta_\theta \widetilde{y}\right)^{+}+2\left|y^1\right| \text { and }\left(\delta_\theta \tilde{y}\right)^{-} \leq\left(\delta_\theta y\right)^{+}+2\left|y^2\right|,
$$
we have
$$
\begin{aligned}
\exp \left\{p\lambda \left|\delta_\theta y_t\right|\right\} \vee \exp \left\{p\lambda \left|\delta_\theta \tilde{y}_t\right|\right\} & \leq \exp \left\{p\lambda\left(\left(\delta_\theta y_t\right)^{+}+\left(\delta_\theta \tilde{y}_t\right)^{+}+2\left|y_t^1\right|+2\left|y_t^2\right|\right)\right\} \\
& \leq \mathbb{E}_t \left[ \exp \left\{p\lambda\left(|\xi|+\tilde{\chi}+\beta(A_T-A_t)\sup _{s \in[t, T]} \delta_\theta \bar{Y}_s\right)\right\}\right]^2.
\end{aligned}
$$
Applying Doob’s maximal inequality and H\"older’s inequality, we conclude that for each $p \geq 1$ and $t \in [0, T]$,
\begin{equation}
\label{eq_exp estimate}
\begin{aligned}
\mathbb{E}\left[\exp \left\{p\lambda \sup _{s \in[t, T]} \delta_\theta \bar{y}_s\right\}\right] & \leq \mathbb{E}\left[\exp \left\{p\lambda \sup _{s \in[t, T]}\left|\delta_\theta y_s\right|\right\} \exp \left\{p\lambda \sup _{s \in[t, T]}\left|\delta_\theta \tilde{y}_s\right|\right\}\right] \\
& \leq 4 \mathbb{E}\left[\exp \left\{4 p\lambda\left(|\xi|+\tilde{\chi}+\beta(A_T-A_t)\sup _{s \in[t, T]} \delta_\theta \bar{Y}_s\right)\right\}\right].
\end{aligned}
\end{equation}
Set $C_1:=\sup _{0 \leq s \leq T}\left|L_s(0)\right|+2 \kappa \sup _{s \in[0, T]} \mathbb{E}\left[\left|y_s^1\right|+\left|y_s^2\right|\right]$. Recalling (\ref{L_relation}) and assumption (H6), we obtain
$$
\left|\delta_\theta Y_t\right| \leq C_1+\left|\delta_\theta y_t\right|+\kappa \sup _{t \leq s \leq T} \mathbb{E}\left[\left|\delta_\theta y_s\right|\right] \text { and }\left|\delta_\theta \widetilde{Y}_t\right| \leq C_1+\left|\delta_\theta \widetilde{y}_t\right|+\kappa \sup _{t \leq s \leq T} \mathbb{E}\left[\left|\delta_\theta \widetilde{y}_s\right|\right], \forall t \in[0, T].
$$
Then, by definition of $\delta_\theta \bar{Y}_t$, we obtain
$$
\delta_\theta \bar{Y}_t=\left|\delta_\theta Y_t\right|+\left|\delta_\theta \widetilde{Y}_t\right| \le 2C_1+\delta_\theta \bar{y}_t+2\kappa \sup _{t \leq s \leq T} \mathbb{E}\left[\left|\delta_\theta \bar y_s\right|\right].
$$
Therefore, in conjunction with Jensen’s inequality, this implies that for each $p \geq 1$ and $t \in [0, T]$,
\begin{equation}
\label{jensen01}
\begin{aligned}
\mathbb{E}\left[\exp \left\{p\lambda \sup _{s \in[t, T]} \delta_\theta \bar{Y}_s\right\}\right] &\le e^{2 p\lambda C_1}\mathbb{E}\left[\exp \left\{p\lambda \sup _{s \in[t, T]} \delta_\theta \bar{y}_s\right\}\right]\exp \left\{2 \kappa p\lambda \sup _{s \in[t, T]}\sup _{u \in[s, T]}\mathbb E\left[ \delta_\theta \bar{y}_u\right]\right\}
\\
&\le e^{2 p\lambda C_1}\mathbb{E}\left[\exp \left\{p\lambda \sup _{s \in[t, T]} \delta_\theta \bar{y}_s\right\}\right]\exp \left\{2 \kappa p\lambda \sup _{s \in[t, T]}\mathbb E\left[ \delta_\theta \bar{y}_s\right]\right\}
\\
&\le e^{2 p\lambda C_1}\mathbb{E}\left[\exp \left\{p\lambda \sup _{s \in[t, T]} \delta_\theta \bar{y}_s\right\}\right]\exp \left\{2 \kappa p\lambda \mathbb E\left[\sup _{s \in[t, T]} \delta_\theta \bar{y}_s\right]\right\}
\\
&\leq e^{2 p\lambda C_1} \mathbb{E}\left[\exp \left\{p\lambda \sup _{s \in[t, T]} \delta_\theta \bar{y}_s\right\}\right] \mathbb{E}\left[\exp \left\{2 \kappa p\lambda \sup _{s \in[t, T]} \delta_\theta \bar{y}_s\right\}\right] {\quad\text { (Jensen’s inequality)}}\\
&\leq e^{2 p\lambda C_1} \mathbb{E}\left[\exp \left\{(2+4 \kappa) p\lambda \sup _{s \in[t, T]} \delta_\theta \bar{y}_s\right\}\right] \\
&\leq 4 \mathbb{E}\left[\exp \left\{(8+16 \kappa) p\lambda\left(|\xi|+\tilde{\chi}+C_1+\beta(A_T-A_t)\sup _{s \in[t, T]} \delta_\theta \bar{Y}_s\right)\right\}\right],
\end{aligned}
\end{equation}
where we used (\ref{eq_exp estimate}) in the last inequality.
In view of (H1'), we can select a constant $h \in (0, T]$ depending solely on $\beta$ and $\kappa$, such that $T=Nh$, $\max_{1\le i\le N}\{(16+32 \kappa) \beta \|A_{ih}-A_{(i-1)h}\|_\infty\}<1$. By H\"older’s inequality, we derive that for any $p \geq 1$,
$$
\begin{aligned}
& \mathbb{E}\left[\exp \left\{p\lambda \sup _{s \in[T-h, T]} \delta_\theta \bar{Y}_s\right\}\right] \\
& \leq 4 \left(\mathbb{E}\left[\exp \left\{(16+32 \kappa) p\lambda\left(|\xi|+\tilde{\chi}+C_1\right)\right\}\right]\right)^{\frac{1}{2}} \mathbb{E}\left[\exp \left\{(16+32 \kappa) \beta\|A_T-A_{T-h}\|_\infty p\lambda \sup _{s \in[T-h, T]} \delta_\theta \bar{Y}_s\right\}\right]^{1/2} \\
& \leq 4 \mathbb{E}\left[\exp \left\{(16+32 \kappa) p\lambda\left(|\xi|+\tilde{\chi}+C_1\right)\right\}\right] \mathbb{E}\left[\exp \left\{p\lambda \sup _{s \in[T-h, T]} \delta_\theta \bar{Y}_s\right\}\right]^{(8+16 \kappa) \beta\|A_T-A_{T-h}\|_\infty},
\end{aligned}
$$
which together with the fact that $(16+32 \kappa) \beta\|A_T-A_{T-h}\|_\infty<1$ implies that for any $p \geq 1$ and $\theta \in(0,1)$
$$
\mathbb{E}\left[\exp \left\{p\lambda \sup _{s \in[T-h, T]} \delta_\theta \bar{Y}_s\right\}\right] \leq \mathbb{E}\left[4 \exp \left\{(16+32 \kappa) p \lambda\left(|\xi|+\tilde{\chi}+C_1\right)\right\}\right]^{\frac{1}{1-(8+16 \kappa) \beta\|A_T-A_{T-h}\|_\infty}}<\infty.
$$
Note that $Y^1-Y^2=(1-\theta)\left(\delta_\theta Y-Y^2\right)$. It follows that
$$
\mathbb{E}\left[\sup _{t \in[T-h, T]}\left|Y_t^1-Y_t^2\right|\right] \leq(1-\theta)\left(\frac{1}{p\lambda} \sup _{\theta \in(0,1)} \mathbb{E}\left[\exp \left\{p\lambda\sup _{s \in[T-h, T]} \delta_\theta \bar{Y}_s\right\}\right]+\mathbb{E}\left[\sup _{t \in[0, T]}\left|Y_t^2\right|\right]\right).
$$
Letting $\theta \rightarrow 1$, we obtain $Y^1=Y^2$. The representation provided in Lemma \ref{representation} gives $\left(U^1, K^1\right)=\left(U^2, K^2\right)$ on $[T-h, T]$. The uniqueness of the solution across the entire interval stems from the uniqueness observed over each small time interval. Hence, the proof is concluded.
\end{proof}

\begin{remark}
\label{remcvx}
When $f$ is concave in $u$, we employ $\theta l^1-l^2$ and $\theta l^2-l^1$ in the definition of $\delta_\theta l$ and $\delta_\theta\tilde l$, respectively. The parallel proof holds. Hence, in the following discussion, unless explicitly stated otherwise, we always assume that $f$ is convex in $u$.
\end{remark}

\subsection{Existence of mean-reflected BSDEs}
Next, we proceed with the proof of the existence part of Theorem \ref{thm1}. Before demonstrating its existence, we outline some useful a priori estimates. Let us assume that $f$ is convex in $u$ without loss of generality.

According to Corollary \ref{cor_1}, we recursively define a sequence of stochastic processes $\left(Y^{(m)}\right)_{m=1}^{\infty}$ via the following exponential growth BSDE with mean reflection:
$$
\left\{\begin{array}{l}
Y_t^{(m)}=\xi+\int_t^T f\left(s, Y_s^{(m-1)}, U_s^{(m)}\right) d A_s-\int_t^T\int_E U_s^{(m)}(e)q(dsde) +K_T^{(m)}-K_t^{(m)}, \quad 0 \leq t \leq T, \\
\mathbb{E}\left[\ell\left(t, Y_t^{(m)}\right)\right] \geq 0, \quad \forall t \in[0, T] \text { and } \int_0^T \mathbb{E}\left[\ell\left(t, Y_{t^-}^{(m)}\right)\right] d K_t^{(m)}=0,
\end{array}\right.
$$
where $Y^{(0)} \equiv 0$. It is obvious that $\left(Y^{(m)}, U^{(m)}, K^{(m)}\right) \in \mathcal{E}\times \cap_{p\ge 1}{H}_{\nu}^{2,p} \times \mathcal{A}_D$.

\begin{lemma}
\label{lemma_uniform estimate}
 Assume that the conditions in Theorem \ref{thm1} are satisfied. Then, for any $p \geq 1$, we have
$$
\sup _{m \geq 0} \mathbb{E}\left[\exp \left\{p\lambda \sup _{s \in[0, T]}\left|Y_s^{(m)}\right|\right\}\right]<\infty.
$$
\end{lemma}
\begin{lemma}
\label{lemma_limit estimation}
Assume that all conditions of Theorem \ref{thm1} are satisfied. Then, for any $p \geq 1$, we have
$$
\Pi(p):=\sup _{\theta \in(0,1)} \lim _{m \rightarrow \infty} \sup _{q \geq 1} \mathbb{E}\left[\exp \left\{p\lambda \sup _{s \in[0, T]} \delta_\theta \bar{Y}_s^{(m, q)}\right\}\right]<\infty,
$$
where we use the following notations
$$
\delta_\theta Y^{(m, q)}=\frac{ Y^{(m+q)}-\theta Y^{(m)}}{1-\theta}, \delta_\theta \widetilde{Y}^{(m, q)}=\frac{ Y^{(m)}-\theta Y^{(m+q)}}{1-\theta} \text { and } \delta_\theta \bar{Y}:=\left|\delta_\theta Y^{(m, q)}\right|+\left|\delta_\theta \widetilde{Y}^{(m, q)}\right|.
$$
\end{lemma}

The proofs of Lemmas \ref{lemma_uniform estimate} and \ref{lemma_limit estimation} can be found in Appendix \ref{appendix B}.

We conclude this section by establishing the existence of the mean-reflected BSDE (\ref{eq standard}) as stated in Theorem \ref{thm1}.
\begin{proof}[Proof of existence in Theorem \ref{thm1}]
Note that for any integer $p \geq 1$ and for any $\theta \in(0,1)$,
$$
\begin{aligned}
\limsup _{m \rightarrow \infty} \sup _{q \geq 1} \mathbb{E}\left[\sup _{t \in[0, T]}\left|Y_t^{(m+q)}-Y_t^{(m)}\right|^p\right] \le& \limsup _{m \rightarrow \infty} \sup _{q \geq 1} \mathbb{E}\left[\sup _{t \in[0, T]}2^{p-1}\left(\left|Y_t^{(m+q)}-\theta Y_t^{(m)}\right|^p+(1-\theta)^p\left|Y_t^{(m)}\right|^p\right)\right] \\
\leq& 2^{p-1}(1-\theta)^p\left(\frac{\Pi(1) p!}{\lambda^p}+\sup _{m \geq 1} \mathbb{E}\left[\sup _{t \in[0, T]}\left|Y_t^{(m)}\right|^p\right]\right),
\end{aligned}
$$
where we use the fact that for $x>0$, $\frac{x^p}{p!}\le e^x$, and let $x=\frac{\lambda}{1-\theta}\sup_{s\in [0,T]}\left|Y_t^{(m+q)}-\theta Y_t^{(m)}\right|$
in the last inequality.
Letting $\theta$ go to $1$, thanks to Lemmas \ref{lemma_uniform estimate} and \ref{lemma_limit estimation}, it turns out that
$$
\lim _{m \rightarrow \infty} \sup _{q \geq 1} \mathbb{E}\left[\sup _{t \in[0, T]}\left|Y_t^{(m+q)}-Y_t^{(m)}\right|^p\right]=0, \ \forall p \geq 1.
$$

Therefore, there exists $Y$ such that
\begin{equation}
\label{lim Y}
\lim _{m \rightarrow \infty} \mathbb{E}\left[\sup _{t \in[0, T]}\left|Y_t^{(m)}-Y_t\right|^p\right]=0, \ \forall p \geq 1.
\end{equation}
In fact, $Y\in\mathcal E$ thanks to (\ref{Y_space}) and Fatou’s lemma.

Next, we consider the following BSDE,
$$
y_t=\xi+\int_{t}^Tf(s,Y_s,u_s)dA_s-\int_t^T\int_E u_s(e)q(dsde),
$$
which is uniquely solvable with a solution $(y, u) \in \mathscr{E} \times\cap_{p\ge 1} H_{\nu}^{2,p}$.

Define
$$
K_t=\sup _{0 \leq s \leq T} L_s\left(y_s\right)-\sup _{t \leq s \leq T} L_s\left(y_s\right),
$$
and
\begin{equation}
\label{tilde Y}
\tilde{Y}_t=y_t+\sup _{t \leq s \leq T} L_s\left(y_s\right).
\end{equation}
It follows from Lemma \ref{representation} that $(\tilde{Y}, u, K)$ is a deterministic flat solution to the following BSDE with mean reflection:
$$
\left\{\begin{array}{l}
\hat{Y}_t=\xi+\int_t^T f\left(s, Y_s, \hat{U}_s\right) d A_s-\int_t^T \int_E\hat{U}_sq(dsde)+\left(\hat{K}_T-\hat{K}_t\right); \\
\mathbb{E}\left[\ell\left(t, \hat{Y}_t\right)\right] \geq 0, \quad \forall t \in[0, T];\\
\int_0^T \mathbb{E}\left[\ell\left(t, \hat Y_{t^-}\right)\right] d \hat K_t=0.
\end{array}\right.
$$

It remains to show $\tilde {Y}=Y$, a.s., with $Y$ being the limit of $Y^{(m)}$ found in (\ref{lim Y}) and $\tilde{Y}$ defined as (\ref{tilde Y}). To obtain the desired equality, i.e., $Y=\tilde{Y}$, we claim that
\begin{equation}
\label{tilde Y to Y}
\mathbb{E}\left[\sup _{0 \leq s \leq T}\left|\tilde{Y}_s-Y_s\right|\right]=0.
\end{equation}
Indeed, for each $m \in \mathbb{N}$,
$$
\begin{aligned}
{\mathbb{E}} & {\left[\sup _{0 \leq s \leq T}\left|\tilde{Y}_s-Y_s\right|\right] } \\
& \leq {\mathbb{E}}\left[\sup _{0 \leq s \leq T}\left|\tilde{Y}_s-Y_s^{(m)}\right|\right]+{\mathbb{E}}\left[\sup _{0 \leq s \leq T}\left|Y_s^{(m)}-Y_s\right|\right] \\
& \leq {\mathbb{E}}\left[\sup _{0 \leq s \leq T}\left|y_s-y_s^{(m)}\right|\right]+{\mathbb{E}}\left[\sup _{0 \leq s \leq T}\left|L_s\left(y_s\right)-L_s\left(y_s^{(m)}\right)\right|\right]+{\mathbb{E}}\left[\sup _{0 \leq s \leq T}\left|Y_s^{(m)}-Y_s\right|\right] \\
& \leq(1+\kappa) \mathbb{E}\left[\sup _{0 \leq s \leq T}\left|y_s-y_s^{(m)}\right|\right]+{\mathbb{E}}\left[\sup _{0 \leq s \leq T}\left|Y_s^{(m)}-Y_s\right|\right].
\end{aligned}
$$
The last inequality follows from assumption (H6). Given this inequality and (\ref{lim Y}), establishing (\ref{tilde Y to Y}) necessitates demonstrating that
\begin{equation}
\label{y lim}
\lim _{m \rightarrow \infty} {\mathbb{E}}\left[\sup _{0 \leq s \leq T}\left|y_s-y_s^{(m)}\right|\right]=0,
\end{equation}
where $y_t^{(m)}$ is the solution to the following exponential growth BSDE
\begin{equation}
y_t^{(m)}=\xi+\int_t^T f\left(s, Y_s^{(m-1)}, u_s^{(m)}\right) d A_s-\int_t^T\int_E u_s^{(m)}q(dsde).
\end{equation}

For each $\theta \in(0,1)$, we similarly set
$$
\delta_\theta^{(m)} l:=\frac{l-\theta l^{(m)}}{1-\theta}, \quad \delta_\theta^{(m)} \tilde{l}:=\frac{l^{(m)}-\theta l}{1-\theta} \quad \text { and } \quad \delta_\theta^{(m)} \bar{l}=\left|\delta_\theta^{(m)} l\right|+\left|\delta_\theta^{(m)} \tilde{l}\right|,
$$
for $l=y,\ Y$ and $u$.
Consider the following BSDE,
$$
\delta_\theta y_t^{(m)}=\xi+\int_t^T\left(\delta_\theta f^{(m)}\left(s, \delta_\theta u_s^{(m)}\right)+\delta_\theta f_0^{(m)}(s)\right) d A_s-\int_t^T \int_E\delta_\theta u_s^{(m)}(e)q(dsde),
$$
where the generator is given by
$$
\begin{aligned}
& \delta_\theta f_0^{(m)}(t)=\frac{1}{1-\theta}\left(f\left(t, Y_t, u_t\right)-f\left(t, Y_t^{(m-1)}, u_t\right)\right), \\
& \delta_\theta f^{(m)}(t, u)=\frac{1}{1-\theta}\left(-\theta f\left(t, Y_t^{(m-1)}, u_t^{(m)}\right)+f\left(t, Y_t^{(m-1)},(1-\theta) u+\theta u_t^{(m)}\right)\right).
\end{aligned}
$$
From assumptions (H3)(c) and (H3)(e), we obtain
$$
\begin{aligned}
& \delta_\theta f_0^{(m)}(t) \leq \beta\left(\left|Y_t^{(m-1)}\right|+\left|\delta_\theta Y_t^{(m-1)}\right|\right), \\
& \delta_\theta f^{(m)}(t, u) \leq f\left(t, Y_t^{(m-1)}, u\right) \leq \alpha_t+\beta\left(\left|Y_t^{(m-1)}\right|\right)+\frac{1}{\lambda}j_\lambda(t,u).
\end{aligned}
$$
For any $m\geq 1$, denote
$$
\begin{aligned}
& \zeta^{(m)}=|\xi|+\int_0^T \alpha_s d A_s+\beta A_T\left(\sup _{s \in[0, T]}\left|Y_s^{(m-1)}\right|+\sup _{s \in[0, T]}\left|Y_s\right|\right), \\
& \chi^{(m)}=\int_0^T \alpha_s d A_s+2 \beta A_T\left(\sup _{s \in[0, T]}\left|Y_s\right|+\sup _{s \in[0, T]}\left|Y_s^{(m-1)}\right|\right).
\end{aligned}
$$
Assertion (ii) of Lemma \ref{submartingle property} yields, for any $p \geq 1$,
$$
\exp \left\{p\lambda\left(\delta_\theta y_t^{(m)}\right)^{+}\right\} \leq \mathbb{E}_t \exp \left\{
p\lambda\left(|\xi|+\chi^{(m)}+\beta(A_T-A_t)\sup _{s \in[t, T]}\left|\delta_\theta Y_s^{(m-1)}\right|\right)\right\},
$$
and in the same manner, it also holds that
$$
\exp \left\{p\lambda\left(\delta_\theta \widetilde{y}_t^{(m)}\right)^{+}\right\}\leq \mathbb{E}_t\exp \left\{p\lambda\left(|\xi|+\chi^{(m)}+\beta(A_T-A_t)\sup _{s \in[t, T]}\left|\delta_\theta \widetilde{Y}_s^{(m-1)}\right|\right)\right\}.
$$

Thanks to the fact that
$$
\left(\delta_\theta y^{(m)}\right)^{-} \leq\left(\delta_\theta \widetilde{y}^{(m)}\right)^{+}+2\left|y\right| \text { and }\left(\delta_\theta \widetilde{y}^{(m)}\right)^{-} \leq\left(\delta_\theta y^{(m)}\right)^{+}+2\left|y^{(m)}\right|,
$$
we derive, relying on H\"older’s inequality, that
$$
\begin{aligned}
& \exp \left\{p\lambda\left|\delta_\theta y_t^{(m)}\right|\right\} \vee \exp \left\{p\lambda\left|\delta_\theta \widetilde{y}_t^{(m)}\right|\right\} \\
& \leq \exp \left\{p\lambda\left(\left(\delta_\theta y_t^{(m)}\right)^{+}+\left(\delta_\theta \widetilde{y}_t^{(m)}\right)^{+}+2\left|y_t^{(m)}\right|+2\left|y_t\right|\right)\right\} \\
& \leq \mathbb{E}_t\left[\exp \left\{p\lambda\left(|\xi|+\chi^{(m)}+\beta(A_T-A_t)\sup _{s \in[t, T]} \delta_\theta \bar{Y}_s^{(m-1)}\right)\right\}\right]^2 \\
& \quad \times \exp \left\{2 p\lambda\left(\left|y_t^{(m)}\right|+\left|y_t\right|\right)\right\} \\
& \leq \mathbb{E}_t\left[\exp \left\{p\lambda\left(|\xi|+\chi^{(m)}+\beta(A_T-A_t)\left(\sup _{s \in[t, T]} \delta_\theta \bar{Y}_s^{(m-1)}\right)\right)\right\}\right]^2 \\
& \quad \times \mathbb{E}_t\left[\exp \left\{4 p \lambda\zeta^{(m)}\right\}\right].
\end{aligned}
$$
Making use of Doob’s maximal inequality and H\"older’s inequality, we obtain that for all $p>1$ and $t \in [0, T]$,
$$
\begin{aligned}
& \mathbb{E}\left[\exp \left\{p\lambda \sup _{s \in[0, T]} \delta_\theta \bar{y}_s^{(m)}\right\}\right] \\
& \leq 4 \mathbb{E}\left[\exp \left\{8 p\lambda\left(|\xi|+\chi^{(m)}+\beta A_T\left(\sup _{s \in[0, T]} \delta_\theta \bar{Y}_s^{(m-1)}\right)\right)\right\}\right]^{\frac{1}{2}} \mathbb{E}\left[\exp \left\{16 p \lambda\zeta^{(m)}\right\}\right]^{\frac{1}{2}}.
\end{aligned}
$$
Combining Lemmas \ref{lemma_uniform estimate} and \ref{lemma_limit estimation} and H\"older’s inequality, we obtain
$$
\limsup_{m\to\infty}\mathbb{E}\left[\exp \left\{p\lambda \sup _{s \in[0, T]} \delta_\theta \bar{y}_s^{(m)}\right\}\right]<A_p,
$$
where $A_p$ is a constant that depends on $p$ and is free of $\theta$.

Note that $y^{(m)}-y=(1-\theta)\left(\delta_\theta^{(m)} \tilde{y}-y\right)$. It follows that
$$
\mathbb{E}\left[\sup _{t \in[0, T]}\left|y_t-y_t^{(m)}\right|\right] \leq(1-\theta)\left(\frac{1}{p\lambda} \sup _{\theta \in(0,1)} {\mathbb{E}}\left[\exp \left\{p\lambda \sup _{s \in[0, T]} \delta_\theta^{(m)} \bar{y}_s\right\}\right]+ \mathbb{E}\left[\sup _{t \in[0, T]}\left|y_t\right|\right]\right).
$$
First, let $m \rightarrow \infty$, and then let $\theta \rightarrow 1$, it turns out that
$$
\lim _{m \rightarrow \infty} {\mathbb{E}}\left[\sup _{0 \leq s \leq T}\left|y_s-y_s^{(m)}\right|\right]=0.
$$
The proof is complete.
\end{proof}

\begin{remark}
The well-posedness of (\ref{eq standard}) can also be generalized to (\ref{eq general}) through a similar argument as in Remark \ref{remark general}.
\begin{equation}
\label{eq general}
\left\{\begin{array}{l}
Y_t=\xi+\int_t^T f\left(s, Y_s, U_s\right) d A_s+\int_t^Tg(s,Y_s,Z_s)ds-\int_t^TZ_sdW_s-\int_t^T \int_E U_s(e)q(dsde)+K_T-K_t, \quad 0 \leq t \leq T, \\
\mathbb{E}\left[\ell\left(t, Y_t\right)\right] \geq 0, \quad \forall t \in[0, T] \text { and } \int_0^T \mathbb{E}\left[\ell\left(t, Y_{t^-}\right)\right] d K_t=0.
\end{array}\right.
\end{equation}
For each $p\ge 1$, there exists a unique solution $(Y,Z,U,K)\in\mathcal E\times \mathbb H^p\times H_{\nu}^{2,p}\times \mathcal A_D$ to (\ref{eq general}).
\end{remark}

\begin{appendices}
\section{Proofs in section \ref{section non reflected}}
\label{appexdix A}
\begin{proof}[Proof of Lemma \ref{submartingle property}]
Note that $y$ is c\`adl\`ag and $A$ is continuous.

(i) Applying It\^o’s formula to $|y|$,
\begin{equation}
\begin{aligned}
d|y_t|&=sign(y_{t^-})dy_t+dL_t^y+\int_E\left(|y_{t^-}+u_t(e)|-|y_{t^-}|-sign(y_{t^-})u_t(e)\right)p(dtde)\\
&=-sign(y_{t^-})f(s,y_{t},u_t)dA_t+dL_t^y+\int_E\left(|y_t+u_t(e)|-|y_t|-sign(y_{t^-})u_t(e)\right)\phi_t(de)dA_t\\
&\quad+\int_E\left(|y_{t^-}+u_t(e)|-|y_{t^-}|\right)q(dtde).
\end{aligned}
\end{equation}
Define $G_t=e^{\beta A_t}\lambda |y_t|+\int_{0}^t\lambda e^{\beta A_s}\alpha_sdA_s$. Then applying It\^o’s formula to $G_t$, we obtain
$$
\begin{aligned}
dG_t&=e^{\beta A_t}\left[\lambda\beta |y_{t}|dA_t+\lambda d|y_t|+\lambda\alpha_tdA_t\right]\\
&=e^{\beta A_t}\lambda\left(\left(-sign(y_{t^-})f(t,y_{t},u_t)+\alpha_t+\beta |y_t|\right)dA_t+dL_t^y+\int_E\left(|y_t+u_t(e)|-|y_t|-sign(y_{t})u_t(e)\right)\phi_t(de)dA_t\right.\\
&\left.\quad+\int_E\left(|y_{t^-}+u_t(e)|-|y_{t^-}|\right)q(dtde)\right)\\
&\ge -j_\lambda(sign(y_{t^-})e^{\beta A_t}u_t)dA_t+e^{\beta A_t}\lambda\left(dL_t^y+\int_E\left(|y_{t}+u_t(e)|-|y_{t}|-sign(y_{t})u_t(e)\right)\phi_t(de)dA_t\right.\\
&\left.\quad+\int_E\left(|y_{t^-}+u_t(e)|-|y_{t^-}|\right)q(dtde)\right)\\
&\ge e^{\beta A_t}\lambda\int_E\left(|y_{t^-}+u_t(e)|-|y_{t^-}|\right)q(dtde)-j_{\lambda}\left(e^{\beta A_t}\left(|y_{t^-}+u_t(e)|-|y_{t^-}|\right)\right)dA_t,
\end{aligned}
$$
where, for the first inequality, we use the observation for any $k\ge 1$,
$$
j_\lambda(ku)\ge kj_{\lambda}(u).
$$
Moreover, the last inequality follows from the fact $|y+u|-|y| \geq sign(y) u$.

Finally, applying It\^o’s formula to $e^{G_t}$, and considering the integrability condition on $y$ along with (\ref{integrability_condition}), we deduce that $e^{G_t}$ is a submartingale. This implies
$$
\exp\left\{\lambda |y_t|\right\}\le\exp\left\{e^{\beta A_t}\lambda |y_t|\right\}\le \mathbb E_t\left[ \exp\left\{e^{\beta A_T}\lambda|\xi|+\int_t^T\lambda e^{\beta A_t}\alpha_s dA_s\right\}\right].
$$
Hence, for each $p\ge 1$, by Jensen’s inequality, it holds that
$$
\exp\left\{p\lambda |y_t|\right\}\le \mathbb E_t\left[ \exp\left\{pe^{\beta A_T}\lambda|\xi|+\int_t^Tp\lambda e^{\beta A_t}\alpha_s dA_s\right\}\right].
$$

(ii)
Similar to assertion (i), applying It\^o’s formula to $y^+$,
\begin{equation}
\begin{aligned}
dy_t^+&=1_{y_{t^-}>0}dy_t+\frac{1}{2}dL_t^y+\int_E\left((y_{t^-}+u_t(e))^+-y_{t^-}^+-1_{y_{t^-}>0}u_t(e)\right)p(dtde)\\
&=-1_{y_{t^-}>0}f(s,y_{t},u_t)dA_t+\frac{1}{2}dL_t^y+\int_E\left((y_{t}+u_t(e))^+-y_{t}^+-1_{y_{t}>0}u_t(e)\right)\phi_t(de)dA_t\\
&\quad+\int_E\left((y_{t^-}+u_t(e))^+-y_{t^-}^+\right)q(dtde).
\end{aligned}
\end{equation}
Define $G_t=e^{\beta A_t}\lambda y_t^++\int_{0}^t\lambda e^{\beta A_s}\alpha_sdA_s$, and apply It\^o’s formula to $e^{G_t}$. Then, by following the proof of assertion (i) from line to line, we conclude the proof of Lemma \ref{submartingle property}.
\end{proof}

\begin{proof}[Proof of Proposition \ref{priori estimate on U and K}]

The proof is inspired by \cite[Proposition 4.5]{karoui2016quadratic}. Let $\bar G_t=Y_t+\int_0^t\alpha_sdA_s+\int_0^t\beta|Y_s|dA_s$. We first claim that $e^{\lambda\bar G}$ is a positive local submartingale. Indeed, by applying It\^o’s formula to $e^{\lambda\bar G_t}$, we obtain
\begin{equation}
\label{ebarG}
\begin{aligned}
de^{\lambda\bar G_t}&=e^{\lambda\bar G_{t^-}}\left(\lambda d \bar G_t+\int_E\left(e^{\lambda U_t(e)}-\lambda U_t(e)-1\right)p(dtde)\right)\\
&=e^{\lambda\bar G_{t^-}}\left(\left[\lambda\beta |Y_{t}| dA_t+\lambda d Y_t+\lambda\alpha_tdA_t\right]+\int_E\left(e^{\lambda U_t(e)}-\lambda U_t(e)-1\right)p(dtde)\right)\\
&=e^{\lambda\bar G_{t^-}}\left(\left[\lambda\beta |Y_{t}| dA_t+\lambda \left(-f(t,Y_t,U_t)dA_t+\int_EU_t(e)q(dtde)\right)+\lambda\alpha_tdA_t\right]+\int_E\left(e^{\lambda U_t(e)}-\lambda U_t(e)-1\right)p(dtde)\right)\\
&=e^{\lambda\bar G_{t^-}}\left(\left[\lambda\beta |Y_{t}|-\lambda f(t,Y_t,U_t)+\lambda\alpha_t+j_1(\lambda U_t)\right]dA_t+\int_E\left(e^{\lambda U_t(e)}-1\right)q(dtde)\right)\\
&\ge e^{\lambda\bar G_{t^-}}\int_E\left(e^{\lambda U_t(e)}-1\right)q(dtde).
\end{aligned}
\end{equation}
We plug in the growth condition of $f$ in the last inequality.

Furthermore, from (\ref{ebarG}), we observe that $e^{\lambda\bar G}$ has the following decomposition:
\begin{equation}
\label{ebarGdecom}
e^{\lambda\bar G_t}=e^{\lambda\bar G_0}\mathcal E(\bar M_t)\exp(\bar A_t),
\end{equation}
where,
$$
\bar M_t=\int_0^t\int_E\left(e^{\lambda U_s(e)}-1\right)q(dsde),
$$
$\mathcal E(\bar M_t)$ is the Dol\'eans-Dade exponential of $\bar M$,
and $\bar A$ is a non-decreasing process with $\bar A_0=0$ in the form
$$
\bar A_t=\int_0^t\left(\lambda \left(-f(t,Y_s,U_s)+\alpha_s+\beta |Y_s|\right)+j_1(\lambda U_s(e))\right)dA_s.
$$
Then,
\begin{equation}
\label{compact bar G}
d e^{\lambda\bar G_t}=e^{\lambda\bar G_{t^-}}(d\bar A_t+d\bar M_t).
\end{equation}
Note that $\bar G\in\mathcal E$, then, considering (\ref{ebarGdecom}),
$\mathcal E(\bar M_t)$ is a true martingale.

Now, we estimate the quadratic variation of $\bar{M}$ :
$$
d[\bar{M}]_t=\int_E\left(e^{\lambda U_t(e)}-1\right)^2 \phi_t\left(de\right) d A_t+\int_E\left(e^{ \lambda U_t(e)}-1\right)^2 q(dtde).
$$
Obviously, by (\ref{compact bar G}),
$$
d[e^{\lambda\bar{G}}]_t=e^{2 \lambda\bar{G}_{t-}} d[\bar{M}]_t.
$$
We can also determine predictable quadratic variations through direct calculation,
$$
d\langle \bar M\rangle_t=\int_E\left(e^{ \lambda U_t(e)}-1\right)^2 \phi_t\left(de\right) d A_t,
$$
and
$$
d\langle e^{\lambda\bar G}\rangle_t= e^{2 \lambda\bar{G}_{t-}} d\langle\bar{M}\rangle_t.
$$

Then, for any stopping time $\sigma \leq T$, it holds that
\begin{equation}
\label{M-qv}
\langle\bar{M}\rangle_{T}-\langle\bar{M}\rangle_\sigma =\int_{\sigma}^T \frac{d\langle e^{\lambda\bar{G}}\rangle_t}{e^{2 \lambda\bar{G}_t-}}
\le \sup _{\sigma \leq t \leq T}\left(e^{-2 \lambda\bar{G}_t}\right)\left(\langle e^{\lambda\bar{G}}\rangle_{T}-\langle e^{\lambda\bar{G}}\rangle_\sigma\right).
\end{equation}
Next, we obtain an a priori estimate of $\langle e^{\lambda\bar{G}}\rangle_T-\langle e^{\lambda\bar{G}}\rangle_\sigma$ via It\^o’s formula,
\begin{equation}
\label{ito}
d e^{2\lambda\bar G_t}=2 e^{2\lambda \bar{G}_t-}\left(d \bar{M}_t+d \bar{A}_t\right)+d[e^{\lambda\bar{G}}]_t\ge 2 e^{2\lambda \bar{G}_t-}d \bar{M}_t+e^{2\lambda \bar{G}_t-}\int_E\left(e^{\lambda U_t(e)}-1\right)^2 q(dtde)+d\langle e^{\lambda\bar{G}}\rangle_t.
\end{equation}
Taking the conditional expectation on both sides of (\ref{ito}), we obtain
$$
\mathbb{E}\left[\langle e^{\lambda\bar{G}}\rangle_T-\langle e^{\lambda\bar{G}}\rangle_\sigma \mid \mathcal F_\sigma\right]\leqslant \mathbb{E}\left[e^{2\lambda \bar{G}_T}-e^{2 \lambda\bar{G}_\sigma} \mid \mathcal F_\sigma\right]
\leq \mathbb{E}\left[e^{2 \lambda\bar{G}_T} 1_{\sigma<T} \mid \mathcal F_\sigma\right],
$$
where we assume without loss of generality that $\bar M$ and $\int_0^{\cdot}e^{2\lambda \bar{G}_t-}\int_E\left(e^{\lambda U_t(e)}-1\right)^2 q(dtde)$ are martingales; otherwise, one can take advantage of standard localization and monotone convergence arguments.
Then, employing the Garcia–Neveu Lemma, see for example \cite[Lemma 4.3]{Barrieu_2013}, we find that for each $p\ge 1$,
$$
\mathbb{E}\left[\left(\langle e^{\lambda\bar{G}}\rangle_T\right)^p\right] \leqslant p^p \mathbb{E}\left[e^{2 p\lambda \bar{G}_T}\right].
$$
Then, by (\ref{M-qv}),
\begin{equation}
\label{bar M}
\begin{aligned}
\mathbb{E}\left[\left(\langle \bar{M}\rangle_T\right)^p\right] & \leqslant \mathbb{E}\left[\sup _{t \leqslant T}\left(e^{-2 p\lambda \bar{G}_t}\right)\left(\langle e^{\lambda\bar G}\rangle_T\right)^p\right] \\
& \leqslant\left(\mathbb{E}[\sup _{t \leq T} e^{-4 p\lambda\bar{G}_t}]\right)^{\frac{1}{2}}\left(\mathbb{E}\left[\left(\langle e^{\lambda\bar{G}}\rangle_T\right)^{2 p}\right]\right)^{\frac{1}{2}} \\
& \leqslant\left(\mathbb{E}\left[\sup _{t \leq T} e^{-4 p\lambda \bar{G} _t}\right]\right)^{\frac{1}{2}} \cdot(2 p)^p\left(\mathbb{E}\left[e^{4 p \lambda\bar{G}_T}\right]\right)^{\frac{1}{2}} \\
& \leqslant(2 p)^p \mathbb{E}\left[\sup _{t \leq T} e^{4 p\lambda\widetilde{G_t}}\right]\\
&\le C_p\mathbb E\left[e^{8p\lambda(1+\beta\|A_T\|_\infty)Y_*}\right],
\end{aligned}
\end{equation}
where
$
\widetilde{G_t}=|Y_t|+\int_0^t\alpha_sdA_s+\int_0^t\beta|Y_s|dA_s,
$
and $C_p$ is a positive constant that depends on $p$.

Similarly, let $\underline G_t=-Y_t+\int_0^t\alpha_sdA_s+\int_0^t\beta|Y_s|dA_s$. We can estimate the quadratic variation of
$$
\underline M_t=\int_0^t\int_E\left(e^{-\lambda U_s(e)}-1\right)q(dsde).
$$
Following the abovementioned proof step by step, we obtain
\begin{equation}
\label{underline M}
\mathbb{E}\left[\left(\langle \underline{M}\rangle_T\right)^p\right]
\le C_p\mathbb E\left[e^{8p\lambda(1+\beta\|A_T\|_\infty)Y_*}\right].
\end{equation}

Thus, combining (\ref{bar M}) and (\ref{underline M}), we obtain
\begin{equation}
\label{e|u|}
\mathbb E\left[\left(\int_0^T\int_E\left(e^{ \lambda |U_t(e)|}-1\right)^2 \phi_t\left(de\right) d A_t\right)^p\right]\le C_p\mathbb E\left[e^{8p\lambda (1+\beta\|A_T\|_\infty)Y_*}\right].
\end{equation}

Moreover, for each $q\ge 1$, repeating the procedure above for processes $e^{q\lambda\bar G}$ and $e^{q\lambda\underline G}$, we can obtain the following stronger estimates: for each $p,q\ge 1$,
\begin{equation}
\label{e|u|q}
\mathbb E\left[\left(\int_0^T\int_E\left(e^{q \lambda |U_t(e)|}-1\right)^2 \phi_t\left(de\right) d A_t\right)^p\right]\le  C_p\mathbb E\left[e^{8pq\lambda (1+\beta\|A_T\|_\infty)Y_*}\right].
\end{equation}
In fact, we only need to note that according to Definition \ref{def solution}, for each $p\ge 1$, the processes
$$
\int_0^\cdot\int_E\left(e^{p\lambda U_s(e)}-1\right)q(dsde)
$$
and
$$
\int_0^\cdot\int_E\left(e^{-p\lambda U_s(e)}-1\right)q(dsde)
$$
are local martingales and can be assumed as martingales without loss of generality in the proof of (\ref{e|u|q}).
The proof of (\ref{e|u|q}) is similar to that of (\ref{e|u|}) and thus is omitted.

Finally, by Jensen’s inequality,
$$
\begin{aligned}
\mathbb E\left[\left(\int_0^T\int_E|U_t(e)|^2 \phi_t\left(de\right) d A_t\right)^{p/2}\right]\le \left(\mathbb E\left[\left(\int_0^T\int_E|U_t(e)|^2 \phi_t\left(de\right) d A_t\right)^{p}\right]\right)^{1/2}\\
\le C_p\left(\mathbb E\left[\left(\int_0^T\int_E\left(e^{ \lambda |U_t(e)|}-1\right)^2 \phi_t\left(de\right) d A_t\right)^p\right]\right)^{1/2}\le C_p\mathbb E\left[e^{8p\lambda (1+\beta\|A_T\|_\infty)Y_*}\right],
\end{aligned}
$$
where we drop the power $\frac{1}{2}$ in the last inequality because the expectation term is greater than 1. The constant $C_p$ may vary.

\end{proof}

\section{Proofs in section \ref{section mean reflected}}
\label{appendix B}

\begin{proof}[Proof of Lemma \ref{lemma_uniform estimate}]
Referring to the representation in Lemma \ref{representation}, we find for any $m \geq 1$,
\begin{equation}
\label{representation_Ym}
Y_t^{(m)}:=y_t^{(m)}+\sup _{t \leq s \leq T} L_s\left(y_s^{(m)}\right), \quad \forall t \in[0, T],
\end{equation}
where $y_t^{(m)}$ is the solution to the following exponential growth BSDE:
\begin{equation}
y_t^{(m)}=\xi+\int_t^T f\left(s, Y_s^{(m-1)}, u_s^{(m)}\right) d A_s-\int_t^T\int_E u_s^{(m)}q(dsde).
\end{equation}

Applying assertion (i) of Corollary \ref{a priori esti 2} yields for any $t \in[0, T]$,
$$
\exp \left\{\lambda\left|y_t^{(m)}\right|\right\} \leq \mathbb{E}_t \exp \left\{\lambda\left(|\xi|+\int_0^T \alpha_s d A_s+\beta(A_T-A_t)\sup _{s \in[t, T]}\left|Y_s^{(m-1)}\right|\right)\right\}.
$$
Doob’s maximal inequality implies that for each $m \geq 1, \ p \geq 2$ and $t \in[0, T]$,
$$
\begin{aligned}
& \mathbb{E}\left[\exp \left\{p\lambda\sup _{s \in[t, T]}\left|y_s^{(m)}\right|\right\}\right] \\
& \quad \leq 4 \mathbb{E}\left[\exp \left\{p\lambda \left(|\xi|+\int_0^T \alpha_s d A_s+\beta(A_T-A_t)\sup _{s \in[t, T]}\left|Y_s^{(m-1)}\right|\right)\right\}\right].
\end{aligned}
$$
In view of (\ref{representation_Ym}), we obtain
$$
\left|Y_t^{(m)}\right| \leq\left|y_t^{(m)}\right|+\sup _{0 \leq s \leq T}\left|L_s(0)\right|+\kappa \sup _{t \leq s \leq T} \mathbb{E}\left[\left|y_s^{(m)}\right|\right].
$$
Set $\widetilde{\alpha}=\sup _{0 \leq s \leq T}\left|L_s(0)\right|+\int_0^T \alpha_s d A_s$. With the help of Jensen’s inequality, we derive from a similar justification as (\ref{jensen01}) that for any $m \geq 1,\ p \geq 2$ and $t \in[0, T]$,
$$
\begin{aligned}
 \mathbb{E}\left[\exp \left\{p\lambda \sup_{s \in[t, T]}\left|Y_s^{(m)}\right|\right\}\right] &\leq e^{p\lambda \sup _{0 \leq s \leq T}{\left|L_s(0)\right|}} \mathbb{E}\left[\exp \left\{p\lambda \sup _{s \in[t, T]}\left|y_s^{(m)}\right|\right\}\right] \mathbb{E}\left[\exp \left\{\kappa p\lambda \sup _{s \in[t, T]}\left|y_s^{(m)}\right|\right\}\right] \\
& \leq e^{p\lambda \sup _{0 \leq s \leq T}\left|L_s(0)\right|} \mathbb{E}\left[\exp \left\{(2+2 \kappa) p\lambda \sup _{s \in[t, T]}\left|y_s^{(m)}\right|\right\}\right] \\
& \leq 4 \mathbb{E}\left[\exp \left\{(2+2 \kappa) p\lambda\left(|\xi|+\widetilde{\alpha}+\beta(A_T-A_t) \sup _{s \in[t, T]}\left|Y_s^{(m-1)}\right|\right)\right\}\right].
\end{aligned}
$$
Considering (H1'), we can select a constant $h \in(0, T]$ depending solely on $\beta$ and $\kappa$ such that
\begin{equation}
\label{h_condition}
T=Nh,\ \max_{1\le i\le N}\left\{(32+64 \kappa) \beta\|A_{ih}-A_{(i-1)h}\|_\infty\right\}<1.
\end{equation}
Applying H\"older’s inequality, we deduce that for any $p \geq 2$,
\begin{equation}
\label{estimate-rho}
\begin{aligned}
& \mathbb{E}\left[\exp \left\{p\lambda \sup _{s \in[T-h, T]}\left|Y_s^{(m)}\right|\right\}\right] \\
& \leq 4 \left(\mathbb{E}[\exp \{(4+4 \kappa) p\lambda(|\xi|+\widetilde{\alpha})\}]\right)^{\frac{1}{2}}\left(\mathbb{E}\left[\exp \left\{(4+4 \kappa) \beta \|A_T-A_{T-h}\|_\infty p\lambda \sup _{s \in[T-h, T]}\left|Y_s^{(m-1)}\right|\right\}\right]\right)^{1/2} \\
& \leq 4 \left(\mathbb{E}[\exp \{(8+8 \kappa) p\lambda|\xi|\}]\right)^{\frac{1}{4}} \left(\mathbb{E}[\exp \{(8+8 \kappa) p \lambda\widetilde{\alpha}\}]\right)^{\frac{1}{4}} \left(\mathbb{E}\left[\exp \left\{p\lambda \sup _{s \in[T-h, T]}\left|Y_s^{(m-1)}\right|\right\}\right]\right)^{(2+2 \kappa) \beta \|A_T-A_{T-h}\|_\infty},
\end{aligned}
\end{equation}
where we use Jensen’s inequality in the last step. Define $\rho=\frac{1}{1-(2+2\kappa)\beta\max_{1\le i\le N}\|A_{ih}-A_{(i-1)h}\|_\infty}$.

If $N=1$, it follows from (\ref{estimate-rho}) that for each $p \geq 2$ and $m \geq 1$,
$$
\begin{aligned}
& \mathbb{E}\left[\exp \left\{p\lambda \sup _{s \in[0, T]}\left|Y_s^{(m)}\right|\right\}\right] \\
& \leq 4 \mathbb{E}[\exp \{(8+8 \kappa) p\lambda|\xi|\}]^{\frac{1}{4}} \mathbb{E}[\exp \{(8+8 \kappa) p\lambda \widetilde{\alpha}\}]^{\frac{1}{4}} \mathbb{E}\left[\exp \left\{p\lambda \sup _{s \in[0, T]}\left|Y_s^{(m-1)}\right|\right\}\right]^{(2+2 \kappa) \beta \|A_T-A_{T-h}\|_\infty}.
\end{aligned}
$$
Iterating the abovementioned procedure $m$ times yields
\begin{equation}
\label{1st iteration 1}
\mathbb{E}\left[\exp \left\{p\lambda \sup _{s \in[0, T]}\left|Y_s^{(m)}\right|\right\}\right] \leq 4^\rho \mathbb{E}[\exp \{(8+8 \kappa) p\lambda|\xi|\}]^{\frac{\rho}{4}} \mathbb{E}[\exp \{(8+8 \kappa) p\lambda \widetilde{\alpha}\}]^{\frac{\rho}{4}},
\end{equation}
which is uniformly bounded with respect to $m$ owing to assumption (H3)(d).

For $N=2$, following the preceding procedure, we conclude for any $p \geq 2$,
\begin{equation}
\label{2nd iteration 1}
\mathbb{E}\left[\exp \left\{p\lambda \sup _{s \in[T-h, T]}\left|Y_s^{(m)}\right|\right\}\right] \leq 4^\rho \mathbb{E}[\exp \{(8+8 \kappa) p\lambda|\xi|\}]^{\frac{\rho}{4}} \mathbb{E}\left[\exp \{(8+8 \kappa) p\lambda \widetilde{\alpha}\}^{\frac{\rho}{4}}.\right.
\end{equation}
Next, we consider the exponential growth mean-reflected BSDE within the time interval $[0, T-h]$ :
$$
\left\{\begin{array}{l}
Y_t^{(m)}=Y_{T-h}^{(m)}+\int_t^{T-h} f\left(s, Y_s^{(m-1)}, U_s^{(m)}\right) d A_s-\int_t^{T-h}\int_E U_s^{(m)}(e)q(dsde)+K_{T-h}^{(m)}-K_t^{(m)}, \\
\mathbb{E}\left[\ell\left(t, Y_t^{(m)}\right)\right] \geq 0, \quad \forall t \in[0, T-h] \text { and } \int_0^{T-h} \mathbb{E}\left[\ell\left(t, Y_{t^-}^{(m)}\right)\right] d K_t^{(m)}=0.
\end{array}\right.
$$
Considering the derivation of (\ref{1st iteration 1}), we infer that
\begin{equation}
\label{2nd iteration 2}
\begin{aligned}
& \mathbb{E}\left[\exp \left\{p\lambda \sup _{s \in[0, T-h]}\left|Y_s^{(m)}\right|\right\}\right] \leq 4^\rho \mathbb{E}\left[\exp \left\{(8+8 \kappa) p\lambda\left|Y_{T-h}^{(m)}\right|\right\}\right]^{\frac{\rho}{4}} \mathbb{E}[\exp \{(8+8 \kappa) p\lambda \widetilde{\alpha}\}]^{\frac{\rho}{4}} \\
& \leq 4^{\rho+\frac{\rho^2}{4}} \mathbb{E}\left[\exp \left\{(8+8 \kappa)^2 p\lambda|\xi|\right\}\right]^{\frac{\rho^2}{16}} \mathbb{E}\left[\exp \left\{(8+8 \kappa)^2 p\lambda \widetilde{\alpha}\right\}\right]^{\frac{\rho^2}{16}} \mathbb{E}[\exp \{(8+8 \kappa) p\lambda \widetilde{\alpha}\}]^{\frac{\rho}{4}},
\end{aligned}
\end{equation}
where we use (\ref{2nd iteration 1}) in the last inequality. By combining (\ref{2nd iteration 1}) with (\ref{2nd iteration 2}) and subsequently applying H\"older’s inequality again, we arrive at the following estimate for any $p \geq 2$
$$
\begin{aligned}
\mathbb{E}\left[\exp \left\{p\lambda \sup _{s \in[0, T]}\left|Y_s^{(m)}\right|\right\}\right] & \leq \mathbb{E}\left[\exp \left\{2 p\lambda \sup _{s \in[0, T-h]}\left|Y_s^{(m)}\right|\right\}\right]^{\frac{1}{2}} \mathbb{E}\left[\exp \left\{2 p\lambda \sup _{s \in[T-h, T]}\left|Y_s^{(m)}\right|\right\}\right]^{\frac{1}{2}} \\
& \leq 4^{\rho+\frac{\rho^2}{8}} \mathbb{E}\left[\exp \left\{(8+8 \kappa)^2 2 p\lambda|\xi|\right\}\right]^{\frac{\rho}{8}+\frac{\rho^2}{32}} \mathbb{E}\left[\exp \left\{(8+8 \kappa)^2 2 p\lambda \widetilde{\alpha}\right\}\right]^{\frac{\rho}{4}+\frac{\rho^2}{32}},
\end{aligned}
$$
which is also uniformly bounded with respect to $m$.

For arbitrary $N$, repeating the abovementioned procedure $N$ times yields
\begin{equation}
\label{Y_space}
\sup _{m \geq 0} \mathbb{E}\left[\exp \left\{p\lambda \sup _{s \in[0, T]}\left|Y_s^{(m)}\right|\right\}\right]<\infty, \ \forall p \geq 1,
\end{equation}
which together with (\ref{representation_Ym}) implies that
$$
\sup _{m \geq 0} \mathbb{E}\left[\exp \left\{p\lambda \sup _{s \in[0, T]}\left|y_s^{(m)}\right|\right\}\right]<\infty,
\ \forall p \geq 1.
$$
\end{proof}

\begin{proof}[Proof of Lemma \ref{lemma_limit estimation}]
Without loss of generality, assume $f(t, y,\cdot)$ is convex. For each fixed $m, q \geq 1$ and $\theta \in(0,1)$, define similarly $\delta_\theta \ell^{(m, q)}, \delta_\theta \widetilde{\ell}^{(m, q)}$ and $\delta_\theta \bar{\ell}^{(m, q)}$ for $y, u$. Then, the pair of processes $\left(\delta_\theta y^{(m, q)}, \delta_\theta u^{(m, q)}\right)$ solves the following BSDE:
\begin{equation}
\label{delta-BSDE}
\delta_\theta y_t^{(m, q)}=\xi+\int_t^T\left(\delta_\theta f^{(m, q)}\left(s, \delta_\theta u_s^{(m, q)}\right)+\delta_\theta f_0^{(m, q)}(s)\right) d A_s-\int_t^T \int_E\delta_\theta u_s^{(m, q)}(e)q(dsde),
\end{equation}
where the generator is given by
$$
\begin{aligned}
& \delta_\theta f_0^{(m, q)}(t)=\frac{1}{1-\theta}\left(f\left(t, Y_t^{(m+q-1)}, u_t^{(m+q)}\right)-f\left(t, Y_t^{(m-1)}, u_t^{(m+q)}\right)\right), \\
& \delta_\theta f^{(m, q)}(t, u)=\frac{1}{1-\theta}\left(-\theta f\left(t, Y_t^{(m-1)}, u_t^{(m)}\right)+f\left(t, Y_t^{(m-1)},(1-\theta) u+\theta u_t^{(m)}\right)\right).
\end{aligned}
$$
From assumptions (H3)(c) and (H3)(e), we deduce
$$
\begin{aligned}
& \delta_\theta f_0^{(m, q)}(t) \leq \beta\left(\left|Y_t^{(m-1)}\right|+\left|\delta_\theta Y_t^{(m-1, q)}\right|\right), \\
& \delta_\theta f^{(m, q)}(t, u) \leq f\left(t, Y_t^{(m-1)}, u\right) \leq \alpha_t+\beta\left(\left|Y_t^{(m-1)}\right|\right)+\frac{1}{\lambda}j_\lambda(t, u).
\end{aligned}
$$
For any $m,q\geq 1$, denote
$$
\begin{aligned}
& \zeta^{(m,q)}=|\xi|+\int_0^T \alpha_s d A_s+\beta A_T\left(\sup _{s \in[0, T]}\left|Y_s^{(m-1)}\right|+\sup _{s \in[0, T]}\left|Y_s^{(m+q-1)}\right|\right), \\
& \chi^{(m, q)}=\int_0^T \alpha_s d A_s+2 \beta A_T\left(\sup _{s \in[0, T]}\left|Y_s^{(m+q-1)}\right|+\sup _{s \in[0, T]}\left|Y_s^{(m-1)}\right|\right).
\end{aligned}
$$
Applying assertion (ii) of Lemma \ref{a priori esti 2} to (\ref{delta-BSDE}) yields for any $p \geq 1$,
$$
\exp \left\{p\lambda\left(\delta_\theta y_t^{(m, q)}\right)^{+}\right\} \leq \mathbb{E}_t\left[ \exp \left\{
p\lambda\left(|\xi|+\chi^{(m, q)}+\beta(A_T-A_t)\sup _{s \in[t, T]}\left|\delta_\theta Y_s^{(m-1, q)}\right|\right)\right\}\right],
$$
and in the same manner, we also have
$$
\exp \left\{p\lambda\left(\delta_\theta \widetilde{y}_t^{(m, q)}\right)^{+}\right\}
\leq \mathbb{E}_t\left[\exp \left\{p\lambda\left(|\xi|+\chi^{(m, q)}+\beta(A_T-A_t)\sup _{s \in[t, T]}\left|\delta_\theta \widetilde{Y}_s^{(m-1, q)}\right|\right)\right\}\right].
$$
By the fact that
$$
\left(\delta_\theta y^{(m, q)}\right)^{-} \leq\left(\delta_\theta \widetilde{y}^{(m, q)}\right)^{+}+2\left|y^{(m+q)}\right| \text { and }\left(\delta_\theta \widetilde{y}^{(m, q)}\right)^{-} \leq\left(\delta_\theta y^{(m, q)}\right)^{+}+2\left|y^{(m)}\right|,
$$
we derive, applying H\"older’s inequality, that
$$
\begin{aligned}
& \exp \left\{p\lambda\left|\delta_\theta y_t^{(m, q)}\right|\right\} \vee \exp \left\{p\lambda\left|\delta_\theta \widetilde{y}_t^{(m, q)}\right|\right\} \\
& \leq \exp \left\{p\lambda\left(\left(\delta_\theta y_t^{(m, q)}\right)^{+}+\left(\delta_\theta \widetilde{y}_t^{(m, q)}\right)^{+}+2\left|y_t^{(m)}\right|+2\left|y_t^{(m+q)}\right|\right)\right\} \\
& \leq \mathbb{E}_t\left[\exp \left\{p\lambda\left(|\xi|+\chi^{(m, q)}+\beta(A_T-A_t)\sup _{s \in[t, T]} \delta_\theta \bar{Y}_s^{(m-1, q)}\right)\right\}\right]^2 \\
& \quad \times \exp \left\{2 p\lambda\left(\left|y_t^{(m)}\right|+\left|y_t^{(m+q)}\right|\right)\right\} \\
& \leq \mathbb{E}_t\left[\exp \left\{p\lambda\left(|\xi|+\chi^{(m, q)}+\beta(A_T-A_t)\left(\sup _{s \in[t, T]} \delta_\theta \bar{Y}_s^{(m-1, q)}\right)\right)\right\}\right]^2 \\
& \quad \times \mathbb{E}_t\left[\exp \left\{4 p\lambda \zeta^{(m, q)}\right\}\right].
\end{aligned}
$$
Considering Doob’s maximal inequality and H\"older’s inequality, we conclude that for all $p>1$ and $t \in[0, T]$,
$$
\begin{aligned}
& \mathbb{E}\left[\exp \left\{p\lambda \sup _{s \in[t, T]} \delta_\theta \bar{y}_s^{(m, q)}\right\}\right] \\
& \leq 4 \mathbb{E}\left[\exp \left\{8 p\lambda\left(|\xi|+\chi^{(m, q)}+\beta(A_T-A_t)\left(\sup _{s \in[t, T]} \delta_\theta \bar{Y}_s^{(m-1, q)}\right)\right)\right\}\right]^{\frac{1}{2}} \\
& \quad \times \mathbb{E}\left[\exp \left\{16 p\lambda \zeta^{(m, q)}\right\}\right]^{\frac{1}{2}}.
\end{aligned}
$$
Let $C_2:=\sup _{0 \leq s \leq T}\left|L_s(0)\right|+2 \kappa \sup _m \mathbb{E}\left[\sup _{s \in[0, T]}\left|y_s^{(m)}\right|\right]<\infty$. According to assumption (H6) and (\ref{representation_Ym}),
$$
\delta_\theta \bar{Y}_t^{(m, q)} \leq \delta_\theta \bar{y}_t^{(m, q)}+2 \kappa \sup _{t \leq s \leq T} \mathbb{E}\left[\delta_\theta \bar{y}_t^{(m, q)}\right]+2 C_2,
$$
which with Jensen’s inequality implies that for each $p \geq 1$ and $t \in[0, T]$,
$$
\begin{aligned}
&\mathbb{E}\left[\exp \left\{p\lambda \sup _{s \in[t, T]} \delta_\theta \bar{Y}_s^{(m, q)}\right\}\right] \leq e^{2 p\lambda C_2} \mathbb{E}\left[\exp \left\{(2+4 \kappa) p\lambda \sup _{s \in[t, T]} \delta_\theta \bar{y}_s^{(m, q)}\right\}\right] \\
&\leq 4 \mathbb{E}\left[\exp \left\{(16+32 \kappa) p\lambda\left(|\xi|+\chi^{(m, q)}+C_2+\beta(A_T-A_t) \sup _{s \in[t, T]} \delta_\theta \bar{Y}_s^{(m-1, q)}\right)\right\}\right]^{\frac{1}{2}}
\mathbb{E}\left[\exp \left\{(32+64 \kappa) p\lambda \zeta^{(m, q)}\right\}\right]^{\frac{1}{2}}.
\end{aligned}
$$
Choosing $h$ as described in (\ref{h_condition}), we obtain from H\"older’s inequality and Jensen’s inequality that
\begin{equation}
\label{iter-1}
\begin{aligned}
& \mathbb{E}\left[\exp \left\{p\lambda \sup _{s \in[T-h, T]} \delta_\theta \bar{Y}_s^{(m, q)}\right\}\right] \\
& \leq 4 \mathbb{E}[\exp \{(64+128 \kappa) p\lambda|\xi|\}]^{\frac{1}{8}} \mathbb{E}\left[\exp \left\{(64+128 \kappa) p\lambda\left(\chi^{(m, q)}+C_2\right)\right\}\right]^{\frac{1}{8}} \\
& \quad \times \mathbb{E}\left[\exp \left\{(32+64 \kappa) p \lambda\zeta^{(m, q)}\right\}\right]^{\frac{1}{2}} \mathbb{E}\left[\exp \left\{p\lambda \sup _{s \in[T-h, T]} \delta_\theta \bar{Y}_s^{(m-1, q)}\right\}\right]^{(8+16 \kappa) \beta \|A_T-A_{T-h}\|_\infty}.
\end{aligned}
\end{equation}
Set $\widetilde{\rho}=\frac{1}{1-(8+16 \kappa) \beta \max_{1\le i\le N}\|A_{ih}-A_{(i-1)h}\|_\infty}$.

If $N=1$, it follows from (\ref{iter-1}) that for each $p \geq 1$ and $m, q \geq 1$,
\begin{equation}
\label{iter1-1}
\begin{aligned}
& \mathbb{E}\left[\exp \left\{p \lambda\sup _{s \in[0, T]} \delta_\theta \bar{Y}_s^{(m, q)}\right\}\right] \\
& \leq 4^{\tilde{\rho}} \mathbb{E} {[\exp \{(64+128 \kappa) p\lambda|\xi|\}]^{\frac{\bar{\rho}}{8}} \sup _{m, q \geq 1} \mathbb{E}\left[\exp \left\{(64+128 \kappa) p\left(\chi^{(m, q)}+C_2\right)\right\}\right]^{\frac{\bar{\rho}}{8}} } \\
& \quad \times \sup _{m, q \geq 1} \mathbb{E}\left[\exp \left\{(32+64 \kappa) p \lambda\zeta^{(m, q)}\right\}\right]^{\frac{\bar{\rho}}{2}} \mathbb{E}\left[\exp \left\{p\lambda \sup _{s \in[0, T]} \delta_\theta \bar{Y}_s^{(1, q)}\right\}\right]^{(8 \beta \|A_T-A_{T-h}\|_\infty+16 \kappa \beta \|A_T-A_{T-h}\|_\infty)^{m-1}}.
\end{aligned}
\end{equation}
 Lemma \ref{lemma_uniform estimate} insures that for any $\theta \in(0,1)$
$$
\lim _{m \rightarrow \infty} \sup _{q \geq 1} \mathbb{E}\left[\exp \left\{p\lambda \sup _{s \in[0, T]} \delta_\theta \bar{Y}_s^{(1, q)}\right\}\right]^{(8 \beta \|A_T-A_{T-h}\|_\infty+16 \kappa \beta \|A_T-A_{T-h}\|_\infty)^{m-1}}=1,
$$
which implies that
$$
\begin{aligned}
\sup _{\theta \in(0,1)} & \lim _{m \rightarrow \infty} \sup _{q \geq 1} \mathbb{E}\left[\exp \left\{p\lambda \sup _{s \in[0, T]} \delta_\theta \bar{Y}_s^{(m, q)}\right\}\right] \\
& \leq 4^{\tilde{\rho}} \mathbb{E}[\exp \{(64+128 \kappa) p\lambda|\xi|\}]^{\frac{\bar{\rho}}{8}} \sup _{m, q \geq 1} \mathbb{E}\left[\exp \left\{(64+128 \kappa) p\lambda\left(\chi^{(m, q)}+C_2\right)\right\}\right]^{\frac{\bar{\rho}}{8}} \\
& \quad \times \sup _{m, q \geq 1} \mathbb{E}\left[\exp \left\{(32+64 \kappa) p\lambda \zeta^{(m, q)}\right\}\right]^{\frac{\bar{\rho}}{2}}<\infty.
\end{aligned}
$$
If $N=2$, inherited from the derivation of (\ref{iter1-1}), for any $p \geq 1$,
$$
\begin{aligned}
& \mathbb{E}\left[\exp \left\{p\lambda \sup _{s \in[0, T]} \delta_\theta \bar{Y}_s^{(m, q)}\right\}\right] \\
& \leq 4^{\tilde{\rho}+\frac{\bar{\rho}^2}{16}} \mathbb{E}\left[\exp \left\{(64+128 \kappa)^2 2 p\lambda|\xi|\right\}\right]^{\frac{\bar{\rho}}{16}+\frac{\bar{\rho}^2}{128}} \sup _{m, q \geq 1} \mathbb{E}\left[\exp \left\{(64+128 \kappa)^2 2 p\lambda\left(\chi^{(m, q)}+C_2\right)\right\}\right]^{\frac{\bar{\rho}}{8}+\frac{\bar{\rho}^2}{128}} \\
& \quad \times \sup _{m, q \geq 1} \mathbb{E}\left[\exp \left\{(32+64 \kappa)(64+128 \kappa) 2 p\lambda \zeta^{(m, q)}\right\}\right]^{\frac{\bar{\rho}}{2}+\frac{\bar{\rho}^2}{32}} \\
& \quad \times \mathbb{E}\left[\exp \left\{(64+128 \kappa) 2 p\lambda \sup _{s \in[0, T]} \delta_\theta \bar{Y}_s^{(1, q)}\right\}\right]^{\left(\frac{1}{2}+\frac{\tilde{\rho}}{16}\right)(8 \beta \max_{i=1,2}\|A_{ih}-A_{(i-1)h}\|_\infty+16\kappa \beta \max_{i=1,2}\|A_{ih}-A_{(i-1)h}\|_\infty)^{m-1}},
\end{aligned}
$$
which also implies the desired outcome when $N=2$. For general $N$, repeating the abovementioned procedure $N$ times concludes the proof.
\end{proof}

\end{appendices}

\noindent\textbf{Acknowledgements}

The authors thank the anonymous referees and editors for their valuable advice
and insightful comments, which greatly helped improve the previous version of this paper. This work was partially supported by
NSFC under project No. 12371473 and by the Tianyuan Fund for Mathematics of NSFC under project No. 12326603.

\bibliographystyle{plain}
\bibliography{RBSDEJ}

\begin{thebibliography}{10}

\bibitem{antonelli2016solutions}
Fabio Antonelli and Carlo Mancini.
\newblock Solutions of {BSDE}’s with jumps and quadratic/locally lipschitz
  generator.
\newblock {\em Stochastic Processes and their Applications},
  126(10):3124--3144, 2016.

\bibitem{barles1997backward}
Guy Barles, Rainer Buckdahn, and \'Etienne Pardoux.
\newblock Backward stochastic differential equations and integral-partial
  differential equations.
\newblock {\em Stochastics: An International Journal of Probability and
  Stochastic Processes}, 60(1-2):57--83, 1997.

\bibitem{Barrieu_2013}
Pauline Barrieu and Nicole~El Karoui.
\newblock Monotone stability of quadratic semimartingales with applications to
  unbounded general quadratic {BSDEs}.
\newblock {\em The Annals of Probability}, 41(3B), 2013.

\bibitem{Becherer_2006}
Dirk Becherer.
\newblock Bounded solutions to backward {SDEs} with jumps for utility
  optimization and indifference hedging.
\newblock {\em The Annals of Applied Probability}, 16(4):2027--2054, 2006.

\bibitem{Bremaud1981}
Pierre Br{\'{e}}maud.
\newblock {\em Point Processes and Queues}.
\newblock Springer New York, 1981.

\bibitem{briand2018bsdes}
Philippe Briand, Romuald Elie, and Ying Hu.
\newblock {BSDE}s with mean reflection.
\newblock {\em The Annals of Applied Probability}, 28(1):482--510, 2018.

\bibitem{Briand_2020}
Philippe Briand, Abir Ghannoum, and C{\'{e}}line Labart.
\newblock Mean reflected stochastic differential equations with jumps.
\newblock {\em Advances in Applied Probability}, 52(2):523--562, 2020.

\bibitem{briand2006bsde}
Philippe Briand and Ying Hu.
\newblock {BSDE with quadratic growth and unbounded terminal value}.
\newblock {\em Probability Theory and Related Fields}, 136(4):604--618, 2006.

\bibitem{briand2008quadratic}
Philippe Briand and Ying Hu.
\newblock {Quadratic {BSDE}s with convex generators and unbounded terminal
  conditions}.
\newblock {\em Probability Theory and Related Fields}, 141(3):543--567, 2008.

\bibitem{cohen2012existence}
Samuel~N. Cohen and Robert~J. Elliott.
\newblock Existence, uniqueness and comparisons for {BSDE}s in general spaces.
\newblock {\em The Annals of Probability}, 40(5):2264--2297, 2012.

\bibitem{cohen2015stochastic}
Samuel~N. Cohen and Robert~J. Elliott.
\newblock {\em Stochastic calculus and applications}, volume~2.
\newblock Springer, 2015.

\bibitem{Confortola_2018}
Fulvia Confortola.
\newblock {$L^p$} solution of backward stochastic differential equations driven
  by a marked point process.
\newblock {\em Mathematics of Control, Signals, and Systems}, 31(1), 2018.

\bibitem{Confortola2013}
Fulvia Confortola and Marco Fuhrman.
\newblock Backward stochastic differential equations and optimal control of
  marked point processes.
\newblock {\em SIAM Journal on Control and Optimization}, 51(5):3592--3623,
  2013.

\bibitem{Confortola_2014}
Fulvia Confortola and Marco Fuhrman.
\newblock Backward stochastic differential equations associated to jump markov
  processes and applications.
\newblock {\em Stochastic Processes and their Applications}, 124(1):289--316,
  2014.

\bibitem{Confortola2016}
Fulvia Confortola, Marco Fuhrman, and Jean Jacod.
\newblock Backward stochastic differential equation driven by a marked point
  process: An elementary approach with an application to optimal control.
\newblock {\em The Annals of Applied Probability}, 26(3), 2016.

\bibitem{karoui2016quadratic}
Nicole El~Karoui, Anis Matoussi, and Armand Ngoupeyou.
\newblock Quadratic exponential semimartingales and application to {BSDE}s with
  jumps.
\newblock {\em arXiv preprint arXiv:1603.06191}, 2016.

\bibitem{foresta2021optimal}
Nahuel Foresta.
\newblock Optimal stopping of marked point processes and reflected backward
  stochastic differential equations.
\newblock {\em Applied Mathematics \& Optimization}, 83(3):1219--1245, 2021.

\bibitem{2310.15203}
Zihao Gu, Yiqing Lin, and Kun Xu.
\newblock Mean reflected {BSDE} driven by a marked point process and
  application in insurance risk management.
\newblock {\em arXiv preprint arXiv:2310.15203}, 2023.

\bibitem{he2019semimartingale}
Sheng-wu He, Jia-gang Wang, and Jia-an Yan.
\newblock {\em Semimartingale theory and stochastic calculus}.
\newblock Routledge, 2019.

\bibitem{Hibon2017quadratic}
H\'el\`{e}ne Hibon, Ying Hu, Yiqing Lin, Peng Luo, and Falei Wang.
\newblock Quadratic {BSDE}s with mean reflection.
\newblock {\em Mathematical Control and Related Fields}, 8:721--738, 2017.

\bibitem{hu2022general}
Ying Hu, Remi Moreau, and Falei Wang.
\newblock General mean reflected {BSDE}s.
\newblock {\em arXiv preprint arXiv:2211.01187}, 2022.

\bibitem{jeanblanc2012robust}
Monique Jeanblanc, Anis Matoussi, and Armand Ngoupeyou.
\newblock Robust utility maximization problem in model with jumps and unbounded
  claim.
\newblock {\em arXiv preprint arXiv:1201.2690}, 2012.

\bibitem{kaakai2022utility}
Sarah Kaaka\"i, Anis Matoussi, and Achraf Tamtalini.
\newblock Utility maximization problem with uncertainty and a jump setting.
\newblock {\em arXiv preprint arXiv:2210.07640}, 2022.

\bibitem{kazi2015quadratic}
Nabil Kazi-Tani, Dylan Possama{\"\i}, and Chao Zhou.
\newblock Quadratic {BSDE}s with jumps: a fixed-point approach.
\newblock {\em Electronic Journal of Probability}, 20(66):1--28, 2015.

\bibitem{kobylanski2000backward}
Magdalena Kobylanski.
\newblock Backward stochastic differential equations and partial differential
  equations with quadratic growth.
\newblock {\em The Annals of Probability}, 28(2):558--602, 2000.

\bibitem{last1995marked}
G{\"u}nter Last and Andreas Brandt.
\newblock {\em Marked Point Processes on the real line: the dynamical
  approach}.
\newblock Springer Science \& Business Media, 1995.

\bibitem{Lepeltier_1997}
J-P. Lepeltier and J.~San Martin.
\newblock Backward stochastic differential equations with continuous
  coefficient.
\newblock {\em Statistics {\&} Probability Letters}, 32(4):425--430, 1997.

\bibitem{Liu_2009}
Yuping Liu and Jin Ma.
\newblock Optimal reinsurance/investment problems for general insurance models.
\newblock {\em The Annals of Applied Probability}, 19(4), 2009.

\bibitem{Morlais_2009}
Marie-Am{\'e}lie Morlais.
\newblock Utility maximization in a jump market model.
\newblock {\em Stochastics}, 81(1):1--27, 2009.

\bibitem{morlais2010new}
Marie-Am{\'e}lie Morlais.
\newblock A new existence result for quadratic {BSDE}s with jumps with
  application to the utility maximization problem.
\newblock {\em Stochastic processes and their applications},
  120(10):1966--1995, 2010.

\bibitem{ngoupeyou2010optimisation}
Armand~Brice Ngoupeyou.
\newblock {\em Optimisation des portefeuilles d'actifs soumis au risque de
  d{\'e}faut}.
\newblock PhD thesis, Evry-Val d'Essonne, 2010.

\bibitem{pardoux1990adapted}
\'Etienne Pardoux and Shige Peng.
\newblock Adapted solution of a backward stochastic differential equation.
\newblock {\em Systems \& control letters}, 14(1):55--61, 1990.

\bibitem{SCHEUTZOW_2013}
Michael Scheutzow.
\newblock A stochastic gronwall lemma.
\newblock {\em Infinite Dimensional Analysis, Quantum Probability and Related
  Topics}, 16(02):1350019, 2013.

\bibitem{Tang_1994}
Shanjian Tang and Xunjing Li.
\newblock Necessary conditions for optimal control of stochastic systems with
  random jumps.
\newblock {\em {SIAM} Journal on Control and Optimization}, 32(5):1447--1475,
  1994.

\bibitem{tevzadze2008solvability}
Revaz Tevzadze.
\newblock Solvability of backward stochastic differential equations with
  quadratic growth.
\newblock {\em Stochastic processes and their Applications}, 118(3):503--515,
  2008.

\end{thebibliography}
\end{document}